\documentclass[10pt,a4paper]{amsart}
\usepackage[utf8]{inputenc}
\usepackage{amsmath}
\usepackage{amsfonts}
\usepackage{amssymb}
\usepackage{amsthm}
\usepackage{enumitem}
\usepackage{mathtools}
\usepackage{dsfont}
\usepackage{amsaddr}

\title{Global trace formula for ultra-differentiable Anosov flows}
\author{Malo Jézéquel}
\date{\today}
\address{Laboratoire de Probabilités, Statistique et Modélisation (LPSM) \\ CNRS, Sorbonne Université, Université de Paris \\ 4, Place Jussieu, 75005 Paris, France.}
\email{jezequel@lpsm.paris}
\thanks{This project has received funding from the European Research Council (ERC) under the European Union's Horizon 2020 research and innovation programme (grant agreement No 787304). This work was started while the author was affiliated with Institut de Mathématiques de Jussieu-Paris Rive Gauche.}

\numberwithin{equation}{section}

\newtheorem{lm}{Lemma}[section]
\newtheorem{thm}[lm]{Theorem}
\newtheorem{cor}[lm]{Corollary}
\newtheorem{prop}[lm]{Proposition}
\theoremstyle{definition}
\newtheorem{df}[lm]{Definition}
\newtheorem{rmq}[lm]{Remark}

\newcommand{\p}[1]{\left(#1\right)}
\newcommand{\va}[1]{\left| #1 \right|}
\newcommand{\n}[1]{\left\| #1 \right\|}
\newcommand{\N}{\mathbb{N}}
\newcommand{\Z}{\mathbb{Z}}
\newcommand{\C}{\mathbb{C}}
\newcommand{\R}{\mathbb{R}}
\newcommand{\h}{\mathcal{H}}
\newcommand{\B}{\mathcal{B}}
\newcommand{\s}[3]{\int_{#1} #2 \mathrm{d} #3}
\newcommand{\set}[1]{\left\{ #1 \right\}}

\newcommand{\bul}[1]{\stackrel{\circ}{#1}}
\newcommand{\hra}{\hookrightarrow}
\newcommand{\nhra}{\not\hookrightarrow}

\begin{document}

\begin{abstract}
Adapting tools that we introduced in \cite{lagtf} to study Anosov flows, we prove that the trace formula conjectured by Dyatlov and Zworski in \cite{DZdet} holds for Anosov flows in a certain class of regularity (smaller than $\mathcal{C}^\infty$ but larger than the class of Gevrey functions). The main ingredient of the proof is the construction of a family of anisotropic Hilbert spaces of generalized distributions on which the generator of the flow has discrete spectrum.
\end{abstract}

\maketitle

\section*{Introduction}

Let $V$ be a $\mathcal{C}^\infty$ vector field on a smooth manifold $M$ of dimension $d+1 \geq 3$, and assume that $V$ generates an Anosov flow $\p{\phi^t}_{t \in \R}$ (see Definition \ref{defanosov}). The vector field $V$ may be identified with a differential operator of order $1$ whose spectral properties are of great interest when studying the statistical properties of the flow $\p{\phi^t}_{t \in \R}$. However, the operator $V$ is not elliptic and consequently its spectrum on $L^2\p{M}$ can be quite wild. In \cite{buli,bulicor}, Butterley and Liverani showed that, introducing an appropriate scale of \emph{anisotropic Banach spaces of distributions} on $M$, one may define a suitable notion of spectrum for $V$, the \emph{Ruelle spectrum}\footnote{The top-right part of the spectrum had already been unveiled by Liverani for contact Anosov flows in \cite{Livcontact}.}, whose elements are called Ruelle resonances (see Theorem~\ref{thmfonda} and Definition \ref{defresonances}). After \cite{buli,bulicor}, spaces of anisotropic of distributions have been widely used to study in particular the Ruelle resonances (see for instance \cite{FauSjo} that gives a construction of anisotropic Sobolev spaces using the language of micro-local analysis, or \cite{phdadam,adamarxiv} for another construction). 

One of the most striking applications of the spaces of anisotropic distributions has been the proof of Smale's conjecture on the meromorphic continuation of zeta functions associated to Axiom A flows (see \cite{GLP,DZdet,opensystems,afterword}). The theory of Ruelle zeta functions and dynamical determinant makes a link between Ruelle resonances for the operator $V$ and periodic orbit of the flow $\p{\phi^t}_{t \in \R}$ (see Theorem~\ref{thmdet}). In \cite{DZdet}, Dyatlov and Zworski suggested that there could be another link between these objects, a global trace formula in the sense of the following equality between distributions on $\R_+^*$:
\begin{equation}\label{eqtrformsimple}
\sum_{\lambda \textup{ resonances}} e^{\lambda t} = \sum_{\gamma} \frac{T_\gamma^{\#}}{\va{\det\p{I - \mathcal{P}_\gamma}}} \delta_{T_\gamma},
\end{equation}
where the sum on the right-hand side runs over periodic orbits $\gamma$ of the flow $\p{\phi^t}_{t \in \R}$. If $\gamma$ is a periodic orbit of the flow $\p{\phi^t}_{t \in \R}$, then $T_\gamma$ denotes its length, $T_\gamma^{\#}$ its primitive length and $\mathcal{P}_\gamma$ the associated linearized Poincaré map (which is defined below \eqref{eqdefdet}). The intuition behind \eqref{eqtrformsimple} is based on Guillemin's trace formula (see \cite{guillemin} and \cite[2.2]{DZdet}). 

It may be deduced from work of Fried and Rugh \cite{R1,R2,friedzeta} that the trace formula \eqref{eqtrformsimple} holds for real-analytic Anosov flows. In \cite{lagtf}, we studied a discrete-time analogue of this problem, and our results indicate that formula \eqref{eqtrformsimple} could be wrong for some Anosov $\mathcal{C}^\infty$ flows. However, we suggested in \cite{lagtf} that the trace formula should hold for Gevrey flows (see \cite{lagtf} or \S \ref{secDCC} for a definition). Indeed, we proved in \cite{lagtf} that the discrete-time analogue of the trace formula is true for Gevrey uniformly hyperbolic diffeomorphisms, and the methods that we developed there seemed robust enough to be adapted to the time-continuous case.

Adapting ideas we developed in \cite{lagtf} to the context of Anosov flows, we prove here the global trace formula \eqref{eqtrformsimple} for a class of regularity much larger than Gevrey (see Corollary~\ref{cortrform}) that we define in \S \ref{secDCC}. We prove in fact a slightly more general version \eqref{eqtrform} of the trace formula: we study resonances for the operator $X = V + g$ where $g : M \to \C$ is an ultradifferentiable potential. The main tool of the proof is the construction of a family of adapted Hilbert spaces of anisotropic generalized distributions, see Theorem~\ref{thmmain}.

The paper is structured as follow:

In \S \ref{secsett}, we recall basic facts from the theory of Ruelle resonances and state our main results, Theorem~\ref{thmmain}, Corollary~\ref{cortrform} and Proposition~\ref{propmain}. Theorem~\ref{thmmain} ensures that the Koopman operator \eqref{eqkoopman} has good properties when acting on some Hilbert spaces of anisotropic generalized distributions. The trace formula readily follows as stated in Corollary~\ref{cortrform}. Proposition~\ref{propmain} gives a control on the number of Ruelle resonances that naturally follows from the proof of Theorem~\ref{thmmain}.

In \S \ref{secDCC}, we use the language of Denjoy--Carleman classes to define the regularity that appears in Theorem~\ref{thmmain}, Corollary~\ref{cortrform} and Proposition~\ref{propmain}. We also define spaces of generalized distributions needed for the construction of the Hilbert spaces appearing in Theorem~\ref{thmmain}.

In \S \ref{sectls}, we define a local version of the space $\h$ from Theorem~\ref{thmmain}, and we study in \S \ref{sectlto} the action on this local space of a local model for an Anosov flow (in charts). This is achieved by adapting the techniques that we introduced in \cite{lagtf} for Gevrey uniformly hyperbolic diffeomorphisms both for continuous-time dynamics and for the larger class of regularity that we introduce in \S \ref{secDCC}.

The construction of $\h$ and the proofs of Theorem~\ref{thmmain} and Proposition~\ref{propmain} are carried out in \S \ref{gs1} and \S \ref{gs2}. In \S \ref{gs1}, we design a first space that is well-suited to study the flow $\p{\phi^t}_{t \in \R}$ for large $t$. In \S \ref{gs2}, we use our this space to construct a space well-suited to study the flow $\p{\phi^t}_{t \in \R}$ for all non-negative times and prove the trace formula. The lack of hyperbolicity for $t > 0$ small will be dealt with by considering a decomposition of the powers of the resolvent $\p{z-X}^{-1}$ of the generator $X= V + g$ of the Koopman operator \eqref{eqkoopman} into a compact part (corresponding to large times for which we have uniform hyperbolicity) and a negligible part (corresponding to small times). This strategy is also what allows us to tackle a class of regularity that is larger than Gevrey and get a better result than the one we suggested in \cite{lagtf}.

In Appendix \ref{apprri}, we give a new proof of the fact that Ruelle resonances are intrinsic, which does not require to deal with Schwartz kernel (as it was the case in \cite{FauSjo} for instance). This implies in particular that the Ruelle resonances that appear from the study of the operator $X$ acting on the quite exotic Hilbert space $\h$ given by Theorem~\ref{thmmain} coincide with the Ruelle resonances that have already been defined in the literature.

In Appendix \ref{appDCC}, we give the proofs of Lemmas \ref{lmderiv} and~\ref{lmdomain} from \S \ref{secDCC}.

In Appendix \ref{appdet}, we give, under the hypotheses of Theorem~\ref{thmmain}, a ``Hadamard-like'' factorization \eqref{eqfacto} for the dynamical determinant $d_g$ defined by \eqref{eqdefdet}.

In Appendix \ref{apptrdet}, we prove Proposition~\ref{propapp} and Corollary~\ref{corphrag} as applications of the trace formula.

In Appendix \ref{appupsilon}, we discuss one of the hypothesis of the Theorem~\ref{thmmain}.

\tableofcontents

\section{Settings and statement of results}\label{secsett}

Let $d \geq 2$ be an integer and $M$ a $\p{d+1}$-dimensional $\mathcal{C}^\infty$ manifold. Let $V$ be a $\mathcal{C}^\infty$ vector field on $M$ which generates a $\mathcal{C}^\infty$ flow $\p{\phi^t}_{t \in \R}$. Let $g : M \to \C$ be a $\mathcal{C}^\infty$ function (called the potential).

\begin{df}[Anosov flow]\label{defanosov}
We say that the flow $\p{\phi^t}_{t \in \R}$ is \emph{Anosov} if $V$ does not vanish and for all $x \in M$ there is a decomposition of the tangent space of $M$ at $x$
\begin{equation}\label{eqdecanos}
T_x M = E_x^u \oplus E_x^s \oplus \R V\p{x}
\end{equation}
such that:
\begin{enumerate}[label=(\roman*)]
\item for all $t \in \R, x \in M$ and $\sigma \in \set{u,s}$ we have $D_x \phi^t\p{E_x^\sigma} = E_{\phi^t\p{x}}^\sigma$;
\item there are a metric $\va{\cdot}_x$ on $M$ and constants $C >0$ and $\lambda <1$ such that for all $t \in \R_+$ and $x \in M$ we have
\begin{align*}
\forall v_s \in E_{x}^s : \va{D_x \phi^t\p{v_s}}_{\phi^t\p{x}} \leq C \lambda^t \va{v_s}_x\\
\forall v_u \in E_x^u : \va{D_x \phi^{-t}\p{v_u}}_{\phi^{-t}\p{x}} \leq C \lambda^{t} \va{v_u}_x.
\end{align*}
\end{enumerate}
\end{df}

In the following, we assume that $\p{\phi^t}_{t \in \R}$ is an Anosov flow. Fundamental examples of Anosov flows are geodesic flows on unit tangent bundles of compact Riemannian manifolds of negative sectional curvature \cite[Theorem 17.6.2]{katokhasselblatt} and suspension of Anosov diffeomorphisms.

The main object of our study is the Koopman operator which may be defined for $t \in \R_+$ and $u \in \mathcal{D}'\p{M}$ by
\begin{equation}\label{eqkoopman}
\mathcal{L}_t u = \exp\p{\int_0^t g \circ \phi^{\tau} \mathrm{d}\tau} u \circ \phi^t.
\end{equation}
Notice that $\p{\mathcal{L}_t}_{t \geq 0}$ is a semi-group of operator on $\mathcal{D}'\p{M}$ whose generator is $X = V + g$. The most interesting case is when $g$ is real-valued since the spectral theory for the operator $X$ is then closely related to the statistical properties of the equilibrium state of $\p{\phi^t}_{t \in \R}$ for the potential $ g  - \textup{div}_u \p{V}$ (where $\textup{div}_u\p{V}$ denotes the "unstable divergence" of $V$). In particular, when $g = 0$ we may study the SRB measure for $\p{\phi^t}_{t \in \R}$ and when $g = \textup{div}_u\p{V}$ the measure of maximal entropy ($\textup{div}_u V$ is not smooth in general, but techniques have been developed by Gou\"ezel and Liverani to bypass this difficulty, see \cite{goulivstat}, using ideas that were already present in the physics literature \cite{CviVat}). Notice that, considering applications to statistical properties of the flow, it could be more natural to study the transfer operator, that is the adjoint of the operator \eqref{eqkoopman}. However, since we will state our results for general potential $g$, and the flow $\p{\phi^{-t}}_{t \in \R}$ is also Anosov, the choice of the operator \eqref{eqkoopman} is of no harm. 

However, the spectral theory of $X$ on $L^2\p{M}$ is not satisfactory: we need to use so-called ``anisotropic Banach spaces of distributions" \cite{buli,FauSjo,phdadam,opensystems}. The main theorem to carry out this study is the following. It has been proven first by Butterley and Liverani in the case $g= \textup{div}\p{V}$ in \cite{buli}, with a needed gap filled in \cite{bulicor}. A proof in a very general setting may be proven in \cite{opensystems}.

\begin{thm}[\cite{buli,bulicor,opensystems}]\label{thmfonda}
For every $A > 0$ there is a Banach space $\B$ such that:
\begin{enumerate}[label=(\roman*)]
\item $\mathcal{C}^\infty\p{M} \subseteq \B \subseteq \mathcal{D}'\p{M}$, both inclusions being continuous, the first one having dense image;
\item for all $t \in \R_+$, the operator $\mathcal{L}_t$ defined by \eqref{eqkoopman} is bounded on $\B$;
\item $\p{\mathcal{L}_t}_{t \geq 0}$ forms a strongly continuous semi-group of operators acting on $\B$, whose generator is $X=V+g$;
\item the intersection of $\set{ z \in \C: \Re\p{z} > -A}$ with the spectrum of $X$ acting on $\B$ consists of isolated eigenvalues of finite multiplicity.
\end{enumerate}
\end{thm}

The space $\B$ in Theorem~\ref{thmfonda} is highly non canonical, but in fact the intersection of the spectrum of $X$ acting on $\B$ and $\set{ z \in \C: \Re\p{z} > -A}$ does not depend on the choice of $\B$ (see Lemma~\ref{lmdefres} in Appendix \ref{apprri} and Theorem~\ref{thmdet}). This allows us to define the \emph{Ruelle resonances} of $X$.

\begin{df}[Ruelle resonances]\label{defresonances}
For $\lambda \in \C$ and $m \in \N^*$, we say that $\lambda$ is a Ruelle resonance of $X$ of multiplicity $m$ if for some $A > 0$ such that $\Re\p{\lambda} > - A$ there is a Banach space $\B$ satysfying (i)-(iv) from Theorem~\ref{thmfonda} such that $\lambda$ is an eigenvalue of (algebraic) multiplicity $m$ for $X$ acting on $\B$.
\end{df} 

It is not easy to describe Ruelle resonances in general. A convenient tool to do it is the dynamical determinant, which is defined for $z \in \C$ with $\Re\p{z} \gg 1$ by
\begin{equation}\label{eqdefdet}
d_g\p{z} = \exp\p{- \sum_{\gamma} \frac{T_\gamma^{\#}}{T_\gamma} e^{\int_\gamma g} \frac{e^{-z T_\gamma}}{\va{\det\p{I - \mathcal{P}_\gamma}}} },
\end{equation}
where the sum runs over the (countable set of the) periodic orbits $\gamma$ of the flow $\p{\phi^t}_{t \in \R}$. If $\gamma$ is a periodic orbit of $\p{\phi^t}_{t \in \R}$ then:
\begin{itemize}
\item $T_\gamma$ denotes its length;
\item $T_\gamma^{\#}$ denotes its primitive length, i.e. the length of the shortest periodic orbit $\gamma^{\#}$ with the same image as $\gamma$;
\item $\int_\gamma g$ is the integral of $g$ along $\gamma$, i.e. $\int_\gamma g = \int_0^{T_\gamma} g \circ \phi^{\tau}\p{x} \mathrm{d}\tau$ for any $x$ in the image of $\gamma$;
\item $\mathcal{P}_\gamma$ is a linearized Poincaré map of $\gamma$, that is $\mathcal{P}_\gamma$ is the linear map from $E_x^u \oplus E_x^s$ to itself induced by $D_x \phi^{T_\gamma}$ for some $x \in \gamma$ (the map $\mathcal{P}_\gamma$ depends on the choice of $x$, but its conjugacy class does not).
\end{itemize}

The relationship between dynamical determinant and Ruelle resonances is given by the following result.

\begin{thm}[\cite{GLP,DZdet,opensystems}]\label{thmdet}
The dynamical determinant $d_{g}$ extends holomorphically to the whole complex plane, and the zeros of this continuation are exactly the Ruelle resonances for $X$ (multiplicity taking into account).
\end{thm}

In \cite{DZdet}, Dyatlov and Zworski suggested that another relation should hold between Ruelle resonances and periodic orbits: a trace formula, that can be written as the following equality between distributions on $\R_+^*$:
\begin{equation}\label{eqtrform}
\sum_{\substack{\lambda \textup{ Ruelle resonances} \\ \textup{ of } X}} e^{\lambda t} = \sum_{\gamma} \frac{T_\gamma^{\#} e^{\int_\gamma g}}{\va{\det\p{I - \mathcal{P}_\gamma}}} \delta_{T_\gamma}.
\end{equation}
Notice that we do not use the same convention for Ruelle resonances as in \cite{DZdet}, we consider resonances for $X$ instead for $iX$, which explains why the trace formula \eqref{eqtrform} is not stated as in \cite{DZdet}. Notice also that there is \emph{a priori} no reason for which the left-hand side of \eqref{eqtrform} defines a distribution on $\R_{ + }^*$, or even converges in any sense. Showing that this is actually the case has to be part of the proof of the trace formula.

A natural way to prove such a formula would be to prove that the dynamical determinant $d_g$ continues to an entire function of finite order whose zeros are the Ruelle resonances, see for instance \cite{MunozMarco,MunozMarco2}. Recall here that the order of an entire function $f$ is (we denote by $\log_+$ the positive part of the logarithm)
\begin{equation*}
\limsup_{r \to + \infty} \sup_{\va{z} = r} \frac{\log\p{ 1 + \log_+ \va{f\p{z}}}}{\log r}.
\end{equation*}
As an example of application of trace formula, the following proposition clarifies its relationship with the dynamical determinant, see Appendix \ref{apptrdet} for the proof.

\begin{prop}\label{propapp}
The following statements are equivalent:
\begin{enumerate}[label=(\roman*)]
\item the dynamical determinant $d_g$ defined by \eqref{eqdefdet} extends to an entire function of finite order whose zeros are the Ruelle resonances;
\item the trace formula \eqref{eqtrform} holds and there is $\rho \in \R_+$ such that
\begin{equation}\label{eqsommeresonne}
\sum_{\lambda \textup{ resonances of } X} \frac{1}{1 + \va{\lambda}^\rho} < + \infty.
\end{equation}
\end{enumerate}
Moreover, when (i) and (ii) hold, the order of the holomorphic continuation of $d_g$ is less than $\rho$.
\end{prop}

The implication (ii) $\Rightarrow$ (i) in Proposition~\ref{propapp} expresses the power of the trace formula: when it holds, we may deduce information on the dynamical determinant $d_g$ through the knowledge of its zeros. Proposition~\ref{propapp} admits the following Corollary, which is of main interest when $g$ is not real-valued (when $g$ is real-valued, we may deduce a more precise result  from Jin--Zworski's local trace formula adapting the proof of \cite[Theorem 2]{ltfzwor}).

\begin{cor}\label{corphrag}
If the trace formula holds and if there is $\rho \in \left[0,1 \right[$ such that
\begin{equation}\label{eqconvergencesomme}
\sum_{\lambda \textup{ resonances of } X} \frac{1}{1 + \va{\lambda}^\rho} < + \infty,
\end{equation} 
then the function $d_g$ is constant equal to one\footnote{Notice that when $g$ is real-valued, or when $\p{\phi^t}_{t \in \R}$ has a periodic orbit $\gamma$ such that no other periodic orbit has the same length, then $d_g$ is not constant.} (in particular, $X$ has no resonances).
\end{cor}

See Appendix \ref{apptrdet} for the proof. Corollary \ref{corphrag} is interesting because it gives a lower bound on the number of Ruelle resonances for $X$. It is not far from being sharp as a general bound, consider for instance a constant time suspension of a hyperbolic linear automorphism of the torus\footnote{The Ruelle resonances for a time $1$ suspension of a cat map (with $g=0$) are the $2 i \pi k$'s for $k \in \Z$, so that\eqref{eqconvergencesomme} holds for all $\rho > 1$.}. However, we expect the existence of particular examples for which this bound is far from being sharp (see \cite{lagtf} for discrete-time examples with a lot of resonances).

In \cite{lagtf}, we studied a discrete-time analogue of the trace formula \eqref{eqtrform}. The results from \cite{lagtf} suggest that \eqref{eqtrform} may not be true for every $\mathcal{C}^\infty$ hyperbolic flow and potential but should hold for Gevrey flows with Gevrey potentials (see \cite{lagtf} or \S \ref{secDCC} for the definition of the Gevrey class of regularity). Indeed, we proved in \cite{lagtf} that, while there are $\mathcal{C}^\infty$ counter-examples to the discrete-time analogue of \eqref{eqtrform}, it holds for Gevrey uniformly hyperbolic diffeomorphisms with Gevrey potentials\footnote{In fact, the results from \cite{lagtf} and the present paper suggest that the discrete-time analogue of \eqref{eqtrform} should even hold in the class $\mathcal{C}^{\kappa,\upsilon}$ defined in \S \ref{secDCC} for $\kappa > 0$ and $\upsilon \in \left]1,2\right[$. We think that this could be proven easily using methods from \cite{lagtf} and the present paper. However, in \cite[Theorem 2.12, (v)-(vi)]{lagtf} we proved a bound on the growth of the dynamical determinant for Gevrey hyperbolic map that we do not expect to hold for $\mathcal{C}^{\kappa,\upsilon}$ dynamics. This bound is one of the reasons that make us think that the dynamical determinant of a Gevrey Anosov flow has finite order. See also \cite{DCcercle} for a detailed discussion of dynamical determinant for expanding maps of the circle in various ultradifferentiable classes.}. 

However, in order to tackle the lack of hyperbolicity of the flow $\p{\phi_t}_{t \in \R}$ for small $t \geq 0$, we will decompose the powers of the resolvent $(z-X)^{-1}$ into the contributions of large times (a compact operator) and small times (which is very small, see Lemma \ref{lmdiscretespectrum} for details) and then apply Hennion's argument \cite{hennion} based on Nussbaum formula \cite{nussbaum}. It turns out that this method also allows to give a proof of the trace formula for a larger class of flows and potentials (than Gevrey), but we were not able to prove that the dynamical determinant has finite order for these systems (there is \emph{a priori} no reason for this to be true). This procedure is very similar in spirit with the idea of working with the shifted resolvent $\mathcal{L}_{t_0}\p{z-X}^{-1}$ used in \cite{DZdet} (which is an other way to neglect small times).

In \S \ref{secDCC}, we introduce for all $\upsilon > 1$ and $\kappa \in \left]0,+\infty\right[$, a class of regularity $\mathcal{C}^{\kappa,\upsilon}$ using the language of Denjoy--Carleman classes (see \cite{nqa} for a survey on this topic). These classes are larger than any Gevrey classes of regularity. Moreover, if $M$ is $\mathcal{C}^{\kappa,\upsilon}$ and $\tilde{\upsilon} > \upsilon$, we define a space $\mathcal{D}^{\tilde{\upsilon}}\p{M}'$ of generalized distributions on $M$ and, provided that $M$, $\p{\phi^t}_{t \in \R}$, and $g$ are $\mathcal{C}^{\kappa,\upsilon}$, we extend $\mathcal{L}_t$ and $X$ to operators from $\mathcal{D}^{\tilde{\upsilon}}\p{M}'$ to itself. These notions allow us to state our main result, which states that, acting on a suitable Hilbert space, $X$ has discrete spectrum and operators obtained by integrating the semi-group $\p{\mathcal{L}_t}_{t \geq 0}$ against a smooth function supported away from $t=0$ are trace class, with an explicit formula for their traces (see for instance \cite[Chapter IV]{Gohb} for the theory of trace class operator).

\begin{thm}\label{thmmain}
Assume that there is $\kappa > 0$ and $\upsilon \in \left]1,2\right[$ such that $M,g$ and $\p{\phi^t}_{t \in \R}$ are $\mathcal{C}^{\kappa,\upsilon}$. Then for all $t_0 >0$ there is a separable Hilbert space $\h$ such that
\begin{enumerate}[label=(\roman*)]
\item for all $\tilde{\upsilon} > \upsilon$ sufficiently close to $\upsilon$, we have
$\mathcal{C}^{\infty,\tilde{\upsilon}}\p{M} \subseteq \h \subseteq \mathcal{D}^{\tilde{\upsilon}}\p{M}'$, both inclusions are continuous, and the first one has dense image;
\item for all $t \in \R_+$, the operator $\mathcal{L}_t$ defined by \eqref{eqkoopman} is bounded on $\h$;
\item $\p{\mathcal{L}_t}_{t \geq 0}$ defines a strongly continuous semi-group of operators on $\h$, whose generator coincides with $X$ on its domain, which is $\set{u \in \h: Xu \in \h}$;
\item the spectrum of $X$ acting on $\h$ consists of isolated eigenvalues of finite multiplicity which coincide with the Ruelle resonances of $X$ (multiplicity taken into account);
\item if $h : \R_+^* \to \C$ is $\mathcal{C}^\infty$ and compactly supported in $\left[t_0,+ \infty\right[$ then the operator
\begin{equation}\label{eqlapX}
\int_0^{+ \infty} h\p{t} \mathcal{L}_t \mathrm{d}t : \h \to \h
\end{equation}
is trace class and its non-zero spectrum is the intersection of $\C \setminus \set{0}$ with the image of the spectrum of $X$ by $ \lambda \mapsto \textup{Lap}\p{h}\p{- \lambda}$ (multiplicity taken into account, $\textup{Lap}\p{h}$ denotes the Laplace transform of $h$). Moreover, the trace of the operator \eqref{eqlapX} is given by
\begin{equation*}
\textup{tr}\p{\int_0^{+ \infty} h\p{t} \mathcal{L}_t \mathrm{d}t} = \sum_{\gamma} T_\gamma^{\#} \frac{h\p{T_\gamma}}{\va{\det\p{I - \mathcal{P}_\gamma}}} \exp\p{\int_\gamma g},
\end{equation*}
where the sum on the right-hand side runs over periodic orbits $\gamma$ of the flow $\p{\phi^t}_{t \in \R}$.
\end{enumerate} 
\end{thm}

With Lidskii's trace theorem \cite[Theorem 6.1 p.63]{Gohb}, the last point of Theorem~\ref{thmmain} implies the following Corollary.

\begin{cor}[Trace formula for ultradifferentiable Anosov flows]\label{cortrform}
If $M,g$ and $\p{\phi^t}$ are $\mathcal{C}^{\kappa,\upsilon}$ for some $\kappa > 0$ and $\upsilon \in \left]1,2\right[$ then the trace formula \eqref{eqtrform} holds. In particular, the right-hand side of \eqref{eqtrform} defines a distribution.
\end{cor} 

Maybe it would be more satisfactory to be able to prove that the right-hand side of \eqref{eqtrform} is a distribution on $\R_+^*$ before proving the trace formula. Under the hypothesis of Theorem~\ref{thmmain}, it can be deduced from the fact that the trace class operator norm of \eqref{eqlapX} is less than $C \n{h}_{\mathcal{C}^{d+3}}$ for some constant $C >0$ that depends on $h$ only through its support (this may be deduced from the proof of Theorem~\ref{thmmain}), or from the following by-product of the proof of Theorem~\ref{thmmain}.

\begin{prop}\label{propmain}
If $M,g$ and $\p{\phi^t}_{t \in \R}$ are $\mathcal{C}^{\kappa,\upsilon}$ for some $\kappa > 0$ and $\upsilon \in \left]1,2\right[$ then for all $\epsilon > 0$ we have
\begin{equation*}
\sum_{\lambda \textup{ resonances of } X} \frac{e^{\epsilon \Re\p{\lambda}}}{1 + \va{\lambda}^{d+1+\epsilon}} < + \infty.
\end{equation*}
\end{prop} 

The bound on the number of resonances given by Proposition \ref{propmain} is not sufficient to apply Proposition \ref{propapp} and get a Hadamard factorization \cite[Theorem 2.7.1]{Boas} for the dynamical determinant $d_g$. However, we will derive in Appendix \ref{appdet} a ``Hadamard-like'' factorization for $d_g$.

Finally, although we need $\upsilon < 2$ to prove trace formula, most of the statements in Theorem~\ref{thmmain} remain true when $\upsilon \geq 2$. We discuss in Appendix \ref{appupsilon} the relevance and necessity of the condition $\upsilon < 2$ through the simplest possible example: the doubling map on the circle. See also \cite{DCcercle} for a discussion of transfer operators for dynamics in more general classes of ultradifferentiability.

\begin{prop}\label{propmqr}
If, in Theorem~\ref{thmmain}, we allow $\upsilon \geq 2$, then there is still a Hilbert space $\h$ satisfying (i),(ii),(iii) and (iv). Moreover, under the hypothesis of (v), the operator \eqref{eqlapX} is compact and its spectrum can be described as in Theorem~\ref{thmmain} in terms of Ruelle resonances.
\end{prop}

\section{Denjoy--Carleman classes and ultradifferentiable functions}\label{secDCC}

We define now the classes of regularity $\mathcal{C}^{\kappa,\upsilon}$ that appear in Theorem~\ref{thmmain}. To do so we use the language of Denjoy--Carleman classes, see \cite{nqa} for a survey on this topic. We will also define spaces $\mathcal{D}^\upsilon\p{M}'$ of generalized distributions which are needed because the space $\h$ of Theorem~\ref{thmmain} is not included in the usual space of distributions $\mathcal{D}'\p{M}$ on $M$.

Let $A = \p{A_m}_{m \in \N}$ be an increasing sequence of positive real numbers and $U$ be an open subset of $\R^d$. We define the Denjoy--Carleman class $\mathcal{C}^A\p{U}$ to be the space of $\mathcal{C}^\infty$ functions $f: U \to \C$ such that for each compact subset $K$ of $U$ there are constants $C,R >0$ such that for all $\alpha \in \N^d$ and $x \in K$ we have
\begin{equation*}
\va{\partial^{\alpha} f\p{x}} \leq C R^{\va{\alpha}} \va{\alpha} ! A_{\va{\alpha}}.
\end{equation*}

If $A_m =  ( m! ) ^{\sigma - 1}$ for some $\sigma >1$, the class $\mathcal{C}^A\p{U}$ is the class of $\sigma$-Gevrey functions. For $\sigma = 1$, this is just the class of real-analytic functions on $U$. We will consider in this paper the following classes, that are larger than Gevrey classes. If $\kappa >0$ and $ \upsilon >1 $ define the sequence $A\p{\kappa,\upsilon} = \p{A_m\p{\kappa,\upsilon}}_{m \in \N}$ by
\begin{equation*}
\forall m \in \N : A_m\p{\kappa,\upsilon} = \exp\p{\frac{m^{\upsilon}}{\kappa}}.
\end{equation*}
Notice that the the sequence $A\p{\kappa,\upsilon}$ is logarithmically convex:
\begin{equation*}
\forall m \in \N^* : A_m\p{\kappa,\upsilon}^2 \leq A_{m-1}\p{\kappa,\upsilon} A_{m+1}\p{\kappa,\upsilon}.
\end{equation*}
It is then a well-established fact (see e.g. \cite{nqa} and references therein) that the class $\mathcal{C}^{\kappa,\upsilon} \coloneqq \mathcal{C}^{A\p{\kappa,\upsilon}}$ is closed under multiplication, composition, the inverse function theorem and solving ODEs. Notice that the class $\mathcal{C}^{\kappa,\upsilon}$ is closed under differentiation\footnote{It follows from the fact that the condition (2.1.6) from \cite{nqa} is satisfied if and only if $\upsilon \leq 2$.} if and only if $\upsilon \leq 2$. Since $\mathcal{C}^{\kappa,\upsilon}$ is greater than any Gevrey class, it is non-quasi-analytic and contains partitions of unity. We are not aware of any references in the literature dealing specifically with the classes $\mathcal{C}^{\kappa,\upsilon}$ that we use here. However, the particular classes that are studied in \cite{ultsim} for instance and the classes $\mathcal{C}^{\kappa,\upsilon}$ look a bit alike. We will sometimes refer to the objects (functions, manifolds, etc) having $\mathcal{C}^{\kappa,\upsilon}$ regularity as ultradifferentiable objects. Beware that this is not in any way a canonical name.

The Fourier transform will be a key tool in this paper, it is thus natural to introduce a suitable class of rapidly decreasing functions and associated spaces of tempered generalized distributions. This is often done in the literature, in particular when dealing with Gevrey classes (see for instance \cite{pilipotemp, pilipoconv}). Notice that we will use the following convention for the Fourier transform: if $f \in L^{1}\p{\R^d}$ and $\xi \in \R^d$ we set
\begin{equation*}
\mathbb{F}\p{f}\p{\xi} = \hat{f}\p{\xi} = \int_{\R^d} e^{-i x \xi}f\p{x} \mathrm{d}x.
\end{equation*}
For all $ \kappa > 0$, $\upsilon >1 $ and $f \in \mathcal{C}^{\infty}\p{\R^d}$, define
\begin{equation*}
\n{f}_{\kappa,\upsilon} = \sup_{\substack{x \in \R^d \\ \alpha \in \N^d \\ m \in \N}} \p{1 + \va{x}}^m \va{\partial^{\alpha} f \p{x}} \exp\p{- \frac{\p{m + \va{\alpha}}^\upsilon}{\kappa}}.
\end{equation*}
Then define, for $\upsilon > 1$,
\begin{equation}\label{eqdefsupsilon}
\mathcal{S}^{\upsilon} = \set{ f \in \mathcal{C}^{\infty}\p{\R^d} : \forall \kappa \in \R_+^*, \n{f}_{\kappa,\upsilon} < + \infty},
\end{equation}
which is a Fréchet space when endowed with the family of semi-norms $\n{\cdot}_{\kappa,\upsilon}$ for $\kappa > 0$. Notice that $\mathcal{S}^\upsilon$ is contained in the usual space of Schwartz functions and that the elements of $\mathcal{S}^\upsilon$ are in the Denjoy--Carleman class $\mathcal{C}^{\kappa,\upsilon}$ for every $\kappa > 0$. One may also check that $\mathcal{S}^{\upsilon}$ is closed under differentiation. We will denote by $\p{\mathcal{S}^{\upsilon}}'$ the space of continuous linear forms on $\mathcal{S}^{\upsilon}$ endowed with the weak-star topology. This space will play the role of tempered distributions in our context.

\begin{prop}\label{propfourier}
If $\upsilon >1$, then the Fourier transform from $\mathcal{S}^{\upsilon}$ to itself is a continuous isomorphism.
\end{prop}

\begin{proof}
We start by proving that the Fourier transform is continuous from $\mathcal{S}^\upsilon$ to itself. Let $0 < \kappa' < \kappa$. Let $f \in \mathcal{S}^\upsilon$ and recall that for all $\xi \in \R^d$ and $\alpha,\beta \in \N^d$ we have\footnote{There is an error in the expression for $\xi^\alpha \partial^\beta \hat{f}\p{\xi}$ in the proof of \cite[Proposition 5.3]{lagtf}. However, the proof is easily fixed by using the correct formula that we give here.}
\begin{equation*}
\begin{split}
\xi^\alpha \partial^\beta \hat{f}\p{\xi} & = (-i)^{\va{\alpha} + \va{\beta}} \int_{\R^d} e^{-i x \xi} \partial^\alpha\p{x^\beta f(x)} \mathrm{d}x\\
        & = (-i)^{\va{\alpha} + \va{\beta}} \sum_{\substack{\gamma_1 + \gamma_2 = \alpha \\ \gamma_2 \preceq \beta}} \frac{\alpha! \beta!}{\gamma_1!\gamma_2!\p{\beta - \gamma_2}!} \int_{\R^d} e^{-i x \xi} x^{\beta - \gamma_2} \partial^{\gamma_1} f(x) \mathrm{d}x,
\end{split}
\end{equation*}
where $\gamma_2 \preceq \beta$ means that each coordinate of $\gamma_2$ is smaller than the corresponding coordinate of $\beta$. Then, notice that there is a constant $C > 0$ such that, for every $\gamma_1,\gamma_2,\beta \in \N^d$ such that $\gamma_2 \preceq \beta$, we have
\begin{equation*}
\begin{split}
\va{\int_{\R^d} e^{_i x \xi} x^{\beta - \gamma_2} \partial^{\gamma_1} f(x) \mathrm{d}x} \leq C \n{f}_{\kappa,\upsilon} \exp\p{\frac{\p{\va{\beta}- \va{\gamma_2} + \va{\gamma_1} + d + 1}^\upsilon}{\kappa}}.
\end{split}
\end{equation*}
Moreover, up to making $C$ larger we also have, for every $\gamma_2 \in \N^d$,
\begin{equation*}
\begin{split}
\gamma_2 ! \leq C \exp\p{\frac{\va{\gamma_2}^\upsilon}{\kappa}}.
\end{split}
\end{equation*}
Consequently, we find that for all $\xi \in \R^d$ and all $\alpha,\beta \in \N^d$, the quantity $\va{\xi^\alpha \partial^\beta \hat{f}(\xi)}$ is smaller than 
\begin{equation*}
\begin{split}
& C^2 \n{f}_{\kappa,\upsilon} \sum_{\substack{\gamma_1 + \gamma_2 = \alpha \\ \gamma_2 \preceq \beta}} \frac{\alpha!}{\gamma_1! \gamma_2!} \frac{\beta!}{\gamma_2! \p{\beta - \gamma_2}!} \exp\p{\frac{\va{\gamma_2}^\upsilon + \p{\va{\beta}- \va{\gamma_2} + \va{\gamma_1} + d + 1}^\upsilon}{\kappa}} \\
      & \qquad \qquad \qquad \qquad \qquad \qquad \qquad \leq C^2 \n{f}_{\kappa,\upsilon} 2^{\va{\alpha} + \va{\beta}} \exp\p{\frac{\p{\va{\alpha} + \va{\beta} + d + 1}^\upsilon}{\kappa}}.
\end{split}
\end{equation*}

Using the fact that for $\ell \in \N$
\begin{equation*}
\va{\xi}^{2\ell} = \p{\sum_{j=1}^d \va{\xi_j}^2}^{\ell} = \sum_{\va{\alpha} = \ell} c\p{\alpha} \xi^{2\alpha},
\end{equation*}
where $\sum_{\va{\alpha} = \ell} c\p{\alpha} = d^{\ell}$, we see that, for some new constant $C > 0$, we have for all $m \in \N, \xi \in \R^d$ and $\alpha,\beta \in \N^d$:
\begin{equation}\label{eqmajfour}
\p{1 + \va{\xi}}^m \va{\partial^\beta \hat{f}\p{\xi}} \leq C \n{f}_{\kappa,\upsilon} \p{4 \sqrt{d}}^m 2^{\va{\beta}} \exp\p{ \frac{\p{ m + \va{\beta} + d + 2}^\upsilon}{\kappa}}.
\end{equation}
Indeed, we can deal first with the case $m$ even and then argue that $\p{1 + \va{\xi}}^m \leq \p{1 + \va{\xi}}^{m + 1}$. Finally, since $\kappa' < \kappa$ and $\frac{\p{n+d+2}^{\upsilon}}{\kappa} - \frac{n^{\upsilon}}{\kappa'} \underset{n \to + \infty}{\sim} - \frac{\kappa - \kappa'}{\kappa \kappa'} n^\upsilon  $, we see that, for some new constant $C > 0$, we have
\begin{equation*}
\n{\hat{f}}_{\kappa',\upsilon} \leq C \n{f}_{\kappa,\upsilon},
\end{equation*}
and the Fourier transform is indeed continuous from $\mathcal{S}^\upsilon$ to itself. The same argument gives that the inverse Fourier transform is also continuous from $\mathcal{S}^\upsilon$ to itself. Moreover, since $\mathcal{S}^\upsilon$ is included in the space of Schwartz function on $\R^d$, the elements of $\mathcal{S}^\upsilon$ satisfy the Fourier Inversion Formula. Hence, the Fourier transform is indeed a continuous automorphism of $\mathcal{S}^\upsilon$.
\end{proof}

Proposition~\ref{propfourier} allows to define the Fourier transform on $\p{\mathcal{S}^\upsilon}'$ by duality in the usual way. Since $\mathcal{S}^{\upsilon}$ is closed by multiplication, for all $\psi \in \mathcal{S}^\upsilon$ we may define the Fourier multiplier $\psi\p{D} : \p{\mathcal{S}^\upsilon}' \to \p{\mathcal{S}^\upsilon}'$ by
\begin{equation*}
\forall u \in \p{\mathcal{S}^\upsilon}' : \psi\p{D} u = \mathbb{F}^{-1}\p{\psi. \hat{u}}.
\end{equation*}
It is well-known that the Fourier transform of a $\mathcal{C}^{\infty}$ compactly supported function decays faster than the inverse of any polynomial. For functions in the class $\mathcal{C}^{\kappa,\upsilon}$ this statement is made quantitative in Proposition~\ref{propdecfour} below. This is the key point that will allow us in \S \ref{sectls} to construct Sobolev-like spaces of anisotropic generalized distributions that are the pieces from which we will construct the space $\h$ from Theorem~\ref{thmmain} in \S \ref{gs1} and \S \ref{gs2}.

\begin{prop}\label{propdecfour}
For every $R>0$ and $\upsilon > 1$, there are constants $C >0$ and $\kappa >0$ such that, for all $f \in \mathcal{S}^\upsilon$ and $\xi \in \R^d$, we have
\begin{equation}\label{eqfourierupsilon}
\va{\hat{f}\p{\xi}} \leq C \n{f}_{\kappa,\upsilon} \exp\p{-R ( \ln\p{1 + \va{\xi}})^\frac{\upsilon}{\upsilon-1}}.
\end{equation}
\end{prop}

\begin{proof}
Choose $\kappa > 0$ large enough so that 
\begin{equation*}
R' \coloneqq \kappa^{\frac{1}{\upsilon -1}}\p{\frac{1}{\upsilon^{\frac{1}{\upsilon-1}}} - \frac{1}{\upsilon^{\frac{\upsilon}{\upsilon-1}}}} > R.
\end{equation*}
Then apply \eqref{eqmajfour} from the proof of Proposition~\ref{propfourier} with $\beta = 0$ to get a constant $C >0$ such that, for all $ \xi \in \R^d$ and $m \in \N$, we have
\begin{equation*}
\va{\hat{f}\p{\xi}} \leq C \n{f}_{\kappa,\upsilon} \p{\frac{4 \sqrt{d}}{1 + \va{\xi}}}^m \exp\p{\frac{\p{m + d + 2}^{\upsilon}}{\kappa}}.
\end{equation*}
When $\va{\xi}$ is small, we bound $\hat{f}(\xi)$ by taking $m = 0$. When $\va{\xi}$ is large enough so that the following expression makes sense and is non-negative, we take
\begin{equation*}
m = \left\lfloor \p{- \ln \p{\frac{4 \sqrt{d}}{1 + \va{\xi}}}}^{\frac{1}{\upsilon-1}} \p{\frac{\kappa}{\upsilon}}^{\frac{1}{\upsilon - 1}} - d - 2 \right \rfloor.
\end{equation*}
With this choice of $m$ we have
\begin{equation*}
\begin{split}
& \p{\frac{4 \sqrt{d}}{1 + \va{\xi}}}^m \exp\p{\frac{\p{m + d + 2}^{\upsilon}}{\kappa}} \\
     & \qquad \leq \exp\Bigg( \p{\va{\ln \p{\frac{4 \sqrt{d}}{1 + \va{\xi}}}}^{\frac{1}{\upsilon-1}}\p{\frac{\kappa}{\upsilon}}^{\frac{1}{\upsilon - 1}} - d -3}  \ln\p{\frac{4 \sqrt{d}}{1 + \va{\xi}}} \\
     & \qquad \qquad \qquad \qquad \qquad \qquad \qquad \qquad \qquad + \frac{\kappa^{\frac{1}{\upsilon - 1}}}{\upsilon^{\frac{\upsilon}{\upsilon-1}}} \va{\ln\p{\frac{4 \sqrt{d}}{1 + \va{\xi}}}}^{\frac{\upsilon}{\upsilon - 1}}\Bigg) \\
     & \qquad \leq \p{\frac{1+ \va{\xi}}{4 \sqrt{d}}}^{d+3} \exp\p{\kappa^{\frac{1}{\upsilon - 1}}\p{\frac{1}{\upsilon^{\frac{\upsilon}{\upsilon-1}}} - \frac{1}{\upsilon^{\frac{1}{\upsilon-1}}}} \p{\ln\p{\frac{1 + \va{\xi}}{4 \sqrt{d}}}}^{\frac{\upsilon}{\upsilon - 1}}} \\
     & \qquad \leq \p{\frac{1+ \va{\xi}}{4 \sqrt{d}}}^{d+3} \exp\p{ - R' \p{\ln\p{\frac{1 + \va{\xi}}{4 \sqrt{d}}}}^{\frac{\upsilon}{\upsilon - 1}}},
\end{split}
\end{equation*}
and the result follows then from the fact that (recall that $R' > R$)
\begin{equation*}
\begin{split}
\p{\frac{1+r}{4 \sqrt{d}}}^{d+3} \exp\p{R \p{\ln\p{1 + r}}^{\frac{\upsilon}{\upsilon-1}} - R' \p{\ln\p{\frac{1 + r}{4 \sqrt{d}}}}^{\frac{\upsilon}{\upsilon - 1}}} \underset{r \to + \infty}{\to} 0.
\end{split}
\end{equation*}

\end{proof}

We need to extend the notion of ultradifferentiability to more general objects than complex-valued functions in order to make sense of Theorem~\ref{thmmain}. For instance, we will define what a $\mathcal{C}^{\kappa,\upsilon}$ manifold is. To do it, we follow ideas that may be found in \cite{nqa}, notice however that when $\upsilon > 2$ the sequence $\p{A_m}_{m \in \N}$ is not a DC-weight sequence in the sense of \cite{nqa}, so that we cannot apply most of their results. Hopefully, it will be clear in the remaining of the section that, whereas the general theory of our ultradifferentiability classes may not be very satisfactory, this is of no harm in our pedestrian approach to the problem of the trace formula.

We say that a map $f : U \to \C^N$, where $N$ is some integer, is $\mathcal{C}^{\kappa,\upsilon}$ if its components are $\mathcal{C}^{\kappa,\upsilon}$. A $\mathcal{C}^{\kappa,\upsilon}$ manifold is a $\mathcal{C}^\infty$ manifold endowed with a maximal atlas whose changes of charts are $\mathcal{C}^{\kappa,\upsilon}$. A map $f : M \to N$ between two $\mathcal{C}^{\kappa,\upsilon}$ manifolds is said to be $\mathcal{C}^{\kappa,\upsilon}$ if it is $\mathcal{C}^{\kappa,\upsilon}$ ``in charts".

We define now topological vector spaces associated to the classes of regularity defined above. If $M$ is a $\mathcal{C}^{\kappa,\upsilon}$ manifold for some $\kappa >0$ and $\upsilon >1$ then $M$ has a natural $\mathcal{C}^{\kappa',\tilde{\upsilon}}$ manifold structure for all $\kappa' >0$ and $ \tilde{\upsilon} > \upsilon$, so that we may define the class $\mathcal{C}^{\infty,\tilde{\upsilon}}\p{M}$ of functions from $M$ to $\C$ that are $\mathcal{C}^{\kappa',\tilde{\upsilon}}$ for all $\kappa' >0$. Notice that all $\mathcal{C}^{\kappa,\upsilon}$ functions from $M$ to $\C$ belong to $\mathcal{C}^{\infty,\tilde{\upsilon}}\p{M}$ if $\tilde{\upsilon} > \upsilon$. 

Notice that if $\upsilon > 2$ then the class $\mathcal{C}^{\kappa,\upsilon}$ is not closed under differentiation and in particular in this case the tangent bundle $TM$ has no natural $\mathcal{C}^{\kappa,\upsilon}$ structure. However, derivatives of $\mathcal{C}^{\kappa,\upsilon}$ functions are $\mathcal{C}^{\kappa',\upsilon}$ for all $ 0 <\kappa' < \kappa$. Thus the tangent bundle $TM$ may be endowed naturally with a $\mathcal{C}^{\kappa',\upsilon}$ structure, so that it makes sense to talk about a $\mathcal{C}^{\kappa',\tilde{\upsilon}}$ vector field when $\tilde{\upsilon} > \upsilon$, or $\tilde{\upsilon} = \upsilon$ and $\kappa' < \kappa$. Integrating such a vector field gives rise to a $\mathcal{C}^{\kappa',\tilde{\upsilon}}$ flow $\p{\phi^t}_{t \in \R}$ (that is, the map $\p{x,t} \mapsto \phi^t\p{x}$ is $\mathcal{C}^{\kappa',\tilde{\upsilon}}$), see \cite{nqa,dcode}. A consequence of this fact is that if $V$ is a $\mathcal{C}^{\kappa',\tilde{\upsilon}}$ vector field on $M$ that does not vanish then $V$ is locally conjugated \emph{via} $\mathcal{C}^{\kappa',\tilde{\upsilon}}$ charts to a constant vector field on $\R^{d}$. This implies in particular that if $\upsilon' > \tilde{\upsilon}$ then $\mathcal{C}^{\infty,\upsilon'}$ is stable under differentiation with respect to $V$ (this operation is even continuous with respect to the topology that we define below).

If $M$ is compact, we endow $\mathcal{C}^{\infty,\tilde{\upsilon}}\p{M}$ with a structure of Fréchet space in the following way: if $U$ is an open subset of $M$ and $V$ is an open subset of $\R^d$, if $\psi : U \to V$ is a $\mathcal{C}^{\infty,\tilde{\upsilon}}$ chart, $\varphi$ is an element of $\mathcal{C}^{\infty,\tilde{\upsilon}}$ supported in $U$ and $\kappa >0$, define the semi-norm $\n{\cdot}_{\psi,\varphi,\kappa,\tilde{\upsilon}}$ by
\begin{equation*}
\forall u \in \mathcal{C}^{\infty,\tilde{\upsilon}} : \n{u}_{\psi,\varphi,\kappa,\tilde{\upsilon}} = \sup_{\substack{\alpha \in \N^d \\ x \in V }} \va{\partial^{\alpha} \p{\p{\varphi u}\circ \psi^{-1}}\p{x}} \exp\p{- \frac{\va{\alpha}^{\tilde{\upsilon}}}{\kappa}}.
\end{equation*} 
The topology of $\mathcal{C}^{\infty,\tilde{\upsilon}}\p{M}$ is generated by a countable family of these semi-norms: since $M$ is compact we can cover $M$ by a finite number of domain of charts and take a partition of unity subordinated to this cover, then we only need to let $\kappa$ runs through the integers. The completeness of $\mathcal{C}^{\infty,\tilde{\upsilon}}\p{M}$ is easily verified. One can also check using Leibniz formula that pointwise multiplication $\mathcal{C}^{\infty,\tilde{\upsilon}}\p{M} \times \mathcal{C}^{\infty,\tilde{\upsilon}}\p{M} \to \mathcal{C}^{\infty,\tilde{\upsilon}}\p{M}$ is continuous. Notice also that if $N$ is another $\mathcal{C}^{\kappa,\upsilon}$ manifold and $\psi : M \to N$ is a $\mathcal{C}^{\kappa,\upsilon}$ local diffeomorphism then the map $\mathcal{C}^{\infty,\tilde{\upsilon}}\p{N} \ni u \mapsto u \circ \psi \in \mathcal{C}^{\infty,\tilde{\upsilon}}\p{M}$ is continuous.

We will also need the space $\mathcal{D}^{\tilde{\upsilon}}\p{M}$ of $\mathcal{C}^{\infty,\tilde{\upsilon}}$ densities on $M$: this is the space of complex measures of $M$ which are absolutely continuous with respect to Lebesgue and whose density in any $\mathcal{C}^{\infty,\tilde{\upsilon}}$ chart is $\mathcal{C}^{\infty,\tilde{\upsilon}}$. We endow $\mathcal{D}^{\tilde{\upsilon}}\p{M}$ with a Fréchet structure as we did for $\mathcal{C}^{\infty,\tilde{\upsilon}}\p{M}$ (notice that these two spaces may be identified by the choice of a particular element of $\mathcal{D}^{\tilde{\upsilon}}(M)'$). We will denote by $\mathcal{D}^{\tilde{\upsilon}}\p{M}'$ the space of continuous linear functionals on $M$ on $\mathcal{D}^{\tilde{\upsilon}}\p{M}$, that we endow with the weak-star topology. Notice that if $u \in \mathcal{C}^{\infty,\tilde{\upsilon}}\p{M}$ then $u$ defines an element of $\mathcal{D}^{\tilde{\upsilon}}\p{M}'$ that we also denotes by $u$, by the formula
\begin{equation*}
\forall \mu \in \mathcal{D}^{\tilde{\upsilon}}\p{M} : \langle u , \mu \rangle = \int_{M} u \mathrm{d}\mu.
\end{equation*}
We define in this way an injection of $\mathcal{C}^{\infty,\tilde{\upsilon}}\p{M}$ into $\mathcal{D}^{\tilde{\upsilon}}\p{M}'$ that can be shown to be continuous and to have dense image (by mollifying elements of $\mathcal{D}^{\tilde{\upsilon}}\p{M}'$ by convolution for instance).

Now, let $M$ be a $\p{d+1}$-dimensional $\mathcal{C}^{\kappa,\upsilon}$ compact manifold for some $\kappa > 0$ and $\upsilon >1$. Let $\p{\phi^t}_{t \in \R}$ be a $\mathcal{C}^{\kappa,\upsilon}$ flow on $M$ (that is, the map $M \times \R \ni \p{x,t} \mapsto \phi^t\p{x}$ is $\mathcal{C}^{\kappa,\upsilon}$). Then the generator $V$ of the flow $\p{\phi^t}_{t \in \R}$ is a $\mathcal{C}^{\kappa',\upsilon}$ vector field for all $\kappa' < \kappa$. Choose $g : M \to \C$ a $\mathcal{C}^{\kappa,\upsilon}$ potential. Let $\tilde{\upsilon} > \upsilon$ and define for all $t \in \R$ the continuous operator $\mathcal{L}_t$ on $\mathcal{C}^{\infty,\tilde{\upsilon}}\p{M}$ by
\begin{equation*}
\forall u \in \mathcal{C}^{\infty,\tilde{\upsilon}}\p{M}: \forall x \in M : \mathcal{L}_t u\p{x} = \exp\p{\int_0^t g \circ \phi^{\tau}\p{x} \mathrm{d}\tau} u \circ \phi^{t}\p{x}.
\end{equation*}
Here, let us notice that the prefactor in the definition of $\mathcal{L}_t$ is a $\mathcal{C}^{\kappa,\upsilon}$ function (since this class of regularity is closed under composition). It is convenient\footnote{It makes easier to prove that Ruelle resonances are intrinsic in Appendix \ref{apprri} or to define the norm $\n{\cdot}_{\h}$ in \eqref{eqdefnh} for instance.} to extend $\mathcal{L}_t$ and $X = V+g$ from $\mathcal{D}^{\tilde{\upsilon}}\p{M}'$ to itself. To do so, we need to compute their adjoints. Choose $\mu \in \mathcal{D}^{\tilde{\nu}}\p{M}$ positive and fully supported, it induces an isomorphism between $\mathcal{D}^{\tilde{\upsilon}}\p{M}$ and $\mathcal{C}^{\infty,\tilde{\upsilon}}\p{M}, \nu \mapsto \frac{\mathrm{d}\nu}{\mathrm{d}\mu}$. Then notice that $\frac{\mathrm{d}\p{\p{\phi^t}_* \mu}}{\mathrm{d}\mu}$ satisfies for all $x \in M$ and $t,t' \in \R$ the cocycle equation
\begin{equation*}
\frac{\mathrm{d}\p{\p{\phi^{t+t'}}_* \mu}}{\mathrm{d}\mu}\p{x} = \frac{\mathrm{d}\p{\p{\phi^{t'}}_* \mu}}{\mathrm{d}\mu}\p{x} \frac{\mathrm{d}\p{\p{\phi^t}_* \mu}}{\mathrm{d}\mu}\p{\phi^{-t'}\p{x}},
\end{equation*}
so that we have
\begin{equation*}
\forall x \in M: \forall t \in \R : \frac{\mathrm{d}\p{\p{\phi^t}_* \mu}}{\mathrm{d}\mu}\p{x} = \exp\p{ - \int_0^t \textup{div}\p{V} \circ \phi^{- \tau}\p{x} \mathrm{d}\tau},
\end{equation*}
where the divergence of $V$ is defined by 
\begin{equation*}
\forall x \in M : \textup{div}\p{V}\p{x} = - \frac{\mathrm{d}}{\mathrm{d}t} \left. \p{ \frac{\mathrm{d}\p{\p{\phi^t}_* \mu}}{\mathrm{d}\mu}\p{x}} \right|_{t = 0}.
\end{equation*}
Notice that $\textup{div}\p{V}$ is a $\mathcal{C}^{\kappa',\upsilon}$ function for all $\kappa ' < \kappa$. Then the formal adjoint of $\mathcal{L}_t$ may be defined on $\mathcal{D}^{\tilde{\upsilon}}\p{M}$ by
\begin{equation*}
\p{\mathcal{L}_t}^* \nu = \exp\p{\int_0^t \p{g - \textup{div}\p{V}} \circ \phi^{- \tau}\mathrm{d}\tau} \frac{\mathrm{d}\nu}{\mathrm{d}\mu} \circ \phi^{- \tau} \mathrm{d}\mu
\end{equation*}
and the formal adjoint of $X$ by
\begin{equation*}
X^* \nu = \p{-V - \textup{div}\p{V} + g} \frac{\mathrm{d}\nu}{\mathrm{d}\mu} \mathrm{d}\mu.
\end{equation*}
These two operators are continuous, so that $X$ and $\mathcal{L}_t$ may be extended as continuous operators on $\mathcal{D}^{\tilde{\upsilon}}\p{M}'$. Notice that $X$ and $\mathcal{L}_t$ commute.

We will need Lemmas \ref{lmderiv} and~\ref{lmdomain} to prove Theorem~\ref{thmmain}. Their proofs are given in Appendix \ref{appDCC}.

\begin{lm}\label{lmderiv}
\begin{enumerate}[label=(\roman*)]
\item If $u \in \mathcal{C}^{\infty,\tilde{\upsilon}}\p{M}$ then the map $\R \ni t \mapsto \mathcal{L}_t u \in \mathcal{C}^{\infty,\tilde{\upsilon}}\p{M}$ is $\mathcal{C}^\infty$ and its derivative is $t \mapsto \mathcal{L}_t X u = X \mathcal{L}_t u$.
\item If $u \in \mathcal{D}^{\tilde{\upsilon}}\p{M}'$ then the map $\R \ni t \mapsto \mathcal{L}_t u \in \mathcal{D}^{\tilde{\upsilon}}\p{M}'$ is $\mathcal{C}^\infty$ and its derivative is $t \mapsto \mathcal{L}_t X u = X \mathcal{L}_t u$.
\end{enumerate}
\end{lm}

\begin{lm}\label{lmdomain}
Let $\B$ be a Banach space such that $\B \subseteq \mathcal{D}^{\tilde{\upsilon}}\p{M}'$, the inclusion being continuous. Assume that, for all $t \in \R_+$, the operator $\mathcal{L}_t$ is bounded from $\B$ to itself, and that $\p{\mathcal{L}_t}_{t \geq 0}$ is a strongly continuous semi-group of operator of $\B$. Then the generator of $\p{\mathcal{L}_t}_{t \geq 0}$ coincides with $X$ on its domain which is
\begin{equation*}
\set{u \in \B : Xu \in \B}.
\end{equation*}
\end{lm}

\section{Local spaces}\label{sectls}

We define now ``local'' spaces $\h_{\Theta,\alpha}$ that will be the basic pieces to construct the space $\h$ from Theorem~\ref{thmmain}. These spaces will depend on the choice of a system of cones $\Theta$: this system encodes the three distinguished directions from Definition \ref{defanosov} of an Anosov flow (that is why the space is called anisotropic). These spaces are Sobolev-like spaces similar to the spaces from \cite[Definition 4.16]{Bal2}  or from \cite{Tsu0} (for discrete-time systems) or \cite{phdadam,adamarxiv} (even though the approach is a bit different, spaces in \cite{FauSjo} are also Sobolev-like spaces). As in \cite{Tsu0,Tsu,Bal2,phdadam,adamarxiv}, we will use Paley--Littlewood decomposition to study these spaces and the action of Koopman operators on them. However, as in \cite{lagtf}, we cannot use the usual dyadic Paley--Littlewood decomposition since the weights that we use to define our Sobolev-like spaces have a growth faster than polynomial, so that we will introduce an adapted Paley--Littlewood-like decomposition.

First of all, we need to define the systems of cones that we will use. As in \cite{lagtf}, we need to consider system of potentially a large number of cones, in order to deal with the low hyperbolicity of the flow for small times. The interior and the adherence of a subset $X$ of a topological space will be denoted respectively by $\bul{X}$ and $\overline{X}$. If $C$ and $C'$ are two cones in an Euclidean space, we write $C \Subset C'$ for $\overline{C} \subseteq \bul{C'} \cup \set{0}$. The dimension of a cone $C$ in an Euclidean space $E$ is by definition the maximum dimension of a linear subspace of $E$ contained in $C$.

\begin{df}[System of cones]\label{defsyst}
Let $\p{E,\langle ., . \rangle}$ be an Euclidean vector space, $e \in E$ and $r \geq 2$ be an integer. A system of $r+2$ cones with respect to the direction $e$ is a family $\Theta =\p{C_0,C_1,\dots,C_r,C_f}$ of non-empty closed cones in $E$ such that
\begin{enumerate}[label=(\roman*)]
\item $\stackrel{\circ}{C_0} \cup \bul{C_1} \cup \bul{C_f} = E \setminus \set{0}$;
\item $C_f$ is one-dimensional and there is $c >0$ such that for all $\xi \in  C_f $ we have $ \va{\langle \xi,e\rangle} \geq c \va{\xi}$;
\item there are integers $d_u$ and $d_s$ such that $d_u + d_s +1 = \dim E$, $C_0$ is $d_s$-dimensional and, for $i \in \set{1,\dots,r}$, the cone $C_i$ is $d_u$-dimensional;
\item if $i \in \set{1,\dots,r-1}$ then $C_{i+1} \Subset C_i$;
\item $C_0 \cap C_2 = C_f \cap C_2 = \set{0}$.
\end{enumerate}
\end{df}

$\R^{d+1}$ will always be endowed with its canonical Euclidean structure and system of cones in $\R^{d+1}$ will always be with respect to the direction of $e_{d+1} = \p{0,\dots,0,1}$. We will mainly use Definition \ref{defsyst} with $E = \R^{d+1}$, however, it will be convenient in the proof of Lemma~\ref{lmdecoupe} to have at our disposal the definition of a system of cones in a general Euclidean space. 

If $\p{C_0,\dots,C_r,C_f}$ is a system of $r+2$ cones in $\R^{d+1}$ (with respect to the direction $e_{d+1}$) then we can choose $\p{\varphi_0, \varphi_1, \dots,\varphi_{r-1},\varphi_f}$ a Gevrey\footnote{This ensures that it is $\mathcal{C}^{\kappa,\upsilon}$ for any $\kappa > 0$ and $\upsilon >1$, so that all the Fourier multipliers that appear later are automatically well-defined.} partition of unity on $\mathbb{S}^d$ such that:
\begin{itemize}
\item for $i \in \set{0,\dots,r-1,f}$, the function $\varphi_i$ is supported in the interior of $C_i \cap \mathbb{S}^d$;
\item if $i \in \set{1,\dots,r-2}$ then $\varphi_i$ vanishes on a neighborhood of $\mathbb{S}^d \cap C_{i+2}$.
\end{itemize}
Indeed, the interiors of $C_0\cap \mathbb{S}^d$, $\p{C_f \setminus C_2} \cap \mathbb{S}^d$, $\p{C_1 \setminus C_3} \cap \mathbb{S}^d, \dots, \p{C_{r-2} \setminus C_r} \cap \mathbb{S}^d$ and $C_{r-1} \cap \mathbb{S}^d$ form an open cover of $\mathbb{S}^d$.

Fix $\alpha  \in  \left]0,1\right[$ for the remaining of the section. Choose a Gevrey function $ \chi : \R \to \left[0,1\right]$ such that $\chi\p{x} = 1$ if $x \leq \frac{1}{2} $ and $\chi\p{x} = 0$ if $x \geq 1$. Define for all $n \geq 1$ and $\xi \in \R^{d+1}$, $\chi_n\p{\xi} = \chi\p{2^{-n} \va{\xi}}$ and $\chi_{\alpha,n}\p{\xi}= \chi\p{\va{\xi} - 2^{n^\alpha}}$, set also $\chi_n = \chi_{\alpha,n} = 0$ if $n\leq 0$. Then set for $n \in \N$, $\psi_n\p{\xi} = \chi_{n+1}\p{\xi} - \chi_n\p{\xi}$ and $\psi_{\alpha,n}\p{\xi} = \chi_{\alpha,n+1}\p{\xi} - \chi_{\alpha,n}\p{\xi}$. Thus we have for $n \geq 1$
\begin{equation*}
\textup{supp } \psi_n \subseteq \set{ \xi \in \R^{d+1} : 2^{n-1} \leq \va{\xi} \leq 2^{n+1} }
\end{equation*}
and
\begin{equation*}
\textup{supp } \psi_{\alpha,n} \subseteq \set{ \xi \in \R^{d+1} : 2^{n^\alpha} \leq \va{\xi} \leq 2^{\p{n+1}^{\alpha}}+ 1}.
\end{equation*}
In addition, $\textup{supp } \psi_0$ and $\textup{supp } \psi_{\alpha,0}$ are contained in $ \set{\xi \in \R^{d+1} : \va{\xi} \leq 5}$. Moreover, we have $\sum_{n \geq 0} \psi_n = \sum_{n \geq 0} \psi_{\alpha,n} = 1$. Set
\begin{equation*}
\Gamma = \N \times \set{0,\dots,r-1,f}.
\end{equation*}  
Define for $\p{n,i} \in \Gamma$ the function $\psi_{\Theta,n,i}$ by
\begin{equation*}
\psi_{\Theta,n,i}\p{\xi} = \begin{cases}
\psi_n\p{\xi} \varphi_i\p{\frac{\xi}{\va{\xi}}} & \textrm{ if } n \geq 1, \\
\frac{\psi_0\p{\xi}}{r-1} & \textrm{ if } n=0,
\end{cases} 
\end{equation*}
if $i \in \set{1,\dots,r-2,f}$, and by
\begin{equation*}
\psi_{\Theta,n,i}\p{\xi} =  \p{1 - \psi_0\p{\xi}} \psi_{\alpha,n}\p{\xi} \varphi_i\p{\frac{\xi}{\va{\xi}}}
\end{equation*}
if $i \in \set{0,r-1}$, so that we have
\begin{equation*}
\sum_{\p{n,i} \in \Gamma} \psi_{\Theta,n,i} = 1.
\end{equation*}

We will give a Sobolev-like definition of the local space $\h_{\Theta,\alpha}$ (Definition \ref{defh}) by mean of a weight $w_{\Theta,\alpha}$ (see \eqref{eqdefw}). If this description is convenient to prove the basic properties of $\h_{\Theta,\alpha}$ (see Proposition~\ref{propbaseh}), we will rather use in the following sections a Paley--Littlewood-like description of the space $\h_{\Theta,\alpha}$ (see Proposition~\ref{propequiv}), for any $\upsilon \in \left]1,\frac{1}{1-\alpha}\right[$ we have:
\begin{equation*}
\h_{\Theta,\alpha} = \set{u \in \p{\mathcal{S}^\upsilon}' : \sum_{\p{n,i} \in \Gamma} \p{2^{n \beta_i} \n{\psi_{\Theta,n,i}\p{D}u}_2}^2 < + \infty}
\end{equation*}
where
\begin{equation}\label{eqdefbeta1}
\beta_0 = d+ 2, \beta_{r-1} = - \p{d+2}, \beta_{f}= -\p{d+2}
\end{equation}
and
\begin{equation}\label{eqdefbeta2}
\beta_i = -\p{i+1} \p{d+2} \textup{ for } i \in \set{1,\dots,r-2}.
\end{equation}
The main idea behind the choice of the $\beta_i$ is that the expected regularity of elements of $\h_{\Theta,\alpha}$ (measured \emph{via} integrability of the Fourier transform) must decrease under the action of the linear model of the dynamics (the $\beta_i$ play the role here of an analogue of the escape function from \cite{FauSjo}). The particular choice has been made so that computations are as easy as possible. Our parameters have been designed in order to make the Paley--Littlewood description as simple as possible, at the cost of a definition of the weight $w_{\Theta,\alpha}$ that may seem a bit heavy. It is defined for $\xi \in \R^{d+1}$ by
\begin{equation}\label{eqdefw}
\begin{split}
 & w_{\Theta,\alpha}\p{\xi} = \psi_0\p{\xi} + \p{1 - \psi_0\p{\xi}}\left( \sum_{i \in \set{0,r-1}} \varphi_i\p{\frac{\xi}{\va{\xi}}} e^{\frac{\beta_i \ln\p{1 + \va{\xi}}^{\frac{1}{\alpha}}}{\p{\ln 2}^{\frac{1}{\alpha}-1}}} \right. \\ & \left. \qquad \qquad \qquad \qquad \qquad \qquad \qquad \qquad \qquad + \sum_{i \in \set{1,\dots,r-2,f}} \varphi_i\p{\frac{\xi}{\va{\xi}}} \left\langle \xi \right\rangle^{\beta_i}\right),
\end{split}
\end{equation}
where
\begin{equation*}
\left\langle\xi\right\rangle = \sqrt{1 + \va{\xi}^2} \textup{ for } \xi \in \R^{d+1}.
\end{equation*}

\begin{df}\label{defh}
Define the space (for any $\upsilon \in \left]1, \frac{1}{1-\alpha}\right[$)
\begin{equation*}
\h_{\Theta,\alpha} = \set{u \in \p{\mathcal{S}^\upsilon}' : \hat{u} \in L^2_{\textup{loc}} \textrm{ and } \s{\R^{d+1}}{\va{\hat{u}\p{\xi}}^2 w_{\Theta,\alpha}\p{\xi}^2 }{\xi} < + \infty}
\end{equation*}
endowed with the scalar product
\begin{equation*}
\langle u , v \rangle_{\Theta,\alpha} = \s{\R^d}{\overline{\hat{u}\p{\xi}} \hat{v}\p{\xi} w_{\Theta,\alpha}\p{\xi}^2}{\xi}.
\end{equation*}
Recall \eqref{eqdefsupsilon} for the definition of $\mathcal{S}^\upsilon$ and \eqref{eqdefw} for the definition of $w_{\Theta,\alpha}$.
\end{df}

\begin{prop}\label{propbaseh}
$\h_{\Theta,\alpha}$ is a separable Hilbert space that does not depend on the choice of $\upsilon$. For all $1 < \upsilon < \frac{1}{1-\alpha}$, the space $\mathcal{S}^\upsilon$ is continuously contained and dense in $\mathcal{H}_{\Theta,\alpha}$, and $\mathcal{H}_{\Theta,\alpha}$ is continuously contained in $\p{\mathcal{S}^\upsilon}'$.
\end{prop}

\begin{proof}
The map
\begin{equation*}
\begin{array}{ccccc}
A & : & \h_{\Theta,\alpha} & \to & L^2\p{\R^{d+1}} \\
 & & u & \mapsto & \hat{u} w_{\Theta,\alpha}
\end{array}
\end{equation*}
is clearly an isometry. Choose $\upsilon < \frac{1}{1-\alpha}$, thanks to Propositions \ref{propfourier} and~\ref{propdecfour} (recall \eqref{eqfourierupsilon}), and since $\frac{1}{\alpha} < \frac{\upsilon}{\upsilon-1}$, the map $u \mapsto \widehat{u} . w_{\Theta,\alpha}^{-1}$ is continuous from $\mathcal{S}^\upsilon$ to $L^2\p{\R^{d+1}}$. Thus the map $B : u \mapsto \mathbb{F}^{-1}\p{u w_{\Theta,\alpha}^{-1}}$ is continuous from $L^2\p{\R^{d+1}}$ to $\p{\mathcal{S}^{\upsilon}}'$. But if $u \in L^2\p{\R^{d+1}}$ then it is clear that $B u \in \h_{\Theta,\alpha}$ with $\n{Bu}_{\Theta,\alpha} = \n{u}_2$. Now, since $A$ and $B$ are inverses of each other, $\h_{\Theta,\alpha}$ is isometric to $L^2\p{\R^{d+1}}$ and thus a separable Hilbert space.

Proposition~\ref{propdecfour} implies that $\mathcal{S}^\upsilon$ is continuously contained in $\h_{\Theta,\alpha}$ and that the inclusion of $\h_{\Theta,\alpha}$ in $\p{\mathcal{S}^\upsilon}'$ is continuous. Let $u \in \h_{\Theta,\alpha}$ be in the orthogonal space to $\mathcal{S}^\upsilon$. If $\rho$ is a compactly supported element of $\mathcal{S}^\upsilon$, then, for all $v \in \mathcal{S}^\upsilon$, we have
\begin{equation*}
\int_{\R^{d+1}} \rho\p{\xi} \overline{\hat{u}\p{\xi}} w_{\Theta,\alpha}\p{\xi}^2 v\p{\xi} \mathrm{d}\xi = \left\langle u, \mathbb{F}^{-1}\p{\rho.v} \right\rangle_{\Theta,\alpha} = 0.
\end{equation*}
Thus the function $\rho \bar{\hat{u}} w_{\Theta,\alpha}^2 \in L^1\p{\R^{d+1}}$ vanishes (take for $v$ a convolution kernel), and so does $u$. Consequently, $\mathcal{S}^\upsilon$ is dense in $\h_{\Theta,\alpha}$.

To see that $\h_{\Theta,\alpha}$ does not depend on the choice of $\upsilon$, just notice that, if we use $\tilde{\upsilon} \in \left]\upsilon,\frac{1}{1 - \alpha}\right[$ instead of $\upsilon$ in the definition of $\h_{\Theta,\alpha}$, then we obtain another Hilbert space $\widetilde{\h}_{\Theta,\alpha}$. But then $\widetilde{\h}_{\Theta,\alpha} \subseteq \h_{\Theta,\alpha} $, and the inclusion is isometric and has a dense image (because $\widetilde{\h}_{\Theta,\alpha}$ contains $\mathcal{S}^\upsilon$). Since $\widetilde{\h}_{\Theta,\alpha}$ and $\h_{\Theta,\alpha}$ are both Hilbert spaces, they must coincide.
\end{proof}

\begin{rmq}\label{rmqdense}
It is clear from the proof that in fact the elements of $\mathcal{S}^\upsilon$ whose Fourier transform is compactly supported form a dense subset of $\h_{\Theta,\alpha}$.
\end{rmq}

\begin{prop}\label{propequiv}
Let $1 < \upsilon < \frac{1}{1-\alpha}$ and $u \in \p{\mathcal{S}^\upsilon}'$. Then $u \in \h_{\Theta,\alpha}$ if and only if
\begin{equation}\label{eqequiv}
\sum_{\p{n,i} \in \Gamma} \p{2^{n \beta_i } \n{\psi_{\Theta,n,i}\p{D}u}_2}^2 < + \infty.
\end{equation}
Moreover, the square root of this quantity defines an equivalent (Hilbertian) norm on $\mathcal{H}_{\Theta,\alpha}$.
\end{prop}

\begin{proof}
First, notice that there is $C >0$ such that, if $n \in \N$, $i \in \set{1,\dots,r-2,f}$ and $\xi \in \textup{supp } \psi_{\Theta,n,i}$, then
\begin{equation*}
\frac{1}{C} 2^{n \beta_i} \leq \left\langle \xi \right\rangle^{\beta_i} \leq C 2^{n \beta_i}.
\end{equation*}
Up to enlarging $C$, it is also true that if $n \in \N$, $i \in \set{0,r-1}$ and $\xi \in \textup{supp } \psi_{\Theta,n,i}$ then
\begin{equation*}
\frac{1}{C} 2^{n \beta_i} \leq e^{\frac{\beta_i \ln\p{1 + \va{\xi}}^{\frac{1}{\alpha}}}{\ln 2^{\frac{1}{\alpha}-1}}} \leq C 2^{n \beta_i}.
\end{equation*}
Now, using the fact that the intersection number of the support of the $\psi_{\Theta,n,i}$ for $\p{n,i} \in \Gamma$ is finite, we find another constant $C >0$ such that for all $\xi \in \R^{d+1}$ we have
\begin{equation}\label{eqcont}
\frac{1}{C} w_{\Theta,\alpha}\p{\xi}^2 \leq \sum_{\p{n,i} \in \Gamma} \p{2^{n \beta_i}\psi_{\Theta,n,i}\p{\xi} }^2 \leq C w_{\Theta,\alpha}\p{\xi}^2. 
\end{equation}

From this, we get immediately that if $u \in \h_{\Theta,\alpha}$ then \eqref{eqequiv} holds. Reciprocally, if \eqref{eqequiv} holds, then $\hat{u}$ is in $L_{\textup{loc}}^2$ (the sum $\sum_{\p{n,i} \in \Gamma} \psi_{\Theta,n,i}$ is locally finite) and from \eqref{eqcont} we get that $u \in \h_{\Theta,\alpha}$. The equivalence of norms is an immediate consequence of \eqref{eqcont}.

\end{proof}

Proposition~\ref{propequiv} suggests to define the auxiliary Hilbert space
\begin{equation}\label{eqdefB}
\begin{split}
\B & = \left\{ \p{u_{n,i}}_{\p{n,i} \in \Gamma} \in \prod_{\p{n,i} \in \Gamma}L^2\p{\R^{d+1}} : \right. \\ & \left. \quad  \qquad  \qquad  \sum_{\p{n,i} \in \Gamma} \p{2^{n \beta_i} \n{ u_{n,i} }_2}^2 < + \infty \right\}.
\end{split}
\end{equation}
Define the map
\begin{equation}\label{eqdefQTheta}
\begin{array}{ccccc}
\mathcal{Q}_\Theta & : & \h_{\Theta,\alpha} & \to & \B \\
 & & u & \mapsto & \p{\psi_{\Theta,n,i}\p{D}u}_{\p{n,i}\in \Gamma}
 \end{array}.
\end{equation}
For $\p{n,i} \in \Gamma$ define also the natural projection and inclusion
\begin{equation*}
\begin{array}{ccccc}
\pi_{n,i} & : & \B & \to & L^2\p{\R^{d+1}} \\
 & & \p{u_{\ell,j}}_{\p{\ell,j} \in \Gamma} & \mapsto & u_{n,i}
\end{array}
\end{equation*}
and
\begin{equation*}
\begin{array}{ccccc}
\iota_{n,i} & : & L^2\p{\R^{d+1}} & \to & \B \\
 & & u & \mapsto & \p{u \delta_{\p{n,i} = \p{\ell,j}}}_{\p{\ell,j} \in \Gamma}.
\end{array}
\end{equation*}

\section{Local transfer operator}\label{sectlto}

We are now going to study a local model for the Koopman operator \eqref{eqkoopman} associated to an Anosov flow $\p{\phi^t}_{t \in \R}$ on a $\p{d+1}$-dimensional manifold $M$. The main result of this section is Proposition~\ref{proplocop} which is a local version of Theorem~\ref{thmmain}.

As a local model for a flow, we will consider a family $\p{\mathcal{T}_t}_{t \in \R}$ of diffeomorphisms of $\R^{d+1}$ such that if we define $F : \R^d \to \R^{d+1}$ by $x \mapsto \mathcal{T}_{0}\p{x,0}$ (here we make the identification $\R^{d+1} \simeq \R^d \times \R$) then we have
\begin{equation}\label{eqdeffamily}
\forall t \in \R : \forall \p{x,y} \in  \R^d  \times \R \simeq \R^{d+1} : \mathcal{T}_t\p{x,y} = F\p{x} + y  e_{d+1}  + te_{d+1}.
\end{equation}
We will say that $F$ is the map associated to the family of diffeomorphisms $\p{\mathcal{T}_t}_{t \in \R}$. Reciprocally, if $F : \R^{d} \to \R^{d+1}$ is an immersion, we define by \eqref{eqdeffamily} the associated family of diffeomorphisms $\p{\mathcal{T}_t}_{t \in \R}$ (provided they actually are diffeomorphisms). 

\begin{rmq}\label{rmqdefTt}
Let us explain why we use such a family of diffeomorphisms as a local model for a flow. We want to study the flow $\p{\phi^t}_{t \in \R}$ in the neighbourhood of a fixed time $\tilde{t}_0$. To do it, we take charts $\kappa$ and $\kappa'$ for $M$ and we study the family of diffeomorphisms $\p{\mathcal{T}_t}_{t \in \R}$ defined by the formula 
\begin{equation*}
\mathcal{T}_t = \kappa \circ \phi^{\tilde{t}_0 + t} \circ \kappa'^{-1}.
\end{equation*}
Of course, this is not in general a family of diffeomorphisms from $\R^{d+1}$ to itself (\emph{a priori} the domain of $\mathcal{T}_t$ depends on $t$). However, it is more convenient to deal with diffeomorphisms of the whole $\R^{d+1}$, and we will consequently provide extensions of the $\mathcal{T}_t$ to $\R^{d+1}$ when applying Proposition~\ref{proplocop} in \S \ref{gs1} (see Lemma~\ref{lmdecoupe}). These extensions are far from canonical, but the use of a cutoff function will ensure that none of the objects that we consider in \S \ref{gs1} depend on the choices we will make in a relevant way.

It is natural to ask for $\kappa$ and $\kappa'$ to be flow boxes, that is, if $V$ is the generator of the flow $\p{\phi^t}_{t \in \R}$, we require $\kappa^*\p{e_{d+1}} = V$ and $\kappa'^*\p{e_{d+1}} = V$ (we identify $e_{d+1}$ with the constant vector field with value $e_{d+1}$). This requirement implies \eqref{eqdeffamily} for small $t$ and $y$, and, since we are only interested here in the behaviour of $\p{\phi_t}_{t \in \R}$ locally in both space and time, we may modify the definition of $\mathcal{T}_t$ for large $t$ and design our extension to ensure that \eqref{eqdeffamily} holds (we refer to the proof of Lemma~\ref{lmdecoupe} for details). Once again, this will be of no harm in the global analysis thanks to the use of cutoff function in both time and space.
\end{rmq}
 
In this section, we will study such a family with no reference to a particular Anosov flow. We will need further assumptions to do so. The first one is that $F$ (or equivalently $\mathcal{T}_0$ or any $\mathcal{T}_t$ for $t \in \R$) is $\mathcal{C}^{\kappa,\upsilon}$ for some $\kappa > 0$ and $\upsilon > 1$. The second one is a condition of hyperbolicity that we will express using cones.

Let $r \geq 2$ be an integer and choose two systems of $r+2$ cones (with respect to the direction $e_{d+1}$ as usual) $\Theta = \p{C_0,\dots, C_r,C_f}$ and $\Theta' = \p{C_0',\dots,C_r',C_f'}$. We assume that $\p{\mathcal{T}_t}_{t \in \R}$ is cone-hyperbolic from $\Theta'$ to $\Theta$ in the following sense:
\begin{enumerate}[label = (\roman*)]
\item for all $x \in \R^{d+1}$, $i \in \set{1,\dots,r}$ and $t \in \R$ we have\footnote{Here, $A^{\textup{tr}}$ denotes the transpose of $A$.}
\begin{equation*}
D_x \mathcal{T}_t^{\textup{tr}}\p{C_i} \subseteq C'_{\min\p{i+2,r}};
\end{equation*}
\item for all $x \in \R^{d+1}$ and $t \in \R$ we have 
\begin{equation*}
D_x \mathcal{T}_t^{\textup{tr}}\p{C_f} \cap C'_0 = \set{0};
\end{equation*}
\item there is $\Lambda > 1$ such that for all $x \in \R^{d+1}$, all $\xi \in C_{r-1}$, and all $t \in \R$ we have
\begin{equation*}
\va{D_x \mathcal{T}_t^{\textup{tr}}\p{\xi}} \geq \Lambda \va{\xi};
\end{equation*}
\item for the same $\Lambda >1$, for all $x \in \R^{d+1}$, all $\xi \in \R^d$, and all $t \in \R$ such that $D_x \mathcal{T}_t^{\textup{tr}} \p{\xi} \in C_0'$  we have\footnote{Notice that the condition $D_x \mathcal{T}_t^{\textup{tr}} \p{\xi} \in C_0'$ implies in particular that $\xi \in C_0$, as a consequence of (i) and (ii).}
\begin{equation*}
\va{D_x \mathcal{T}_t^{\textup{tr}} \p{\xi}} \leq \Lambda^{-1} \va{\xi}.
\end{equation*}
\end{enumerate}

\begin{rmq}
Notice that the definition of the cone-hyperbolicity of the family $\p{\mathcal{T}_t}_{t \in \R}$ only involves the derivatives $D_x \mathcal{T}_t^{\textup{tr}}$. However, these derivatives do not depend on $t$ (this is a consequence of \eqref{eqdeffamily}). Consequently, one only needs to check that (i)-(iv) hold for $t=0$. This fact may be surprising since hyperbolicity is usually a phenomenon that can only be observed after a small amount of time, but recall Remark \ref{rmqdefTt}: in the application, the family $\p{\mathcal{T}_t}_{t \in \R}$ will only be used to describe the flow $\p{\phi_t}_{t \in \R}$ near some time $\tilde{t}_0$. Then, provided that $\tilde{t}_0 > 0$, the family $\p{\mathcal{T}_t}_{t \in \R}$ will be cone-hyperbolic (see \S \ref{gs1} for the details).
\end{rmq}

\begin{rmq}
Notice that if $\p{\mathcal{T}_t^1}_{t \in \R}$ and $\p{\mathcal{T}_t^2}_{t \in \R}$ are two families of diffeomorphisms as above, then their composition may naturally be defined as $\p{\mathcal{T}_t^1 \circ \mathcal{T}_0^2}_{t \in \R}$. Moreover, if there are systems of cones $\Theta,\Theta'$ and $\Theta''$ such that $\p{\mathcal{T}_t^1}_{t \in \R}$ is cone-hyperbolic from $\Theta'$ to $\Theta''$ and $\p{\mathcal{T}_t^2}_{t \in \R}$ is cone-hyperbolic from $\Theta$ to $\Theta'$ then $\p{\mathcal{T}_t^1 \circ \mathcal{T}_0^2}_{t \in \R}$ is cone-hyperbolic from $\Theta$ to $\Theta''$.
\end{rmq}

We will also consider a $\mathcal{C}^{\infty}$ family $\p{G_t}_{t \in \R}$ of $\mathcal{S}^\upsilon$ functions from $\R^{d+1}$ to $\C$, such that there is a compact subset $K$ of $\R^{d+1}$ such that, if $x \in \R^{d+1} \setminus K$ and $t \in \R$, then $G_t\p{x} = 0$.

We will study the family $\p{\mathcal{L}_t}_{t \in \R}$ of local transfer operator defined by
\begin{equation*}
\mathcal{L}_t u = G_t \p{u \circ \mathcal{T}_t}.
\end{equation*}
This definition makes sense for $u \in \mathcal{S}^{\tilde{\upsilon}}$ (for any $\tilde{\upsilon} > \upsilon$) and may be extended by duality to $u \in \p{\mathcal{S}^{\tilde{\upsilon}}}'$.

The main result of this section is Proposition~\ref{proplocop}, which can be seen as a local version of Theorem~\ref{thmmain}.

\begin{prop}\label{proplocop}
Let $\alpha \in \left]\frac{\upsilon - 1}{\upsilon}, 1 \right[$. For all $t \in \R$ the transfer operator $\mathcal{L}_t$ is bounded from $\h_{\Theta,\alpha}$ to $\h_{\Theta',\alpha}$. Moreover, the family $\p{\mathcal{L}_t}_{t \in \R}$ is strongly continuous (as a family of operators from $\h_{\Theta,\alpha}$ to $\h_{\Theta',\alpha}$), hence it is measurable. 

Moreover, if $\alpha < \frac{1}{2}$, if $k$ is a non-negative integer and if $h : \R \to \C$ is a compactly supported $k$th time differentiable function whose $k$th derivative has bounded variation then the operator
\begin{equation}\label{eqopschatt}
\s{\R}{h\p{t} \mathcal{L}_t}{t} : \h_{\Theta,\alpha} \to \h_{\Theta',\alpha}
\end{equation}
is in the Schatten class\footnote{See \cite[Chapter IV.11]{Gohb} for the definition and basic properties of Schatten classes} $S_{q}$ for all $q \geq 1$ such that $q > \frac{d+1}{k+1}$. Moreover, there is a constant $C >0$, which depends on $h$ only through its support, such that
\begin{equation*}
\n{\s{\R}{h\p{t} \mathcal{L}_t}{t}}_{S_q} \leq C\p{\n{h}_{\mathcal{C}^{k-1}} + \n{h^{\p{k}}}_{\textup{BV}}},
\end{equation*}
where $\n{\cdot}_{S_q}$ denotes the $S_q$ Schatten class norm and $\n{\cdot}_{\textup{BV}}$ the bounded variation norm.

If $k+1 > d+1$ and $\Theta = \Theta'$ we have
\begin{equation*}
\textup{tr}\p{\s{\R}{h\p{t}\mathcal{L}_t}{t}} = \sum_{p \circ F\p{x} = x} \frac{h\p{T\p{x}}}{\va{\det \p{I - p \circ D_x F}}}\int_{\R} G_{T\p{x}}\p{x,y} \mathrm{d}y ,
\end{equation*}
where $p$ is the orthogonal projection from $\R^{d+1}$ to $\R^d \simeq \R^d \times \set{0}$ and, for $x \in \R^d$, the number $T(x)$ is defined by $F(x) = p\p{F(x)} + \p{0, - T(x)}$. 

Without the hypothesis $\alpha < \frac{1}{2}$, it remains true that the operator \eqref{eqopschatt} is compact.
\end{prop}

\begin{rmq}
Since $\alpha > \frac{\upsilon - 1}{\upsilon}$, we may choose $\tilde{\upsilon} > \upsilon$ such that $\tilde{\upsilon} < \frac{1}{1 - \alpha}$. Then $\h_{\Theta,\alpha} \subseteq \p{\mathcal{S}^{\tilde{\upsilon}}}'$ and thus $\mathcal{L}_t u$ is well-defined as an element of $\p{\mathcal{S}^{\tilde{\upsilon}}}'$ when $t \in \R$ and $u \in \h_{\Theta,\alpha}$.
\end{rmq}

\begin{rmq}
Notice that the spaces $\h_{\Theta,\alpha}$ and $\h_{\Theta',\alpha}$ depend \textit{a priori} not only on $\Theta$ (and $\alpha$) but also on the choice of partitions of unity $\p{\varphi_0,\dots,\varphi_{r-1},\varphi_f}$ and $\p{\varphi_0',\dots,\varphi_{r-1}',\varphi_f'}$ on $\mathbb{S}^d$ as in \S \ref{sectls}. However, in view of Proposition~\ref{proplocop}, this choice is mostly irrelevant and the dependence on $\Theta$ and $\Theta'$ is the fundamental point.
\end{rmq}

The remainder of this section is devoted to the proof of Proposition~\ref{proplocop}. For this, we introduce in Lemma~\ref{lmconvstrong} a family of auxiliary operators $\p{\mathcal{M}_t}_{t \in \R}$ acting on the space $\mathcal{B}$ defined in \eqref{eqdefB}. Then, we prove that the family $\p{\mathcal{M}_t}_{t \in \R}$ has the properties that we expect from $\p{\mathcal{L}_t}_{t \in \R}$ : boundedness and strong continuity are proven in Lemma~\ref{lmconvstrong} (with the help of the preparatory Lemmas \ref{lmdist} and~\ref{lmmajipp}, see \S \ref{subsecauxiliary}), that an operator similar to \eqref{eqopschatt} is in a Schatten class is proven in Lemma~\ref{lmschatten} (with the help of Lemmas \ref{lmtc}, \ref{lmtcnl} and~\ref{lmippf}, see \S \ref{subsecschatten}) and the formula for the trace is given in Lemma \ref{lmtrace} (see \S \ref{subsectrace}). Finally, we end the proof of Proposition~\ref{proplocop} by showing that $\p{\mathcal{L}_t}_{t \in \R}$ inherits these properties from $\p{\mathcal{M}_t}_{t \in \R}$.

\subsection{The auxiliary operators $\mathcal{M}_t$.}\label{subsecauxiliary}

We will need smooth functions $\tilde{\varphi}_0,\dots,\tilde{\varphi}_{r-1}$, and $\tilde{\varphi}_f : \mathbb{S}^d \to \left[0,1\right]$ such that
\begin{itemize}
\item if $i \in \set{0,\dots,r-1,f}$ then $\tilde{\varphi}_i$ is supported in the interior of $C_i \cap \mathbb{S}^d$;
\item if $i \in \set{1,\dots,r-2}$ then $\tilde{\varphi}_i$ vanishes on a neighborhood of $C_{i+2} \cap \mathbb{S}^d$;
\item if $i \in \set{0, \dots,r-1,f}$, $ x \in \mathbb{S}^d$, and $\varphi_i\p{x} \neq 0$ then $\tilde{\varphi}_i\p{x} =1$.
\end{itemize}
Define then $\tilde{\psi}_n = \chi_{n+2} - \chi_{n-1}$ and $\tilde{\psi}_{\alpha,n} = \chi_{\alpha,n+b} - \chi_{\alpha,n-b}$ for $n \geq 0$, where $b$ is chosen large enough so that for all $n \in \N^*$ we have
\begin{equation*}
2^{\p{n+1}^\alpha} - 2^{\p{n+b}^{\alpha}} + 1 \leq \frac{1}{2} \textup{ and } 2^{n^\alpha}- 2^{\p{n-b}^{\alpha}} \geq 1.
\end{equation*}
If $\p{n,i} \in \Gamma$ set
\begin{equation*}
\tilde{\psi}_{\Theta,n,i}\p{\xi} = \begin{cases} 
\tilde{\psi}_n\p{\xi} \tilde{\varphi}_i\p{\frac{\xi}{\va{\xi}}} & \textrm{ if } n \geq 1, \\
\tilde{\psi_0}\p{\xi} & \textrm{ if } n=0,
\end{cases}
\end{equation*}
if $i \in \set{1,\dots,r-2,f}$, and 
\begin{equation*}
\tilde{\psi}_{\Theta,n,i}\p{\xi} = \begin{cases}
\tilde{\psi}_{\alpha,n}\p{\xi} \tilde{\varphi}_i\p{\frac{\xi}{\va{\xi}}} & \textrm{ if } n \geq 1, \\
\tilde{\psi}_{\alpha,0}\p{\xi} & \textrm{ if } n=0,
\end{cases}
\end{equation*}
if $i \in \set{0,r-1}$. Thus $\psi_{\Theta,n,i}\p{\xi} \neq 0$ implies $\tilde{\psi}_{\Theta,n,i}\p{\xi} = 1$. Now if $\p{n,i},\p{\ell,j} \in \Gamma$, and $t \in \R$ define an operator $S_{t,n,i}^{\ell,j}\ : L^2\p{\R^{d+1}} \to L^2\p{\R^{d+1}}$ by
\begin{equation}\label{eqdefentree}
S_{t,n,i}^{\ell,j} = \psi_{\Theta',n,i}\p{D} \circ \mathcal{L}_t \circ \tilde{\psi}_{\Theta,\ell,j}\p{D}.
\end{equation}

As announced above, we define in Lemma~\ref{lmconvstrong} a family of auxiliary operators whose study will take most of the remainder of this section.

\begin{lm}\label{lmconvstrong}
For all $t \in \R$ the sum 
\begin{equation}\label{eqsomme}
\sum_{\p{n,i},\p{\ell,j} \in \Gamma} \iota_{n,i} \circ S_{t,n,i}^{\ell,j} \circ \pi_{\ell,j}
\end{equation}
converges in the strong operator topology to an operator $\mathcal{M}_t : \B \to \B$ that depends continuously on $t$ in the strong operator topology.
\end{lm}

The proof of Lemma~\ref{lmconvstrong} is based on Lemmas \ref{lmdist} and~\ref{lmmajipp}. In order to prove Lemma~\ref{lmconvstrong}, we first define a relation $\hra$ on $\Gamma$ that indexes the transitions (in the frequency space) that would occur for a linear dynamics, in the spirit of \cite{phdadam,adamarxiv}. Our local space has been designed so that it corresponds either to a transition from high regularity to low regularity (which makes this part of the action smoothing) or to a stationary frequency in the direction of the flow (we will integrate in this direction, so that it also corresponds to a smoothing operator). The other transitions do not happen in the linear case, and we will control this non-linearity using not only the hyperbolicity of the dynamics but also its high regularity. Choose $a >0$ such that for all $x \in K$ and $t \in \R$ we have
\begin{equation*}
a < \n{ \p{D_x \mathcal{T}_t^{tr}}^{-1}}^{-1}.
\end{equation*}
Choose also $\nu$ such that $ 0 < \nu < \frac{\log_2 \Lambda}{\alpha}$. We define now the relation $\hra$. For $\p{\ell,j},\p{n,i} \in \Gamma$, we say that $\p{\ell,j} \hra \p{n,i}$ holds if either of the following conditions is satisfied:
\begin{itemize}
\item $i = j = 0$ and $\ell \geq n + \nu n^{1-\alpha}$;
\item $i = j = r-1$ and $ n \geq \ell + \nu \ell^{1-\alpha}$;
\item $j=0$ and $i \in \set{1,\dots,r-1,f}$;
\item $j \in \set{1,\dots,r-2,f},i = r-1$ and $\ell \leq n^\alpha + 4 - \log_2 a$;

\item $j=f,i \in \set{1,\dots,r-2}$ and $n \geq \ell -4 + \log_2 a$;
\item $i,j \in \set{1,\dots,r-2}$ with $i \geq j+1$ and $n \geq \ell -4 + \log_2 a$;

\item $i = j = f$ and $\va{\ell-n} \leq 10 - \log_2 c$, where $c$ is such that for all $\xi = \p{\xi_1,\dots,\xi_{d+1}} \in C_f \cup C_f'$ we have $\va{\xi_{d+1}} \geq c \va{\xi}$ (such a constant exists by our definition of a system of cones).
\end{itemize}
In all other cases, we say that $\p{\ell,j} \nhra \p{n,i}$. Let us list the cases in which $\p{\ell,j} \nhra \p{n,i}$ in prevision of the proof of Lemma \ref{lmdist}:
\begin{itemize}
\item $i=j=0$ and $\ell < n + \nu n^{1 - \alpha}$;
\item $i=j=r$ and $n < \ell + \nu \ell^{1-\alpha}$;
\item $i = r-1, j \in \set{1,\dots,r-2,f}$ and $\ell > n^{\alpha}+4-\log_2 a$;
\item $i \in \set{1,\dots,r-2}, j =f$ and $n < \ell-4 + \log_2 a$;
\item $i,j \in \set{1,\dots,r-2}, i \geq j+1$ and $n < \ell - 4 + \log_2 a$;
\item $i=j=f$ and $\va{\ell-n} > 10 - \log_2 c$; 
\item $j=f$ and $i =0$;
\item $j = r-1$ and $i \neq r-1$;
\item $j \in \set{1,\dots,r-1}$ and $i \in \set{0,\dots,j,f}$.
\end{itemize}

Lemma~\ref{lmdist} is the main tool to use the hyperbolicity of the dynamics to rule out the transitions of frequencies that do not occur in the linear picture. 

\begin{lm}\label{lmdist}
For $ i \in \set{0,\dots,r-1,f}$, set  $\alpha_i=\alpha$ if $i=0$ or $r-1$, and $\alpha_i=1$ otherwise.There are $c' >0$ and $N>0$ such that if $\p{\ell,j},\p{n,i} \in \Gamma$ we have: $\p{\ell,j} \hra \p{n,i}$ or $\max\p{n,\ell} \leq N$ or, for all $x \in K$ and $t \in \R$,
\begin{equation*}
d\p{\textup{supp } \psi_{\Theta',n,i} , D_x\mathcal{T}_t^{\textup{tr}}\p{\textup{supp } \tilde{\psi}_{\Theta,\ell,j}} } \geq c' \max\p{2^{n^{\alpha_i}},2^{\ell^{\alpha_j}}}.
\end{equation*}
\end{lm}

\begin{proof}

We will make massive use of the following fact in this proof : if $C_+$ and $C_-$ are two closed cones in $\R^{d+1}$ such that $C_+ \cap C_- = \set{0}$ (we say that such cones are transverse) then for all $\xi \in C_+$ and $\eta \in C_-$ we have
\begin{equation}\label{eqdistcone}
d\p{\xi,\eta} \geq \mu \max\p{\va{\xi},\va{\eta}}
\end{equation}
where $\mu = \min\p{d\p{C_+\cap \mathbb{S}^d,C_-},d\p{C_- \cap \mathbb{S}^d,C_+}} > 0$.

Assume that $\p{n,i},\p{\ell,j} \in \Gamma$ are such that $\p{\ell,j} \nhra \p{n,i}$ and $\max\p{n,l} > N$ for some large $N$, and take $\xi \in \textup{ supp }\psi_{\Theta',n,i}, \eta \in \textup{ supp } \tilde{\psi}_{\Theta,\ell,j}$ and $t \in \R$. We go through the different cases in which $\p{\ell,j} \nhra \p{n,i}$ as listed above.
\begin{itemize}
\item If $i=j=0$ and $\ell < n + \nu n^{1-\alpha}$, there are two possibilities: either $ D_x \mathcal{T}_t^{\textup{tr}} \p{\eta}  \notin C_0'$, and we can conclude with \eqref{eqdistcone} (since $\varphi'_0$ is supported in the interior of $C_0'$), or $ D_x \mathcal{T}_t^{\textup{tr}} \p{\eta} \in C_0'$, and by cone-hyperbolicity we have
\begin{equation*}
\begin{split}
\va{\xi} - \va{ D_x \mathcal{T}_t^{\textup{tr}} \p{\eta}} & \geq 2^{n^\alpha} - \Lambda^{-1}\p{2^{\p{\ell+b}^\alpha} + 1}\\ & \geq 2^{n^\alpha} - \Lambda^{-1} \p{2^{\p{n + \nu n^{1 - \alpha} + b}^\alpha} + 1}\\
& \geq 2^{n^\alpha} \p{1 - 2^{\p{n + \nu n^{1-\alpha} + b}^\alpha - \log_2  \Lambda - n^\alpha} - 2^{-n^\alpha}}.
\end{split}
\end{equation*}
We can then conclude if $N$ is large enough, since
\begin{equation*}
\p{n + \nu n^{1 - \alpha} + b}^\alpha - \log_2 \Lambda - n^\alpha \underset{n \to + \infty}{\to} \alpha \nu - \log_2 \Lambda < 0
\end{equation*}
and
\begin{equation}\label{eqbiz}
2^{\p{n + \nu n^{1- \alpha}}^\alpha} \leq C 2^{n^\alpha},
\end{equation}
for some constant $C >0$ that does not depend on $n$ (we used here that $\p{n + \nu n^{1-\alpha}}^\alpha \underset{n \to + \infty}{=} n^\alpha + \alpha \nu + o\p{1}$). 
\item If $i=j=r-1$ and $n < \ell + \nu \ell^{1 - \alpha}$ then
\begin{equation*}
\begin{split}
\va{ D_x \mathcal{T}_t^{\textup{tr}} \p{\eta}} - \va{\xi} & \geq \Lambda 2^{\p{\ell-b}^\alpha} - \p{2^{\p{n+1}^\alpha} + 1} \\
    & \geq \Lambda 2^{\p{\ell - b}^\alpha} - \p{2^{\p{\ell + \nu \ell^{1 - \alpha} + 1}^\alpha} + 1} \\
    & \geq 2^{\ell^\alpha} \p{\Lambda 2^{\p{\ell-b}^\alpha - \ell^\alpha} - 2^{\p{\ell + \nu \ell^{1 - \alpha} + 1}^\alpha - \ell^\alpha} - 2^{-\ell^\alpha}}.
\end{split}
\end{equation*}
We can conclude if $N$ is large enough, since
\begin{equation*}
\Lambda 2^{\p{\ell-b}^\alpha - \ell^\alpha} - 2^{\p{\ell + \nu \ell^{1 - \alpha} + 1}^\alpha - \ell^\alpha} - 2^{-\ell^\alpha} \underset{\ell \to + \infty}{\to} \Lambda - 2^{\alpha \nu} > 0,
\end{equation*}
and \eqref{eqbiz} still holds when $n$ is replaced by $\ell$.
\item If $j \in \set{1,\dots,r-2,f}, i=r-1$ and $\ell > n^\alpha + 4 - \log_2 a$, then
\begin{equation*}
\begin{split}
\va{ D_x \mathcal{T}_t^{\textup{tr}}\p{\eta}} - \va{\xi} & \geq a 2^{\ell-2} - \p{2^{\p{n+1}^\alpha} + 1} \\
     & \geq a 2^{\ell-2} - 2^{n^\alpha + 1} - 1 \\
     & \geq a2^{\ell - 3} -1.
\end{split} 
\end{equation*}
\item If $i \in \set{1,\dots,r-2},j=f$ and $ n < \ell - 4 + \log_2 a$ then
\begin{equation*}
\begin{split}
\va{ D_x \mathcal{T}_t^{\textup{tr}}\p{\eta}} - \va{\xi} & \geq a 2^{\ell-2} - 2^{n+1} \\
                & \geq a 2^{\ell - 3}.
\end{split}
\end{equation*}

\item The case $i,j \in \set{1,\dots,r-2}, i \geq j+1$ and $n < \ell - 4 + \log_2 a$ is dealt as the previous one.

\item If $i=j=f$ and $\va{\ell-n} > 10 - \log_2 c$, then just notice that the $d+1$th coordinate of $ D_x \mathcal{T}_t^{\textup{tr}} \p{\eta} - \xi$ is $\eta_{d+1} - \xi_{d+1}$ and consequently
\begin{equation*}
\va{ D_x \mathcal{T}_t^{\textup{tr}}\p{\eta} - \xi} \geq \va{\eta_{d+1} - \xi_{d+1}}.
\end{equation*} 
Since in addition we have $\va{\xi_{d+1}} \geq c \va{\xi}$ and $\va{\eta_{d+1}} \geq c \va{\eta}$, we can conclude in this case (discussing whether $\va{\xi}$ or $\va{\eta}$ is larger).
\item The three last cases are dealt with by cone hyperbolicity using \eqref{eqdistcone} (the support of $\psi_{\Theta',n,i}$ and the image of the support of $\tilde{\psi}_{\Theta,\ell,j}$ by $ D_x\mathcal{T}_t^{\textup{tr}}$ are contained in transverse cones).
\end{itemize}
\end{proof}

We now use Lemma~\ref{lmdist} to control the transitions that do not happen in the linear picture.

\begin{lm}\label{lmmajipp}
There is $\delta > 1$ such that, for every bounded interval $I$ of $\R$, there is $C >0$ such that if $\p{\ell,j} \nhra \p{n,i} $ for $\p{n,i},\p{\ell,j} \in \Gamma$, then for all $t \in I$ we have, recalling \eqref{eqdefentree},
\begin{equation*}
\n{S_{t,n,i}^{\ell,j}}_{L^2 \to L^2} \leq C \exp\p{- \frac{\max\p{n,\ell}^\delta}{C}}.
\end{equation*} 
\end{lm}

\begin{proof}
First of all, notice that $\mathcal{L}_t$ is bounded from $L^2$ to $L^2$ (uniformly when $t \in I$) and, since for all $\p{n,i},\p{\ell,j} \in \Gamma$ and $t \in I$, we have
\begin{equation*}
\n{S_{t,n,i}^{\ell,j}}_{L^2 \to L^2} \leq \n{\mathcal{L}_t}_{L^2 \to L^2},
\end{equation*}
the case of $\max\p{n,\ell} \leq N$ is dealt with by taking $C$ large enough.

Now take $\p{n,i},\p{\ell,j} \in \Gamma$ and $t \in I$ such that $\p{\ell,j} \nhra \p{n,i}$ and $\max\p{n,\ell} > N$. If $u \in L^2\p{\R^{d+1}}$ then we have, using Plancherel's formula,
\begin{equation}\label{eqnorme2}
\begin{split}
& (2 \pi)^{2(d+1)}\n{S_{t,n,i}^{\ell,j} u}_2^2 \\ & \quad = \int_{\R^{d+1}} \psi_{\Theta',n,i}\p{\xi}^2 \Bigg|\int_{\p{\R^{d+1}}^2} e^{-i x \xi} e^{i \mathcal{T}_t\p{x} \eta} \tilde{\psi}_{\Theta,\ell,j}\p{\eta} G_t\p{x} \hat{u}\p{\eta} \mathrm{d}x \mathrm{d}\eta \Bigg|^2 \mathrm{d}\xi.
\end{split}
\end{equation}
We are going to bound the inner integral. To do so, define for all $x \in \R^{d+1}$ and $j \in \set{1,\dots,d+1}$ the linear form $l_j\p{x}$ on $\R^{d+1} \times \R^{d+1}$ by $l_j\p{x}\p{\xi,\eta} = i\p{\partial_j \mathcal{T}_t\p{x} \eta - \xi_j}$. Define also for all $x \in \R^{d+1}$ the quadratic form $\Phi\p{x}$ on $\R^{d+1} \times \R^{d+1}$ by $\Phi\p{x}\p{\xi,\eta} = \va{ D_x \mathcal{T}_t^{\textup{tr}}\p{\eta} - \xi}^2$. Now for all $t \in I$ and $k \in \N$ we define a kernel $\mathcal{K}_{k,t} : \R^{d+1} \times \R^{d+1} \times \R^{d+1} \to \C$ by induction: we set $\mathcal{K}_{0,t}\p{x,\xi,\eta} = G_t\p{x}$ and for all $k \in \N$
\begin{equation*}
\mathcal{K}_{k+1,t}\p{x,\cdot,\cdot} = \sum_{j=1}^{d+1} \partial_{x_j} \p{\frac{l_j\p{x} \mathcal{K}_{k,t}\p{x,\cdot,\cdot}}{\Phi\p{x}}}.
\end{equation*}
Integrating by parts in $y$ we see that the inner integral of \eqref{eqnorme2} is equal, for all $k \in \N, t \in I$ and $\xi \in \R^{d+1}$, to
\begin{equation}\label{eqinnint}
\int_{\p{\R^{d+1}}^2} e^{-i x \xi} e^{i \mathcal{T}_t\p{x} \eta} \tilde{\psi}_{\Theta,\ell,j}\p{\eta} \mathcal{K}_{k,t}\p{x,\xi,\eta} \hat{u}\p{\eta} \mathrm{d}x \mathrm{d}\eta.
\end{equation}
To bound the kernel $\mathcal{K}_{k,t}$, we notice that it is the sum of at most $\p{5 \p{d+1}}^k k!$ terms of the form
\begin{equation}\label{eqform}
\begin{split}
\p{x,\xi,\eta} \mapsto \pm \frac{\partial^{\sigma} G_t\p{x}}{\p{\Phi\p{x}\p{\xi,\eta}}^{k+m}} & \partial^{\gamma_1} l_{j_1}\p{x}\p{\xi,\eta} \dots \partial^{\gamma_k} l_{j_k}(x,\xi) \\ & \times \partial^{\mu_1} \Phi\p{x}\p{\xi,\eta} \dots \partial^{\mu_m} \Phi\p{x}\p{\xi,\eta},
\end{split}
\end{equation}
where $m \leq k$ is an integer, $j_1,\dots,j_k \in \set{1,\dots,d+1}$, and $\sigma,\gamma_1,\dots,\gamma_k,\mu_1,\dots,\mu_m$ are elements of $\N^{d+1}$ such that $\va{\sigma} + \va{\gamma_1} + \dots + \va{\gamma_k} + \va{\mu_1} + \dots +\va{\mu_m} = k$ (all the derivatives are with respect to the variable $x$).

Now, Lemma~\ref{lmdist} implies that if $x \in K$, if $ \xi \in \textup{ supp } \psi_{\Theta',n,i}$ and if $\eta \in \textup{ supp } \tilde{\psi}_{\Theta,\ell,j}$ then
\begin{equation*}
\begin{split}
\Phi\p{x}\p{\xi,\eta} & \geq \p{c'}^2 \p{\max\p{2^{n^{\alpha_i}}, 2^{n^{\alpha_j}}}}^2 \\ & \geq c_1 \max\p{2^{n^{\alpha_i}}, 2^{n^{\alpha_j}}} \max\p{\va{\xi},\va{\eta}} \geq c_2 \max\p{\va{\xi},\va{\eta}}^2,
\end{split}
\end{equation*}
for some positive constants $c_1$ and $c_2$. Consequently, there is a constant $C >0$ such that if $l$ is a linear map from $\R^{d+1} \times \R^{d+1} \to \C$ and if $q$ is a quadratic map $\R^{d+1} \times \R^{d+1} \to \C$ then we have, for all $x \in K, \xi \in \textup{ supp } \psi_{\Theta',n,i}$ and $\eta \in \textup{ supp } \tilde{\psi}_{\Theta,\ell,j}$
\begin{equation*}
\va{\frac{l\p{\xi,\eta}}{\Phi\p{x}\p{\xi,\eta}}} \leq C \frac{\n{l}}{\max\p{2^{n^{\alpha_i}}, 2^{\ell^{\alpha_j}}}} \textup{ and }  \va{\frac{q\p{\xi,\eta}}{\Phi\p{x}\p{\xi,\eta}}} \leq C \n{q}.
\end{equation*}
The choice of the norms on the spaces of linear and quadratic maps $\R^{d+1} \times \R^{d+1} \to \C$ is of course irrelevant. Thus for such $x,\xi$ and $\eta$ any term of the form \eqref{eqform} is bounded by
\begin{equation*}
\begin{split}
C^{2k} \p{\max\p{2^{n^{\alpha_i}}, 2^{\ell^{\alpha_j}}}}^{-k} \n{\partial^{\sigma} G_t}_{\infty} & \n{\partial^{\gamma_1} l_{j_1}}_\infty \dots \n{\partial^{\gamma_k} l_{j_k}} \\ & \times \n{\partial^{\mu_1} \Phi}_{\infty} \dots \n{\partial^{\mu_m} \Phi}_\infty ,
\end{split}
\end{equation*}
where $\n{\cdot}_\infty$ refers to the supremum of the corresponding norm on $K$. Now, notice that, since $\mathcal{T}_0$ is $\mathcal{C}^{\kappa,\upsilon}$ then for any $\kappa' < \kappa$ the maps $l_{1},\dots,l_{d+1}$ (valued in the space of linear maps from $\R^{d+1}\times \R^{d+1}$ to $\C$) and $\Phi$ (valued in the space of quadratic maps from $\R^{d+1} \times \R^{d+1}$ to $\C$) are $\mathcal{C}^{\kappa',\upsilon}$ (we can event take $\kappa' = \kappa$ if $\upsilon \leq 2$). Thus there are constants $M,R >0$ such that for all $\mu \in \N^d$, we have
\begin{equation*}
\n{\partial^{\mu} \Phi}_\infty \leq M R^{\va{\mu}} \va{\mu}! \exp\p{\frac{\va{\mu}^\upsilon}{\kappa'}},
\end{equation*}
for all $\gamma \in \N^d$ and $j \in \set{1,\dots,d+1}$, we have
\begin{equation*}
\n{\partial^\gamma l_j}_\infty \leq M R^{\va{\gamma}} \va{\gamma}! \exp\p{\frac{\va{\gamma}^\upsilon}{\kappa'}},
\end{equation*}
and for all $t \in I$ and $\sigma \in \N^d$, we have
\begin{equation*}
\n{\partial^\sigma G_t}_\infty \leq M R^{\va{\sigma}} \va{\sigma}! \exp\p{\frac{\va{\sigma}^\upsilon}{\kappa'}}.
\end{equation*}
Thus each term of the form \eqref{eqform} is bounded by
\begin{equation*}
C^{2k} M^{2k + 1} R^k k^k \exp\p{\frac{k^\upsilon}{\kappa'}} 2^{- k \max\p{n,\ell}^\alpha}
\end{equation*}
when $x \in K, \xi \in \textup{ supp } \psi_{\Theta',n,i},\eta \in \textup{ supp } \tilde{\psi}_{\Theta,\ell,j}$ and $t \in I$. Consequently, for such $x,\xi,\eta$ and $t$ the kernel $\mathcal{K}_{k,t}\p{x,\xi,\eta}$ is bounded for all integers $k$ by
\begin{equation}\label{eqbornk}
2^{-k \max\p{n,\ell}^\alpha} \p{5\p{d+1}}^k C^{2k} M^{2k + 1} R^k k^{2k} \exp\p{\frac{k^\upsilon}{\kappa'}}.
\end{equation} 
Now, choose $\kappa'' >0$ such that $\frac{1}{\kappa'} + 2 \leq \frac{1}{\kappa''}$ and pick new values of the constants $M$ and $R$ so that \eqref{eqbornk} is now smaller than
\begin{equation*}
M \p{\frac{R}{2^{\max\p{n,\ell}^\alpha}}}^k \exp\p{\frac{k^\upsilon}{\kappa''}}. 
\end{equation*}
Now, using this estimate and Cauchy--Schwarz in \eqref{eqinnint}, we bound the inner integral in \eqref{eqnorme2} by
\begin{equation*}
\widetilde{C} \n{u}_2 2^{\frac{(d+1)\ell}{2}} \p{\frac{R}{2^{\max\p{n,\ell}^\alpha}}}^k \exp\p{\frac{k^\upsilon}{\kappa''}},
\end{equation*}
which gives 
\begin{equation*}
\n{S_{t,n,i}^{\ell,j} u}_2 \leq C' \n{u}_2 2^{\frac{(\ell + n)(d+1)}{2}} \p{\frac{R}{2^{\max\p{n,\ell}^\alpha}}}^k \exp\p{\frac{k^\upsilon}{\kappa''}}.
\end{equation*}
Now take $k = \left\lfloor \p{\frac{- \kappa'' \ln\p{\frac{R}{2^{\max\p{n,\ell}^\alpha}}}}{\upsilon}}^{\frac{1}{\upsilon - 1}} \right\rfloor$ to get (with new constants and $\delta = \frac{\alpha \upsilon}{\upsilon - 1} > 1$, see the proof of Proposition \ref{propdecfour} for a similar computation)
\begin{equation*}
\n{S_{t,n,i}^{\ell,j}u}_2 \leq C \n{u}_2 2^{\frac{(\ell + n)(d+1)}{2}} \exp\p{- \frac{\max\p{n,\ell}^\delta}{C}}.
\end{equation*}
Finally, we get rid of the factor $2^{\frac{(\ell + n)(d+1)}{2}}$ by taking larger $C$.
\end{proof}

We can now prove Lemma~\ref{lmconvstrong} about the family $\p{\mathcal{M}_t}_{t \in \R}$ of auxiliary operators.

\begin{proof}[Proof of Lemma~\ref{lmconvstrong}]
First of all, thanks to Lemma~\ref{lmmajipp}, the sum
\begin{equation*}
\sum_{\substack{\p{n,i},\p{\ell,j} \in \Gamma \\ \p{\ell,j} \nhra \p{n,i}}} \iota_{n,i} \circ S_{t,n,i}^{\ell,j} \circ \pi_{\ell,j}
\end{equation*}
converges absolutely in norm operator topology.

Now, we have to deal with the sum
\begin{equation}\label{eqsommehook}
\sum_{\substack{\p{n,i},\p{\ell,j} \in \Gamma \\ \p{n,i} \hra \p{\ell,j}}} \iota_{n,i} \circ S_{t,n,i}^{\ell,j} \circ \pi_{\ell,j}.
\end{equation}
To do so, notice that there is some constant $C$ depending on $I$ such that, for all $t \in I$ and $\p{n,i},\p{\ell,j} \in \Gamma$, we have
\begin{equation}\label{eqmajnormop}
\n{\iota_{n,i} \circ S_{t,n,i}^{\ell,j} \circ \pi_{\ell,j}}_{\B \to \B} \leq C 2^{n \beta_i} 2^{-\ell \beta_j}.
\end{equation}
Then the sum \eqref{eqsommehook} can be divided into seven sums that correspond to the different cases in the definition of $\hra$. It is elementary, using \eqref{eqmajnormop}, to see that the first six converge in norm operator topology. Consequently, we are left with the sum
\begin{equation}\label{eqleftsum}
\sum_{\substack{n,\ell \in \N \\ \va{n-\ell} \leq M}} \iota_{n,f} \circ S_{t,n,f}^{\ell,f} \circ \pi_{\ell,f}
\end{equation}
for some $M >0$. For all $N_1 \in \N$, define the operator
\begin{equation*}
P_{N_1} = \sum_{\substack{0 \leq n,\ell \leq N_1 \\ \va{n-\ell} \leq M}} \iota_{n,f} \circ S_{t,n,f}^{\ell,f} \circ \pi_{\ell,f}.
\end{equation*}
Pick $u = \p{u_{m,k}}_{\p{m,k} \in \Gamma} \in \B$. Then if $N_2 \geq N_1 \geq 0$, we have
\begin{equation}\label{eqdernier}
\begin{split}
\n{\p{P_{N_2} - P_{N_1}}u}_{\B}^2 & \leq 2 \sum_{n = 0}^{N_2} 2^{-2\p{d+2}n}\n{\sum_{\substack{N_1 < \ell \leq N_2 \\ \va{\ell - n} \leq M}} S_{t,n,f}^{\ell,f} u_{\ell,f}}_2^2 \\ & \quad  + 2 \sum_{n = N_1 + 1}^{N_2} 2^{-2\p{d+2}n}\n{\sum_{\substack{0 \leq \ell \leq N_1 \\ \va{\ell - n} \leq M}} S_{t,n,f}^{\ell,f} u_{\ell,f}}_2^2.
\end{split}
\end{equation}
Next, we have by the triangle inequality,
\begin{equation*}
\begin{split}
\n{\sum_{\substack{N_1 < \ell \leq N_2 \\ \va{\ell - n} \leq M}} S_{t,n,f}^{\ell,f} u_{\ell,f}}_2^2 \leq \p{\sum_{\substack{N_1 < \ell \leq N_2 \\ \va{\ell - n} \leq M}} \n{S_{t,n,f}^{\ell,f}u_{\ell,f}}_2}^2 \leq C \p{\sum_{\substack{N_1<\ell \leq N_2 \\ \va{\ell -n} \leq M}} \n{u_{\ell,f}}_2}^2,
\end{split}
\end{equation*}
for some constant $C>0$. Then, from the Cauchy--Schwarz inequality, we get
\begin{equation*}
\begin{split}
\p{\sum_{\substack{N_1<\ell \leq N_2 \\ \va{\ell -n} \leq M}} \n{u_{\ell,f}}_2}^2 & =\p{\sum_{\substack{N_1<\ell \leq N_2 \\ \va{\ell -n} \leq M}} 2^{\ell \p{d+2}} 2^{-\ell \p{d+2}} \n{u_{\ell,f}}_2}^2 \\
   & \leq \sum_{\substack{N_1<\ell \leq N_2 \\ \va{\ell -n} \leq M}} 2^{2 \ell\p{d+2}} \sum_{\substack{N_1<\ell \leq N_2 \\ \va{\ell -n} \leq M}} 2^{-2 \ell \p{d+2}} \n{u_{\ell,f}}_2^2 \\
   &\leq C' 2^{2n\p{d+2}} \sum_{\substack{N_1<\ell \leq N_2 \\ \va{\ell -n} \leq M}} 2^{-2 \ell \p{d+2}} \n{u_{\ell,f}}_2^2 
\end{split}
\end{equation*}
for another constant $C' >0$. Consequently, we can bound the first sum in \eqref{eqdernier}
\begin{equation*}
\begin{split}
\sum_{n = 0}^{N_2} 2^{-2n\p{d+2}} \n{\sum_{\substack{N_1 < \ell \leq N_2 \\ \va{\ell - n} \leq M}} S_{t,n,f}^{\ell,f} u_{\ell,f}}_2^2 & \leq C C' \sum_{n=0}^{N_2} \sum_{\substack{N_1<\ell \leq N_2 \\ \va{\ell -n} \leq M}} 2^{-2 \ell \p{d+2}} \n{u_{\ell,f}}_2^2  \\
    & \leq \widetilde{C} \sum_{ \ell > N_1} 2^{- 2 \ell \p{d+2}} \n{u_{\ell,f}}_2^2,
\end{split}
\end{equation*}
where in the last line we notice that, when $\ell$ is fixed, there are at most $2M + 1$ values of $n$ for which $\va{\ell - n} \leq M$. Working similarly with the second sum, we see that there is a constant $C$ such that
\begin{equation*}
\n{\p{P_{N_2} - P_{N_1}}u}_{\B}^2 \leq C \sum_{\ell \geq N_1 - M} 2^{- 2\p{d+2} \ell} \n{u_{\ell,f}}_2^2,
\end{equation*}
and thus the sequence $\p{P_{N_1}u}_{N_1 \geq 0}$ is Cauchy in $\B$. Consequently, the sequence $\p{P_N}_{N \geq 0}$ converges in strong operator topology, hence, so does the sum \eqref{eqleftsum}.

To prove that $\mathcal{M}_t$ depends continuously on $t$ in the strong operator topology, just notice that when $u$ is fixed the sum
\begin{equation*}
\sum_{\p{n,i},\p{\ell,j} \in \Gamma} \iota_{n,i} \circ S_{t,n,i}^{\ell,j} \circ \pi_{\ell,j} u
\end{equation*}
converges uniformly (in $t \in I$) to $\mathcal{M}_t u$ and each of its terms is continuous with respect to $t$ (to see this, notice that if $\p{n,i},\p{\ell,j} \in \Gamma$ then $S_{t,n,i}^{\ell,j}$ is locally uniformly bounded as an operator from $L^2$ to $L^2$, and the continuity is easily proven for smooth $u$).
\end{proof}

\subsection{Schatten class properties}\label{subsecschatten}

Now let $h : \R_+^* \to \C$ be a compactly supported function as in Proposition~\ref{proplocop}. If $\p{n,i},\p{\ell,j} \in \Gamma$, then write
\begin{equation*}
H_{n,i}^{\ell,j} = \s{\R}{h\p{t} S_{t,n,i}^{\ell,j}}{t},
\end{equation*}
where we recall that $S_{t,n,i}^{\ell,j}$ is defined by \eqref{eqdefentree}.

Notice that the sum
\begin{equation*}
\sum_{\p{n,i},\p{\ell,j} \in \Gamma} \iota_{n,i} \circ H_{n,i}^{\ell,j} \circ \pi_{\ell,j}
\end{equation*}
converges in strong operator topology to $\int_{\R} h\p{t} \mathcal{M}_t \mathrm{d}t$, since the convergence in Lemma~\ref{lmconvstrong} is uniform. To prove Proposition~\ref{proplocop}, we want now to prove that this operator is in a Schatten class (or at least compact), this is the point of Lemma~\ref{lmschatten}. To do so we need first to establish a bunch of lemmas: Lemma~\ref{lmtc} will be used to deal with the transition of frequencies corresponding to the linear model of the dynamics apart from the direction of the flow, Lemma~\ref{lmtcnl} will settle the problem of frequency transitions corresponding to the non-linearity, and Lemmas \ref{lmippf} and~\ref{lmfiniterank} will be used to deal with stationary frequencies in the direction of the flow.

\begin{lm}\label{lmtc}
There is a constant $C >0$ such that, for all $\p{n,i},\p{\ell,j} \in \Gamma$, the trace class operator norm of $H_{n,i}^{\ell,j} : L^2 \to L^2$ is bounded by $C 2^{\frac{\p{d+1} n^{\alpha_i}}{2}} 2^{\frac{\p{d+1}\ell^{\alpha_j}}{2}}$, where $\alpha_i = \alpha$ if $i= 0$ or $i=r-1$ and $\alpha_i = 1$ otherwise.
\end{lm}

\begin{proof}
Notice that if $u \in L^2$ then $\psi_{\Theta',n,i}\p{D} u = \mathbb{F}^{-1}\p{\psi_{\Theta',n,i}} \ast u$. Consequently, we have\footnote{If $E,F$ are Banach spaces, $e \in F$ and $l \in E'$, we denote by $e \otimes l$ the rank $1$ operator defined by $e \otimes l (u) = l(u). e$ for $u \in E$.}
\begin{equation}\label{eqopnoy}
H_{n,i}^{\ell,j} = \int_K \mathbb{F}^{-1}\p{\psi_{\Theta',n,i}}\p{ \cdot - y} \otimes \p{\int_{\R} h\p{t} G_t\p{y} \delta_{\mathcal{T}_t\p{y}} \circ \tilde{\psi}_{\Theta,\ell,j}\p{D} \mathrm{d} t} \mathrm{d}y.
\end{equation}
And then the result follows from the fact that
\begin{equation*}
\n{\mathbb{F}^{-1}\p{\psi_{\Theta',n,i}}}_2 = \frac{1}{\sqrt{2 \pi}^{d+1}} \n{\psi_{\Theta',n,i}}_2 \leq C 2^{\frac{\p{d+1}n^{\alpha_i}}{2}}
\end{equation*}
and
\begin{equation*}
\n{\int_{\R} h\p{t} G_t\p{y} \delta_{\mathcal{T}_t\p{y}} \circ \tilde{\psi}_{\Theta,\ell,j}\p{D} \mathrm{d} t}_{\p{L^2}^*} \leq C \n{\tilde{\psi}_{\Theta,\ell,j}}_2 \leq \tilde{C} 2^{\frac{\p{d+1} \ell^{\alpha_j}}{2}},
\end{equation*}
where $\n{\cdot}_{\p{L^2}^*}$ denotes the operator norm on the dual of $L^2\p{\R^{d+1}}$.
\end{proof}

\begin{lm}\label{lmtcnl}
There is a constant $C >0$ and some $\delta >1$ such that, if $\p{\ell,j} \nhra \p{n,i}$ for $\p{n,i}, \p{\ell,j} \in \Gamma$, then the trace class operator norm of $H_{n,i}^{\ell,t} : L^2 \to L^2$ is bounded by $C \exp\p{-\frac{\max\p{n,\ell}^\delta}{C}}$.
\end{lm}

\begin{proof}
We may assume that $ \max\p{n,\ell} > N$. Without loss of generality, we may assume that $K \subseteq \left]-\pi,\pi\right[^{d+1}$ and then, if $u \in L^2\p{\R^{d+1}}$ write (the sum converges in $L^2$)
\begin{equation*}
H_{n,i}^{\ell,j}u = \sum_{k \in \Z^{d+1}} c_k\p{\int_{\R} h\p{t} \mathcal{L}_t \tilde{\psi}_{\Theta,\ell,j}\p{D} u \mathrm{d}t} \psi_{\Theta',n,i}\p{D}\rho_k,
\end{equation*}
where $\rho$ is a function supported in $\left]-\pi,\pi\right[^{d+1}$ that takes value $1$ on $K$, the function $\rho_k$ is defined by $\rho_k\p{x} = \rho\p{x} e^{ikx}$ and if $v$ is supported in $\left]-\pi,\pi\right[^{d+1}$ and $k \in \Z^{d+1}$, its $k$th Fourier coefficient is denoted by $c_k\p{v}$:
\begin{equation*}
c_k\p{v} = \frac{1}{\p{2\pi}^{d+1}} \s{\left]-\pi,\pi\right[^{d+1}}{e^{-ikx}v\p{x}}{x}.
\end{equation*}
By requiring that $\rho$ is $\sigma$-Gevrey (for some $\sigma >1$), we may ensure as in \cite[Lemma 6.5]{lagtf} that (for some constant $C > 0$)
\begin{equation*}
\n{\psi_{\Theta',n,i}\p{D} \rho_k}_2 \leq C 2^{\frac{\p{d+1}n^{\alpha_i}}{2}} \exp\p{-\frac{d\p{k,\textup{supp } \psi_{\Theta',n,i}}^{\frac{1}{\sigma}}}{C}}.
\end{equation*}
Now, if $k \in \Z^{d+1}$ and $\p{\ell,j} \in \Gamma$ define
\begin{equation*}
\delta\p{k,\ell,j} = \sup_{x \in K} d\p{k, D_x \mathcal{T}_t^{\textup{tr}} \p{\textup{supp } \tilde{\psi}_{\Theta,\ell,j}}}.
\end{equation*}
Then integrating by parts as in \cite[Lemma 6.7]{lagtf} or as in Lemma~\ref{lmmajipp} we see that if $\delta\p{k,\ell,j} \geq \epsilon 2^{\ell^{\alpha_j}}$ (for some arbitrary fixed $\epsilon >0$) then
\begin{equation*}
\n{c_k \circ \int_{\R}h\p{t}\mathcal{L}_t \mathrm{d}t \circ \tilde{\psi}_{\Theta,\ell,j}\p{D}}_{\p{L^2}^*} \leq C 2^{\p{d+1} \ell^{\alpha_j}}\exp\p{- \frac{\ln\p{1 + \delta\p{k,\ell,j}}^{\frac{\upsilon}{\upsilon-1}}}{C}}.
\end{equation*}
But now, if $\p{\ell,j} \nhra \p{n,i}$ and $\max\p{n,l} > N$, then, for all $k \in \Z^{d+1}$, either the distance $d\p{k,\textup{supp } \psi_{\Theta',n,i}}$ or the distance $\delta\p{k,\ell,j}$ is greater than $\frac{c'}{2} \max\p{2^{n^{\alpha_i}},2^{\ell^{\alpha_j}}}$, thanks to Lemma~\ref{lmdist}. Moreover, if $\va{k}$ is greater than $C 2^{\max\p{n,\ell}}$ (for some large $C >0$), then we have $\delta\p{k,\ell,j} \geq \epsilon 2^{\ell^{\alpha_j}}$ and $d\p{k,\textup{supp } \psi_{\Theta',n,i}} \geq \epsilon \va{k}$. Thus, the sum
\begin{equation*}
H_{n,i}^{\ell,j} = \sum_{k \in \Z^{d+1}} \p{\psi_{\Theta',n,i}\p{D} \rho_k} \otimes \p{c_k \circ \s{\R}{h\p{t}L_t}{t} \circ \tilde{\psi}_{\Theta,\ell,j}\p{D}}
\end{equation*}
converges in trace class topology, and the estimates above imply the result with $\delta = \frac{\alpha \upsilon}{\upsilon - 1}$.
\end{proof}

\begin{lm}\label{lmippf}
Assume that $h$ is $k$th times differentiable and that is $k$th derivative has bounded variation. Then there is a constant $C >0$ such that for all $n,\ell \in \N$ we have
\begin{equation*}
\n{H_{n,f}^{\ell,f}}_{L^2 \to L^2} \leq C 2^{-\p{k+1}\ell}.
\end{equation*}
\end{lm}

\begin{proof}
If $u \in L^2\p{\R^{d+1}}$ and $x \in \R^{d+1}$, then we have,
\begin{equation*}
H_{n,f}^{\ell,f}u\p{x} = \s{\R^{d+1}}{V_{n,\ell}\p{x,\eta} \hat{u}(\eta)}{\eta},
\end{equation*}
where the kernel $V_{n,\ell}$ is defined by
\begin{equation}\label{eqnoyau}
\begin{split}
& V_{n,\ell}\p{x,\eta}  \\  & \quad = \frac{1}{\p{2\pi}^{2\p{d+1}}} \int_{\p{\R^{d+1}}^3\times \R} e^{i\p{x-z}\xi + i \mathcal{T}_0\p{z} \eta} e^{i t \eta_{d+1}} \psi_{\Theta',n,f}\p{\xi} \\ & \qquad \qquad \qquad \qquad \qquad \qquad \qquad \qquad \times \tilde{\psi}_{\Theta,\ell,f}\p{\eta} h\p{t} G_t\p{z} \mathrm{d}z \mathrm{d}\xi \mathrm{d}t.
\end{split}
\end{equation}
We can assume that $ \ell$ is large enough (the $H_{n,f}^{\ell,f}$'s are uniformly bounded on $L^2$), which ensures that $\eta_{d+1}$ (the last coordinate of $\eta$) does not vanish on the support of $\tilde{\psi}_{\Theta,\ell,f}$. Consequently, we can perform $\p{k+1}$ integrations by parts in $t$ in \eqref{eqnoyau} to get
\begin{equation*}
\begin{split}
& V_{n,\ell}\p{x,\eta} \\ & \quad = \frac{i^{k+1}}{\p{2\pi}^{2\p{d+1}}} \int_{\p{\R \times \R^{d+1}}^3} e^{i\p{x-z}\xi + i \mathcal{T}_0\p{z} \eta} e^{i t \eta_{d+1}} \psi_{\Theta',n,f}\p{\xi} \\ & \qquad \qquad \qquad \qquad \qquad \qquad \times \frac{\tilde{\psi}_{\Theta,\ell,f}\p{\eta}}{\eta_{d+1}^{k+1}} \frac{\mathrm{d}^{k+1}}{\mathrm{d}t^{k+1}}\p{ h\p{t} G_t\p{z}} \mathrm{d}t \mathrm{d}z \mathrm{d}\xi.
\end{split}
\end{equation*}
Using the Leibniz rule, we see that, if $\mu$ denotes the measure of total variation of $h^{\p{k+1}}$, the measure $\frac{\mathrm{d}^{k+1}}{\mathrm{d}t^{k+1}}\p{ h\p{t} G_t\p{z}} \mathrm{d}t$ may be written as $f\p{t,z} \mathrm{d}\mu \p{t}$ for all $z \in \R^{d+1}$. Moreover, $f$ has the following properties: it is measurable, $f\p{t,z} = 0$ if $z \in \R^{d+1} \setminus K$, and $\int_{\R} \sup_{z \in \R^{d+1}} \va{f\p{t,z}} \mathrm{d}\mu\p{t} < + \infty$.  Then, define the function $\Psi_{\ell} : \R^{d+1} \to \R$ by $\Psi_{\ell}\p{\eta} = \frac{\tilde{\psi}_{\Theta,\ell,f}\p{\eta}}{\eta_{d+1}^{k+1}}$, the operator $L_t : L^2\p{\R^{d+1}} \to L^2\p{\R^{d+1}}$ by $L_t u \p{z} = f\p{t,z} . \p{u \circ \mathcal{T}_t\p{z}}$, and notice that we have
\begin{equation*}
H_{n,f}^{\ell,f} = \psi_{\Theta',n,f}\p{D} \circ \int_{\R} L_t \mathrm{d}\mu\p{t} \circ \Psi_{\ell}\p{D}.
\end{equation*}
Finally, notice that $\n{\Psi_\ell}_{\infty} \leq C 2^{- \ell\p{k+1}}$ to end the proof.
\end{proof}

\begin{lm}\label{lmfiniterank}
Let $s>0$ and $\epsilon>0$. Then there is a constant $C >0$ such that for all $N >0$ and $n \in \N$ with $n < N$ there is an operator $F_{n,N} : L^2\p{K} \to L^2\p{\R^{d+1}}$ of rank at most $2^{\p{1+\epsilon}\p{d+1}N}$ such that for all $u \in L^2\p{K}$ we have
\begin{equation*}
\n{\psi_{\Theta',n,f}\p{D}u - F_{n,N}u}_2 \leq C 2^{-sN}.
\end{equation*}
\end{lm}

\begin{proof}
The proof is similar to the proof of \cite[Lemma 4.21]{Tsu}.
\end{proof}

We are now ready to prove Lemma~\ref{lmschatten}.

\begin{lm}\label{lmschatten}
Under the hypotheses of Lemma~\ref{lmippf} and if in addition $\alpha < \frac{1}{2}$, the operator
\begin{equation}\label{eqnucopaux}
\int_{\R} h\p{t} \mathcal{M}_t \mathrm{d}t
\end{equation}
belongs to the Schatten class $S_p$ for every $p \geq 1$ such that $p > \frac{d+1}{k+1}$ . Moreover, its norm in this Schatten class is bounded by $C \p{\n{h}_{\mathcal{C}^{k-1}} + \n{h^{\p{k}}}_{BV}}$ where $C$ depends on $h$ only through its support.

Without the assumption that $\alpha < \frac{1}{2}$, it remains true that the operator defined by \eqref{eqnucopaux} is compact.
\end{lm}

\begin{proof}
We know that
\begin{equation}\label{eqmt}
\int_{\R}h\p{t} \mathcal{M}_t \mathrm{d}t = \sum_{\p{n,i},\p{\ell,j} \in \Gamma} \iota_{n,i} \circ H_{n,i}^{\ell,j} \circ \pi_{\ell,j}
\end{equation}
where the sum converges in the strong operator topology. From Lemma~\ref{lmtcnl}, it is clear that the sum
\begin{equation*}
\sum_{\substack{\p{n,i},\p{\ell,j} \in \Gamma \\ \p{n,i} \nhra \p{\ell,j}}} \iota_{n,i} \circ H_{n,i}^{\ell,j} \circ \pi_{\ell,j}
\end{equation*}
converges in the trace class operator topology. We are left with the sum
\begin{equation*}
\sum_{\substack{\p{n,i},\p{\ell,j} \in \Gamma \\ \p{n,i} \hra \p{\ell,j}}} \iota_{n,i} \circ H_{n,i}^{\ell,j} \circ \pi_{\ell,j}
\end{equation*}
that we can divide, as in the proof of Lemma~\ref{lmconvstrong}, into seven sums corresponding to the different cases in the definition of $\hra$. The first six sums are dealt with by using Lemma~\ref{lmtc}. We will only detail the computation corresponding to the first case in the definition of $\hra$ (i.e. the case $i=j=0$, the case $i=j=r-1$ is dealt with in the same way and the others are easier), in order to highlight where the hypothesis $\alpha < \frac{1}{2}$ is used. If $n,\ell \in \N$, then the trace class operator norm of $\iota_{n,0} \circ H_{n,0}^{\ell,0} \circ \pi_{\ell,0}$ is smaller than $C 2^{\frac{\p{d+1}n^\alpha}{2}} 2^{\frac{\p{d+1}\ell^\alpha}{2}} 2^{\p{d+2}n} 2^{-\p{d+2}\ell}$. Thus, in order to deal with the sum corresponding with the case $i = j =0$ in the definition of $\hra$, we only need to prove that the quantity
\begin{equation}\label{eqsumalpha}
\sum_{\substack{\ell,n \in \N \\ \p{\ell,0} \hra \p{n,0}}} 2^{\frac{\p{d+1}n^\alpha}{2}} 2^{\frac{\p{d+1}\ell^\alpha}{2}} 2^{\p{d+2}n} 2^{-\p{d+2}\ell}
\end{equation}
is finite. Notice that
\begin{equation*}
2^{-\p{d+2}\ell + \frac{\p{d+1}\ell^\alpha}{2}} \underset{\ell \to + \infty}{\sim} \frac{1}{1 - 2^{-\p{d+2}}} \p{2^{-\p{d+2}\ell + \frac{\p{d+1}\ell^\alpha}{2}} - 2^{\p{d+2}\p{\ell+1} + \frac{\p{d+1}\p{\ell + 1}^\alpha}{2}}}
\end{equation*}
so that
\begin{equation*}
\sum_{\ell \geq \ell_0} 2^{-\p{d+2}\ell + \frac{\p{d+1}\ell^\alpha}{2}} \underset{\ell_0 \to + \infty}{\sim} \frac{ 2^{-\p{d+2}\ell_0 + \frac{\p{d+1}\ell_0^\alpha}{2}}}{1 - 2^{-\p{d+2}}}.
\end{equation*}
In particular, there is a constant $C >0$ such that, for all $\ell_0 \in \N$. We have 
\begin{equation*}
\sum_{\ell \geq \ell_0} 2^{-\p{d+2}\ell + \frac{\p{d+1}\ell^\alpha}{2}} \leq C 2^{-\p{d+2}\ell_0 + \frac{\p{d+1}\ell_0^\alpha}{2}}.
\end{equation*}
Now if $n \in \N$, let $\ell_0$ be the smallest integer such that $\ell_0 \geq n + \nu n^{1-\alpha}$, we have then (notice that $\ell_0 \leq B n$ for some constant $B$ that does not depend on $n$)
\begin{equation*}
\begin{split}
\sum_{\substack{\ell \in \N \\ \p{\ell,0} \hra \p{n,0}}} 2^{-\p{d+2}\ell + \frac{\p{d+1}\ell^\alpha}{2}} & = \sum_{\ell \geq \ell_0} 2^{-\p{d+2}\ell + \frac{\p{d+1}\ell^\alpha}{2}} \\
   & \leq C 2^{-\p{d+2}\ell_0 + \frac{\p{d+1}\ell_0^\alpha}{2}} \\
   & \leq C 2^{-\p{d+2}n} 2^{-\p{d+2} \nu n^{1-\alpha}} 2^{\frac{\p{d+1}B^\alpha}{2} n^\alpha}.
\end{split}
\end{equation*}
Thus, we have
\begin{equation}\label{eqcellequiconverge}
\begin{split}
& \sum_{\substack{\ell,n \in \N \\ \p{\ell,0} \hra \p{n,0}}} 2^{\frac{\p{d+1}n^\alpha}{2}} 2^{\frac{\p{d+1}\ell^\alpha}{2}} 2^{\p{d+2}n} 2^{-\p{d+2}\ell} \\ & \qquad \qquad \leq C \sum_{n \in \N} 2^{-\p{d+2} \nu n^{1-\alpha}} 2^{\frac{\p{d+1}\p{B^\alpha + 1}}{2} n^\alpha}, 
\end{split} 
\end{equation}
and this sum is finite since $\alpha < \frac{1}{2}$.

Finally, we are left with the sum
\begin{equation*}
P = \sum_{\substack{n,\ell \in \N \\ \va{n-\ell} \leq M} } \iota_{n,f} \circ H_{n,f}^{\ell,f} \circ \pi_{\ell,f}.
\end{equation*}
Choose $s> k+1$ and $\epsilon >0$, and apply Lemma~\ref{lmfiniterank} to define for all $N >0$ the operator
\begin{equation*}
P_N = \sum_{\substack{0 \leq n,\ell < N \\ \va{n- \ell} \leq M}} \iota_{n,f} \circ F_{n,N} \circ \s{\R}{h\p{t}\mathcal{L}_t}{t} \circ \tilde{\psi}_{\Theta,\ell,j} \circ \pi_{\ell,f},
\end{equation*}
whose rank is at most $N^2 2^{\p{1+\epsilon}\p{d+1} N}$. Then notice, using Lemma~\ref{lmippf}, that we have
\begin{equation}\label{eqappfin}
\begin{split}
& \n{P_N - P}_{\B \to \B} \\ & \quad \leq C\sum_{\substack{n,\ell < N \\ \va{n-\ell} \leq M}} \Bigg( \n{F_{n,N} - \psi_{\Theta',n,f}\p{D}}_{L^2\p{K} \to L^2\p{R^{d+1}}} \\ &  \qquad \qquad \qquad \qquad \qquad \times \n{\s{\R}{h\p{t} \mathcal{L}_t}{t} \circ \tilde{\psi}_{\Theta,\ell,f}\p{D}}_{L^2 \to L^2} \Bigg) \\ & \quad \quad  + C\sum_{\substack{n,\ell \geq N \\ \va{n - \ell} \leq M}} \n{H_{n,f}^{\ell,f}}_{L^2 \to L^2} \\
   & \quad \leq \widetilde{C} \p{N^2 2^{-sN} + 2^{-\p{k+1} N}} \leq C' 2^{-\p{k+1}N},
\end{split}
\end{equation}
for some constants $C,\tilde{C}$ and $C'$ that do not depend on $N$. Letting $N$ tend to infinity, we see that $P$ is compact. Moreover, if $\p{s_m}_{m \geq 0}$ denotes the sequence of singular values of $P$, we get from \eqref{eqappfin} and \cite[Theorem 2.5 p.51]{Gohb}
\begin{equation*}
s_{N^2 2^{\p{1+\epsilon} \p{d+1} N}+1} \leq \tilde{C} 2^{-\p{k+1}N}.
\end{equation*}
Thus, the sequence $\p{s_m}_{m \geq 0}$ is in $\ell^p$ for all $p > \frac{\p{1 + \epsilon} \p{d+1}}{k+1}$ (the sequence $\p{s_m}_{m \geq 0}$ is decreasing). This ends the proof in the case $\alpha < \frac{1}{2}$ since $\epsilon >0$ is arbitrary. Indeed, all the terms in the proof are controlled by the $L^\infty$ norm of $h$, except the one that we bounded using Lemma \ref{lmippf} that is controlled by $\n{h}_{\mathcal{C}^{k-1}} + \n{h^{(k)}}_{BV}$.

In order to deal with the case $\alpha \geq \frac{1}{2}$, notice that we only used the assumption $\alpha < \frac{1}{2}$ to ensure that the series \eqref{eqcellequiconverge} converges. However, if we remove the factor $2^{\frac{\p{d+1}n^\alpha}{2}} 2^{\frac{\p{d+1}\ell^\alpha}{2}}$ from the sum \eqref{eqsumalpha}, this new series converges, just like in the proof of Lemma~\ref{lmconvstrong}. That is, if we consider the operator norm instead of the trace class operator norm, the sums corresponding to the first six cases in the definition of $\hra$ converge, even if $ \alpha \geq \frac{1}{2}$. Consequently, the right-hand side of \eqref{eqmt} always converges in the operator norm topology, and the left-hand side of \eqref{eqmt} is always compact.
\end{proof}

\subsection{Trace of $\int_0^{+ \infty} h\p{t} \mathcal{M}_t \mathrm{d}t$ and proof of Proposition~\ref{proplocop}.}\label{subsectrace}

Before proving that $\p{\mathcal{L}_t}_{t \in \R}$ inherits of the properties of $\p{\mathcal{M}_t}_{t \in \R}$, thus showing Proposition~\ref{proplocop}, we still need to prove that the operator $\int_{\R} h\p{t} \mathcal{M}_t \mathrm{d}t$ has the expected trace, when it makes sense. This is the point of the following lemma.

\begin{lm}\label{lmtrace}
Under the hypotheses of Proposition~\ref{proplocop}, if $\Theta = \Theta'$,if $\alpha < \frac{1}{2}$ and if $k+1 > d+1$ then
\begin{equation*}
\textup{tr}\p{\s{\R}{h\p{t}\mathcal{M}_t}{t}} = \sum_{p \circ F\p{x} = x} \frac{h\p{T\p{x}}}{\va{\det \p{I - p \circ D_x F}}}\int_{\R} G_{T\p{x}}\p{x,y} \mathrm{d}y ,
\end{equation*}
where, for $x \in \R^d$, the number $T(x)$ is defined by $F(x) = p\p{F(x)} + \p{0 , - T(x)}$. 
\end{lm}

\begin{proof}
For all $N \in \N$ write
\begin{equation*}
A_N = \sum_{\substack{\p{n,i},\p{\ell,j} \in \Gamma \\ 0 \leq n,\ell \leq N}} \iota_{n,i} \circ H_{n,i}^{\ell,j} \circ \pi_{\ell,j}
\end{equation*}
and notice that \cite[Theorem 11.3 p.89]{Gohb} implies that
\begin{equation*}
\textup{tr}\p{\s{\R}{h\p{t}\mathcal{M}_t}{t}} = \lim_{N \to +\infty} \textup{tr}\p{A_N}.
\end{equation*}
Moreover, using Lidskii's trace theorem, we see that for all $N \in \N$ we have
\begin{equation*}
\textup{tr}\p{A_N} = \sum_{\substack{\p{n,i} \in \Gamma \\ 0 \leq n \leq N}} \textup{tr}\p{H_{n,i}^{n,i}}.
\end{equation*}
Now, from \eqref{eqopnoy}, we see that
\begin{equation*}
\begin{split}
\textup{tr}\p{H_{n,i}^{n,i}} & = \int_{\R} \int_K h\p{t}G_t\p{w} \tilde{\psi}_{\Theta,n,i}\p{D}\p{\mathbb{F}^{-1}\p{\psi_{\Theta,n,i}}\p{\cdot -w}}\p{\mathcal{T}_t\p{w}} \mathrm{d}w \mathrm{d}t \\
           & = \int_{\R} \int_K h\p{t}G_t\p{w} \mathbb{F}^{-1}\p{\psi_{\Theta,n,i}}\p{\mathcal{T}_t\p{w} - w} \mathrm{d}w \mathrm{d}t.
\end{split}
\end{equation*}
(We used in the second line that if $\psi_{\Theta,n,i}\p{\xi} \neq 0$ then $\tilde{\psi}_{\Theta,n,i}\p{\xi} = 1$). Now let $M$ be such that $K \subseteq \left[-M,M\right]^{d+1}$ and $h$ is supported in $\left[-M,M\right]$. Define the map $g : \R^{d+1} \simeq \R^d \times \R \to \R^{d+1}$ by $g\p{x,t} = F\p{x} - \p{x,-t}$. Notice that for all $ \p{x,y} \in \R^{d+1}$ and $t \in \R$ we have
\begin{equation*}
\mathcal{T}_t\p{x,y} - \p{x,y} = g\p{x,t}.
\end{equation*}
Cone-hyperbolicity implies that the Jacobian of $g$ does not vanish. Consequently we can find a finite family $\p{\rho_a}_{a\in \mathcal{A}}$ of compactly supported $\mathcal{C}^\infty$ functions $\rho_a : \R^{d+1} \to \left[0,1\right]$ such that $\sum_{a \in \mathcal{A}} \rho_a\p{w} = 1$ for all $w \in \left[-M,M\right]^{d+1}$ and for all $a \in \mathcal{A}$ there is a $\mathcal{C}^\infty$ diffeomorphism $g_a : \R^{d+1} \to \R^{d+1}$ that coincides with $g$ on a neighborhood of the support of $\rho_a$. Thus, we have (with $w=(x,y)$)
\begin{equation*}
\begin{split}
\textup{tr}\p{H_{n,i}^{n,i}} & = \sum_{a \in \mathcal{A}} \int_{\left[-M,M\right]} \Bigg( \int_{\left[-M,M\right]^{d+1} } h\p{t} \rho_a\p{x,t} G_t\p{x,y}  \\ 
             & \qquad \qquad \qquad \qquad \qquad \qquad \qquad \times \mathbb{F}^{-1}\p{\psi_{\Theta,n,i}}\p{g_a\p{x,t}} \mathrm{d}x\mathrm{d}t \Bigg) \mathrm{d}y \\
             & = \sum_{a \in \mathcal{A}} \int_{\left[-M,M\right]} \Bigg( \int_{\R^{d+1}} h\p{t_a\p{z}} \rho_a \circ g_a^{-1}\p{z} \frac{ G_{t_a\p{z}}\p{x_{a}\p{z},y}}{\va{\det D_{g_a^{-1}\p{z}} g_a}} \\
             & \qquad \qquad \qquad \qquad \qquad \qquad \qquad \qquad \quad \times \mathbb{F}^{-1}\p{\psi_{\Theta,n,i}}\p{z} \mathrm{d}z \Bigg) \mathrm{d}y 
\end{split}
\end{equation*}
where $g_a^{-1}(z)=\p{x_a(z),t_a(z)}$. Since $\sum_{\p{n,i} \in \Gamma} \psi_{\Theta,n,i}= 1$ we find that for all $a \in \mathcal{A}$ we have
\begin{equation*}
\begin{split}
\lim_{N \to + \infty} \sum_{\substack{\p{n,i} \in \Gamma \\ 0 \leq n \leq N}} & \int_{\left[-M,M\right]} \Bigg( \int_{\R^{d+1}} h\p{t_a\p{z}} \rho_a \circ g_a^{-1}\p{z} \frac{ G_{t_a\p{z}}\p{x_{a}\p{z},y}}{\va{\det D_{g_a^{-1}\p{z}} g_a}} \\ & \qquad \qquad \qquad \qquad \times \mathbb{F}^{-1}\p{\psi_{\Theta,n,i}}\p{z} \mathrm{d}z \Bigg) \mathrm{d}y \\
   = & \int_{\left[-M,M\right]} h\p{t_a\p{0}} \rho_a\p{g_a^{-1}\p{0}} \frac{G_{t_a\p{0}}\p{x_{a}\p{0},y}}{\va{\det D_{g_a^{-1}\p{0}} g_a}} \mathrm{d}y.
\end{split} 
\end{equation*}
And thus
\begin{equation*}
\textup{tr}\p{\s{\R}{h\p{t}\mathcal{M}_t}{t}} = \sum_{a \in \mathcal{A}} \int_{\left[-M,M\right]} h\p{t_a\p{0}} \rho_a\p{g_a^{-1}\p{0}} \frac{G_{t_a\p{0}}\p{x_{a}\p{0},y}}{\va{\det D_{g_a^{-1}\p{0}} g_a}} \mathrm{d}y.
\end{equation*}
Now, notice that $g\p{x,t} = 0$ if and only if $p\circ F\p{x} = x$ and $t = T\p{x}$, thus
\begin{equation*}
\begin{split}
& \int_{\left[-M,M\right]} h\p{t_a\p{0}} \rho_a\p{g_a^{-1}\p{0}} \frac{G_{t_a\p{0}}\p{x_{a}\p{0},y}}{\va{\det D_{g_a^{-1}\p{0}} g_a}} \mathrm{d}y \\
& \qquad \qquad = \sum_{p \circ F\p{x} = x} \int_{\left[-M,M\right]} \rho_a\p{x,T\p{x}} h\p{T\p{x}} \frac{G_{T\p{x}}\p{x,y}}{\va{\det \p{I - p \circ D_x F}}} \mathrm{d}y.
\end{split}
\end{equation*}
Here we noticed that the Jacobian of $g$ do not depend on the last coordinate. Finally, summing over $a \in \mathcal{A}$ we get
\begin{equation*}
\begin{split}
& \textup{tr}\p{\s{\R}{h\p{t}\mathcal{M}_t}{t}} \\ & \quad \quad =\sum_{p \circ F\p{x} = x} \frac{h\p{T\p{x}}}{\va{\det \p{I - p \circ D_x F}}}\int_{\left[-M,M\right]} G_{T\p{x}}\p{x,y} \mathrm{d}y.
\end{split}
\end{equation*}
\end{proof}

We show Proposition~\ref{proplocop} by proving that $\p{\mathcal{L}_t}_{t \in \R}$ also satisfies the properties established for $\p{\mathcal{M}_t}_{t \in \R}$ in Lemmas \ref{lmconvstrong}, \ref{lmschatten} and~\ref{lmtrace}.

\begin{proof}[Proof of Proposition~\ref{proplocop}]
Recall that $\mathcal{Q}_\Theta$ (defined by \eqref{eqdefQTheta}) induces an isomorphism between $\h_{\Theta,\alpha}$ and $\mathcal{Q}_\Theta\p{\h_{\Theta,\alpha}}$, which is a closed subspace of $\B$. We denote by $\mathcal{Q}_\Theta^{-1}$ the inverse isomorphism (and similarly for $\mathcal{Q}_{\Theta'}$).

Now, if $\p{u_{n,i}}_{\p{n,i} \in \Gamma}$ is finitely supported (i.e. there are finitely many $\p{n,i} \in \Gamma$ such that $u_{n,i} \neq 0$) and such that for all $\p{n,i} \in \Gamma$ we have $u_{n,i} \in S^{\tilde{\upsilon}}$ (for some $\tilde{\upsilon} \in \left]\upsilon,\frac{1}{1 - \alpha} \right[$) we write $u = \sum_{\p{n,i} \in \Gamma} \tilde{\psi}_{\Theta,n,i}\p{D} u_{n,i}$ and notice that $\mathcal{L}_t u \in S^{\tilde{\upsilon}}$ for all $t \in \R$, and thus $\mathcal{L}_t u \in \h_{\Theta',\alpha}$. Consequently, $\mathcal{M}_t \p{u_{n,i}}_{\p{n,i} \in \Gamma} = Q_{\Theta'} \mathcal{L}_t u$ is in $\mathcal{Q}_{\Theta'}\p{\h_{\Theta',\alpha}}$. Since such elements are easily seen to be dense in $\B$, it appears that $\mathcal{M}_t$ sends $\B$ into $\mathcal{Q}_{\Theta'}\p{\h_{\Theta',\alpha}}$. We can consequently define the operator $\mathcal{Q}_{\Theta'}^{-1} \circ \mathcal{M}_t \circ \mathcal{Q}_\Theta$.

The calculation above also implies that $\mathcal{L}_t$ and $\mathcal{Q}_{\Theta'}^{-1} \circ \mathcal{M}_t \circ \mathcal{Q}_\Theta$ coincides on $\h_{\Theta,\alpha}$ (since the element of $S^{\tilde{\upsilon}}$ whose Fourier transform is compactly supported are dense in $\h_{\Theta,\alpha}$, and $\h_{\Theta,\alpha}$ and $\h_{\Theta',\alpha}$ are continuously contained in $\p{\mathcal{S}^{\tilde{\upsilon}}}'$, see Remark \ref{rmqdense}). Now, since $\mathcal{L}_t : \h_{\Theta,\alpha} \to \h_{\Theta',\alpha}$ is conjugated to $\mathcal{M}_t : \B \to \B$ (the conjugacy being independent of $t$), it inherits of all the relevant properties of $\mathcal{M}_t$, which ends the proof of Proposition~\ref{proplocop} with Lemmas \ref{lmconvstrong}, \ref{lmschatten} and~\ref{lmtrace} (for the computation of the trace, use Lidskii's trace theorem and the fact that $\mathcal{M}_t$ sends $\B$ into $\mathcal{Q}_{\Theta'}\p{\h_{\Theta',\alpha}}$, and not only let this subspace stable). 
\end{proof}

\section{Global space: first step}\label{gs1}

We are now ready to start the proof of Theorem~\ref{thmmain} using the tools from \S \ref{sectls} and \S \ref{sectlto}. So let $M$ be a compact $d+1$-dimensional $C^{\kappa,\upsilon}$ manifold, let $\p{\phi^t}_{t \in \R}$ be a $\mathcal{C}^{\kappa,\upsilon}$ Anosov flow on $M$, and let $g : M \to \C$ a $\mathcal{C}^{\kappa,\upsilon}\p{M}$ potential. We fix $t_0 > 0$ from now on. We will construct in this section two auxiliary Hilbert spaces $\widetilde{\h}$ and $\widetilde{\h}_0$. The space $\widetilde{\h}_0$ almost satisfies the conclusions of Theorem~\ref{thmmain} (this is the point of Proposition~\ref{proprecap}) but the Koopman operator $\mathcal{L}_t$ from \eqref{eqkoopman} is bounded from $\widetilde{\h}_0$ to itself only for large values of $t$ \emph{a priori}. This problem will be settled in \S \ref{gs2}. The first thing that we need to do in order to construct the spaces $\widetilde{\h}$ and $\widetilde{\h}_0$ is to show that, locally in space and for large times, the action of the flow $\p{\phi^t}_{t \in \R}$ behaves like the local model that we studied in \S \ref{sectlto}, this is the point of Lemma~\ref{lmdecoupe}. Indeed, we construct in Lemma~\ref{lmdecoupe} a system of admissible charts adapted to the dynamics of $\p{\phi^t}_{t \in \R}$ (this is a continuous-time analogue of \cite[Lemma 8.1]{lagtf}). We can then glue copies of the local spaces from \S \ref{sectls} to define the global spaces $\widetilde{\h}$ and $\widetilde{\h}_0$. Finally, we state and prove Proposition~\ref{proprecap}.

\begin{lm}\label{lmdecoupe}
There are a finite set $\Omega$, an integer $r$ and $t_1 \in \left]0,t_0\right[$ such that:
\begin{enumerate}[label=(\roman*)]
\item there is no periodic orbit of $\p{\phi^t}_{t \in \R}$ of length less than $3t_1$;
\item for all $\omega \in \Omega$ there is a $\mathcal{C}^{\kappa,\upsilon}$ chart $\kappa_\omega : U_\omega \to V_\omega$, where $U_\omega$ is an open subset of $M$ and $V_\omega$ an open subset of $\R^{d+1}$, such that $V_\omega = W_\omega \times \left]-t_1,t_1\right[$ for some open subset $W_\omega$ of $\R^d$, and for all $x \in U_\omega : D_x \kappa_\omega \p{V\p{x}} = e_{d+1}$;
\item $\bigcup_{\omega \in \Omega} U_\omega = M$;
\item for all $\omega \in \Omega$, there is a system of $r+2$ cones $\Theta_\omega = \p{C_{0,\omega},\dots,C_{r,\omega},C_{f,\omega}}$ in $\R^{d+1}$ (with respect to the direction $e_{d+1}$);
\item for every $\omega,\omega' \in \Omega$ and $t \in \left[t_0,3t_0\right]$ there is a $\mathcal{C}^{\kappa,\upsilon}$ immersion $F_{\omega,\omega',t} : \R^{d} \to \R^{d+1}$ such that the associated family $\p{\mathcal{T}^{\omega,\omega',t}_{t'}}_{t' \in \R}$ (defined by \eqref{eqdeffamily}) is indeed a family of diffeomorphisms and is cone-hyperbolic from $\Theta_\omega$ to $\Theta_{\omega'}$;
\item for all $\omega,\omega' \in \Omega,t \in \left[t_0,3t_0\right]$ and $t' \in \left]-t_1,t_1\right[$, if $x \in U_\omega$ is such that $\phi^{t+t'}\p{x} \in U_{\omega'}$ then $\mathcal{T}^{\omega,\omega',t}_{t'} \circ \kappa_{\omega}\p{x} = \kappa_{\omega'} \circ \phi^{t+t'}\p{x}$.
\end{enumerate}
\end{lm}

\begin{proof}
Choose a Mather metric $\va{\cdot}_x$ on $M$ (see \cite{Mather}). This metric makes the splitting
\begin{equation}\label{eqdecomp}
T_x M = E_x^u \oplus E_x^s \oplus \R V\p{x} 
\end{equation}
orthogonal and is H\"older-continuous. Moreover, $\va{V\p{x}}_x = 1$ for all $x \in M$ and for all $t \geq 0$ we have $\n{\left. D_x \phi^{t} \right|_{E_x^s}} \leq \lambda^{-t}$ and $\n{\left. D_x \phi^{-t} \right|_{E_x^u}} \leq \lambda^{-t}$ (for the induced norm, $\lambda >1$).

Choose $\gamma >0$ such that 
\begin{equation*}
\lambda^{-2 t_0} \p{\gamma^2 + 1} < 1.
\end{equation*}
Then choose $\gamma_1 \in \left] \frac{1}{\gamma}, \frac{\lambda^{\frac{t_0}{2}}}{\gamma} \right[$ and define for all $i \geq 2$ the number $\gamma_i = \lambda^{-\frac{t_0\p{i-1}}{2} }\gamma_1$. Now choose $r$ large enough so that
\begin{equation*}
\frac{\lambda^{2 t_0}}{ 1 + \gamma_{r-1}^2} > 1.
\end{equation*}
Since $\gamma \gamma_1 \lambda^{-\frac{5 t_0}{8}} < 1$, we may choose $\tilde{\epsilon}_u>0$ and $\tilde{\epsilon}_s >0$ such that
\begin{equation*}
\tilde{\epsilon}_u > \lambda^{-\frac{t_0}{2}} \gamma_1 \tilde{\epsilon}_s \textrm{ and } \tilde{\epsilon}_s > \gamma \lambda^{-\frac{t_0}{8}} \tilde{\epsilon}_u.
\end{equation*}
and small enough so that
\begin{equation*}
\lambda^{-2 t_0} \p{\tilde{\epsilon}_s^2 + \gamma^2 + 1} < 1
\end{equation*}
and
\begin{equation*}
\frac{\lambda^{2 t_0}}{ 1 + \gamma_{r-1}^2 + \tilde{\epsilon}_u^2} > 1.
\end{equation*}
Finally, set $\epsilon_u = \lambda^{-\frac{t_0}{8}} \tilde{\epsilon}_u$ and $\epsilon_s = \lambda^{-\frac{t_0}{2}} \tilde{\epsilon}_s$.

Now, for all $x \in M$, if $\xi \in T_x M$ write $\xi = \xi_u + \xi_s + \xi_0$ the decomposition of $\xi$ with respect to \eqref{eqdecomp}, and define the cones $C_f\p{x}$ and $C_i\p{x}$, for $i \in \N$ by
\begin{equation*}
C_0\p{x} = \set{ \xi \in T_x M : \va{\xi_u}_x \leq \gamma \va{\xi_s}_x \textrm{ and } \va{\xi_0}_x \leq \tilde{\epsilon}_s \va{\xi_s}_x},  
\end{equation*}
\begin{equation}\label{eqdefCix}
C_i\p{x} = \set{ \xi \in T_x M : \va{\xi_s}_x \leq \gamma_i \va{\xi_u}_x \textrm{ and } \va{\xi_0}_x \leq \lambda^{- \frac{\p{i-1} t_0}{4}} \tilde{\epsilon}_u \va{\xi_u}_x}
\end{equation}
for $i \in \N^*$ and 
\begin{equation*}
C_f\p{x} = \set{\xi \in T_x M : \va{\xi_0}_x \geq \epsilon_s \va{\xi_s}_x \textrm{ and } \va{\xi_0}_x \geq \epsilon_u \va{\xi_u}_x}.
\end{equation*}
Notice that all these cones depend Hölder-continuously on $x$. We will see that our choice of parameter ensures that for all $x \in M, \Theta\p{x} = \left(C_0\p{x},\dots,C_r\p{x}, \right.$ $\left. C_f\p{x}\right)$ is a system of $r+2$ cones with respect to the direction $V\p{x}$. Indeed:
\begin{enumerate}[label=(\roman*)]
\item if $\xi \in T_xM \setminus C_f(x)$, since $\gamma \gamma_1 > 1$, we have either $\va{\xi_u}_x < \gamma \va{\xi_s}_x$ or $\va{\xi_s}_x < \gamma_1 \va{\xi_u}_x$. In the first case, either $\va{\xi_0}_x < \tilde{\epsilon}_s \va{\xi_s}_x$, in which case $\xi \in \bul{C_0\p{x}}$, or $\va{\xi_0}_x \geq \tilde{\epsilon}_s \va{\xi_s}_x$, which implies $\xi \in \bul{C_f\p{x}}$, since $\epsilon_s < \tilde{\epsilon}_s$ and $ \va{\xi_0}_x \geq \frac{\tilde{\epsilon}_s}{\gamma} \va{\xi_u}_x > \epsilon_u \va{\xi_u}_x$. Similarly, we can see that in the second case either $\xi \in \bul{C_1\p{x}}$ or $\xi \in \bul{C_f\p{x}}$;
\item if $\xi \in C_f\p{x}$ then $\va{\xi_0}_x \geq \frac{1}{\sqrt{1 + \epsilon_u^{-2} + \epsilon_s^{-2} }} \va{\xi}$, which implies that $C_f\p{x}$ is one dimensional;
\item if $\xi \in C_0\p{x}$ then $\va{\xi}_{x} \leq \sqrt{1 + \gamma^2 + \tilde{\epsilon}_s^2} \va{\xi_s}$ and thus $C_0\p{x}$ is $d_s$-dimensional, where $d_s$ is the dimension of $E_x^s$, for the same reason $C_i\p{x}$ is $d_u$-dimensional for $i \in \set{1,\dots,r}$;
\item $C_{i+1}\p{x} \Subset C_i\p{x}$ for $i \in \set{1,\dots,r-1}$ because $\gamma_{i+1} < \gamma_i$ and $\lambda^{-\frac{i t_0}{2}} \tilde{\epsilon}_u < \lambda^{-\frac{\p{i-1} t_0}{2}} \tilde{\epsilon}_u$;
\item $C_0\p{x} \cap C_2\p{x} = \set{0}$ because $\gamma \gamma_2 < 1$ and $C_f\p{x} \cap C_2\p{x} = \set{0}$ because $\frac{\lambda^{\frac{t_0}{4}} \epsilon_u}{\tilde{\epsilon}_u} = \lambda^{\frac{t_0}{8}} >1$.
\end{enumerate}

Our choice of parameter also ensures that for all $t \geq t_0$ and all $x \in M$ (with $\Lambda = \lambda^{t_0} \min\p{\p{\tilde{\epsilon}_u^2 + \gamma_{r-1}^2 + 1}^{-\frac{1}{2}}, \p{\tilde{\epsilon}_s^2 + \gamma^2 + 1}^{-\frac{1}{2}}} > 1$):
\begin{itemize}
\item for all $i \in \set{1,\dots,r}$ we have $ \p{D_x \phi^t}^{\textup{tr}} \p{C_i\p{\phi^t\p{x}}} \subseteq C_{i+4}\p{x}$ because $\lambda^{-2t} \gamma_i \leq \gamma_{i+4}$ and $\frac{\lambda^{-t}\lambda^{- \frac{\p{i-1}t_0}{4}} \tilde{\epsilon}_u}{\lambda^{- \frac{\p{i+3}t_0}{4} \tilde{\epsilon}_u}} < 1$;
\item $\p{D_x \phi^t}^{\textup{tr}} \p{C_f\p{\phi^t\p{x}}} \cap C_0\p{x} = \set{0}$ because $\epsilon_s \lambda^t > \tilde{\epsilon}_s$;
\item for all $ \xi \in C_{r-1}\p{\phi^t\p{x}}$ we have $\va{ \p{D_x \phi^t}^{\textup{tr}} \p{\xi}}_x \geq \Lambda \va{\xi}_{\phi^t\p{x}}$;
\item for all $\xi \in T_{\phi^t\p{x}}M$ such that $ \p{D_x \phi^t}^{\textup{tr}}\p{\xi} \in C_0\p{x}$ we have the inequality $\va{\p{D_x \phi^t}^{\textup{tr}}\p{\xi}}_x \leq \Lambda^{-1} \va{\xi}_{\phi^t\p{x}}$.
\end{itemize}

Then, for every $x \in M$, we may choose a $\mathcal{C}^{\kappa,\upsilon}$ chart $\kappa_x : V_x \to W_x = B\p{0,\delta_x} \times \left]-t_x,t_x\right[$ such that $\kappa_x\p{x} = 0$, the map $D_x \kappa_x: T_x M \to \R^{d+1}$ is an isometry and, for every $y \in V_x$, we have $D_y \kappa_x \p{V\p{y}} = e_{d+1}$ (we can require the last two points simultaneously because $\va{V\p{x}}_x = 1$). For every $x \in M$, choose a system of $r+2$ cones $\Theta_x = \p{C_{0,x},\dots,C_{r,x},C_{f,x}}$ such that $ D_x \kappa_x^{\textup{tr}}\p{C_{f,x}} \Subset C_f\p{x}$,$ D_x \kappa_x^{\textup{tr}}\p{C_{0,x}} \Subset C_0\p{x}$, and, for every $i \in \set{1,\dots,r}$, we have $\p{ D_x \kappa_x^{\textup{tr}}}^{-1}\p{C_{i+1}\p{x}} \Subset C_{i,x} \Subset \p{ D_x \kappa_x^{\textup{tr}}}^{-1}\p{C_i\p{x}}$. Here we recall that the definition \eqref{eqdefCix} of $C_i(x)$ is valid for any $i \geq 1$. Up to making $V_x$ smaller, we may ensure that for all $y \in V_x$ we have
\begin{equation}\label{eqinccone1}
 D_y \kappa_x^{\textup{tr}}\p{C_{f,x}} \Subset C_f\p{y},  D_y \kappa_x^{\textup{tr}}\p{C_{0,x}} \Subset C_0\p{y},
\end{equation}
for all $i \in \set{1,\dots,r}$ we have
\begin{equation}\label{eqinccone2}
\p{ D_y \kappa_x^{\textup{tr}}}^{-1}\p{C_{i+1}\p{y}} \Subset C_{i,x} \Subset \p{ D_y \kappa_x^{\textup{tr}}}^{-1}\p{C_i\p{y}},
\end{equation}
and, in addition,
\begin{equation}\label{eqquasiiso}
\n{D_y \kappa_x} \leq 1 + \epsilon \textrm{ and } \n{\p{D_y \kappa_x}^{-1}}^{-1} \geq 1 - \epsilon,
\end{equation}
where $\epsilon > 0$ is small enough so that
\begin{equation*}
\frac{1-\epsilon}{1 + \epsilon} \Lambda > 1.
\end{equation*}
By compactness of $M$, there are $x_1,\dots,x_n$ such that $M$ is covered by the open sets $\kappa_{x_i}^{-1}\p{B\p{0,\frac{\delta_{x_i}}{2}} \times \left]-\frac{t_{x_i}}{100},\frac{t_{x_i}}{100} \right[}$ for $i = 1,\dots,n$. Let $t_1 = \min_{i=1,\dots,n} \frac{t_{x_i}}{100}$. By cutting the charts into pieces and translating them, we may assume that for every $i = 1,\dots,n$ we have $t_{x_i} = 100 t_1$ (this could make us lose the fact that $D_{x_i} \kappa_{x_i}$ is an isometry, but this is of no harm since \eqref{eqquasiiso} remains true and that is all we need). Notice that for such a $t_1$ there is no periodic orbit of $\p{\phi^t}_{t \in \R}$ of length less than $3t_1$. If necessary, we reduce the value of $t_1$ so that $t_1 < t_0$. Set $t_2 = 30 t_1$ and let $N = \left\lceil \frac{2 t_0}{t_2} \right\rceil$. Choose $\chi : \R^d \to \left[0,1\right]$ Gevrey, compactly supported and such that $\chi\p{y} = 1 $ if $\va{y} \leq 1$.

If $i,j \in \set{1,\dots,n}$, if $k \in \set{0,\dots,N}$, and if $y \in \overline{B}\p{0,\frac{\delta_{x_i}}{2}}$ are such that the point $\phi^{t_0 + k t_2}\p{\kappa_{x_i}^{-1}\p{y,0}}$ lies in $\kappa_{x_j}^{-1}\p{B\p{0,\frac{\delta_{x_j}}{2}}} \times \left[-t_2,t_2\right] $, and $\eta > 0$ is small enough define $F_{i,j,k,y,\eta} : \R^d \to \R^{d+1} $ by (here we see $\R^d \simeq \R^d \times \set{0}$ as a subset of $\R^{d+1}$)
\begin{equation*}
\begin{split}
& F_{i,j,k,y,\eta}\p{z} = \chi\p{\frac{z-y}{\eta}} \kappa_{x_j} \circ \phi^{t_0 + k t_2} \circ \kappa_{x_i}^{-1}\p{z} \\ & \qquad \qquad \qquad \qquad  + \p{1 - \chi\p{\frac{z-y}{\eta}}} \times \bigg( \kappa_{x_j} \circ \phi^{t_0 + k t_2} \circ \kappa_{x_i}^{-1}\p{y} \\ & \qquad \qquad \qquad \qquad \qquad \qquad \qquad \qquad \qquad \qquad + D_y \p{\kappa_{x_j} \circ \phi^{t_0 + k t_2} \circ \kappa_{x_i}^{-1}}\p{z-y} \bigg)
\end{split}
\end{equation*}
Notice that $F_{i,j,k,y,\eta}$ coincides with $\kappa_{x_j} \circ \phi^{t_0 + k t_2} \circ \kappa_{x_i}^{-1}$ on $B\p{0,\eta}$, and that it can be made arbitrarily close in the $\mathcal{C}^1$ topology to the affine map $ z \mapsto \kappa_{x_j} \circ \phi^{t_0 + k t_2} \circ \kappa_{x_i}^{-1}\p{y} + D_y \p{\kappa_{x_j} \circ \phi^{t_0 + k t_2} \circ \kappa_{x_i}^{-1}}\p{z-y}$ by taking $\eta = \eta_{i,j,k,y}$ small enough. In particular, $F_{i,j,k,y,\eta}$ defines a cone-hyperbolic family of diffeomorphisms $\p{\mathcal{T}_{i,j,k,y,\eta,t'}}_{t' \in \R}$  from $\Theta_{x_i}$ to $\Theta_{x_j}$ (the cone-hyperbolicity follows from the properties of the differential of $\phi^{t_0 + k t_2}$ proven above and the quasi-isometry property \eqref{eqquasiiso} of the charts, to see that the $\mathcal{T}_{i,j,k,y,\eta,t'}$'s are diffeomorphisms just notice that they are proper local diffeomorphism and hence covering of $\R^{d+1}$ by itself). Define $\tilde{\eta}_{i,y} = \min_{\substack{k=0,\dots,N \\ j = 1,\dots,n}} \eta_{i,j,k,y}$ (if $j$ is such that $\phi^{t_0 + k t_2}\p{\kappa_{x_i}^{-1}\p{y,0}} \notin \kappa_{x_j}^{-1}\p{B\p{0,\frac{\delta_{x_j}}{2}}} \times \left[-t_2,t_2\right] $, i.e. there is no allowed transitions from $i$ to $j$ at the considered time, set $\eta_{i,j,k,y} = \infty$ and take for $F_{i,j,k,y, \tilde{\eta}_{i,y}}$ any $\mathcal{C}^{\kappa,\upsilon}$ map that defines a cone-hyperbolic family of diffeomorphisms\footnote{There always is a linear such map.} from $\Theta_{x_i}$ to $\Theta_{x_j}$) .

Notice also that for all $\p{z,z'} \in B\p{0,\tilde{\eta}_{i,y}} \times \left]-t_2,t_2\right[$ and all $t,t' \in \left]-t_2,t_2\right[$ we have
\begin{equation*}
\begin{split}
\kappa_{x_j} \circ \phi^{t_0 + k t_2 + t + t'} \circ \kappa_{x_i}^{-1} \p{z,z'} & =\kappa_{x_j} \circ \phi^{t_0 + k t_2 + t + t' + z'} \circ \kappa_{x_i}^{-1} \p{z,0} \\
    & = F_{i,j,k,y,\tilde{\eta}_{i,y}}\p{z} + z' e_{d+1} + \p{t+t'} e_{d+1}\\
    & = \mathcal{T}^{i,j,k,y,t}_{t'} \p{z,z'},
\end{split}
\end{equation*}
where $\p{\mathcal{T}^{i,j,k,y,t}_{t'}}_{t' \in \R}$ denotes the family of cone-hyperbolic diffeomorphisms associated with $F_{i,j,k,y,\tilde{\eta}_{i,y}} + t e_{d+1}$.

By compactness of $\overline{B}\p{0,\frac{\delta_{x_i}}{2}}$, we may find $y_{i,1},\dots ,y_{i,m_i} \in \overline{B}\p{0,\frac{\delta_{x_i}}{2}}$ such that $\overline{B}\p{0,\frac{\delta_{x_i}}{2}} \subseteq \bigcup_{\ell=1}^{m_i} B\p{y_{i,\ell}, \frac{\tilde{\eta}_{i,y_{i,\ell}}}{2}}$. 

Finally, set
\begin{equation*}
\Omega = \set{\p{i,\ell} : i \in \set{1,\dots,n}, \ell \in \set{1,\dots,m_i}},
\end{equation*}
and, for all $\omega = \p{i,\ell} \in \Omega$,
\begin{equation*}
\begin{split}
V_\omega = B\p{0,\frac{\tilde{\eta}_{i, y_{i,\ell}}}{2}} \times \left]-t_1,t_1\right[, & \quad U_\omega = \kappa_{x_i}^{-1}\p{B\p{0,\frac{\tilde{\eta}_{i, y_{i,\ell}}}{2}} \times \left]-t_1,t_1\right[},\\ \kappa_\omega = \left. \kappa_{x_i} \right|_{U_\omega},& \quad \Theta_{\omega} = \Theta_{x_i}.
\end{split}
\end{equation*}
If $\omega' = \p{j,\ell'} \in \Omega$ and $t \in \left[t_0 + kt_2,t_0 + \p{k+1} t_2\right]$ let
\begin{equation*}
F_{\omega,\omega',t} = F_{i,j,k,y_{i,\ell},\tilde{\eta}_{i,y_{i,\ell}}} + \p{t-k t_2} e_{d+1}.
\end{equation*}
\end{proof}

Choose a Gevrey partition of unity $\p{\varphi_{\omega}}_{\omega \in \Omega}$ subordinated to the open cover $\p{U_\omega}_{\omega \in \Omega}$. Fix $\alpha \in \left] \frac{\upsilon- 1}{\upsilon},1\right[$  (if $\upsilon < 2$, we choose $\alpha < \frac{1}{2}$) and choose $\tilde{\upsilon} \in \left]\upsilon,\frac{1}{1 -\alpha}\right[$. Then define
\begin{equation*}
\begin{array}{ccccc}
\Phi & : & \mathcal{D}^{\tilde{\upsilon}}\p{M}' & \to & \oplus_{\omega \in \Omega} \p{\mathcal{S}^{\tilde{\upsilon}}}' \\
 & & u & \mapsto & \p{\p{\varphi_\omega u}\circ \kappa_w^{-1}}
\end{array}
\end{equation*}
and
\begin{equation*}
\begin{array}{ccccc}
S & : & \oplus_{\omega \in \Omega} \p{\mathcal{S}^{\tilde{\upsilon}}}' & \to & \mathcal{D}^{\tilde{\upsilon}}\p{M}' \\
 & & \p{u_\omega}_{\omega \in \Omega} & \mapsto & \sum_{\omega \in \Omega} \p{h_\omega u_\omega} \circ \kappa_\omega,
\end{array}
\end{equation*}
where $h_\omega : \R^{d+1} \to \left[0,1\right]$ is Gevrey, supported in $W_\omega$, and takes value $1$ on $\kappa_\omega\p{\textup{supp } \varphi_\omega}$. Notice that $S \circ \Phi$ is the identity of $\mathcal{D}^{\tilde{\upsilon}}\p{M}'$. It can be verified that $\Phi$ and $S$ are continuous.

We may now define the first version of the global Hilbert space (the final one will be introduced in \S \ref{gs2}). Define
\begin{equation*}
\h_{\Omega} = \oplus_{\omega \in \Omega} \h_{\Theta_\omega,\alpha}
\end{equation*}
and
\begin{equation*}
\widetilde{\h} = \set{ u \in \mathcal{D}^{\tilde{\upsilon}}\p{M}' : \Phi\p{u} \in \h_{\Omega}},
\end{equation*}
endowed with the norm
\begin{equation*}
\n{u}_{\widetilde{\h}} = \n{\Phi\p{u}}_{\h_{\Omega}} = \sqrt{\sum_{\omega \in \Omega} \n{\p{\varphi_\omega u} \circ \kappa_\omega^{-1}}_{\h_{\Theta_\omega,\alpha}}^2}.
\end{equation*}

\begin{prop}
$\widetilde{\h}$ is a separable Hilbert space (equivalently, $\Phi\p{\widetilde{\h}}$ is closed in $\h_{\Omega}$) that does not depend on the choice of $\tilde{\upsilon}$. The inclusion of $\widetilde{\h}$ in $\mathcal{D}^{\tilde{\upsilon}}\p{M}'$ is continuous, and $\mathcal{C}^{\infty,\tilde{\upsilon}}\p{M}$ is continuously contained in $\widetilde{\h}$.
\end{prop}

\begin{proof}
To see that $\Phi\p{\widetilde{\h}}$ is closed in $\h_\Omega$, just notice that
\begin{equation*}
\Phi\p{\widetilde{\h}} = \set{ u \in \h_\Omega : \Phi S u = u},
\end{equation*}
and that the inclusion of $\h_\Omega $ in $\bigoplus_{\omega \in \Omega} \p{\mathcal{S}^{\tilde{\upsilon}}}'$ is continuous. The inclusion of $\widetilde{\h}$ in $\mathcal{D}^{\tilde{\upsilon}}\p{M}'$ may be written as the composition of $\Phi$, the inclusion of $\h_\Omega$ in $\bigoplus_{\omega \in \Omega} \p{\mathcal{S}^{\tilde{\upsilon}}}'$ and $S$. It is thus continuous. Finally, $\Phi$ sends $\mathcal{C}^{\infty,\tilde{\upsilon}}\p{M}$ continuously into $\bigoplus_{\omega \in \Omega} \mathcal{S}^{\tilde{\upsilon}}$, which is continuously contained in $\h_{\Omega}$, thus $\mathcal{C}^{\infty,\tilde{\upsilon}}\p{M}$ is contained in $\widetilde{\h}$, the inclusion being continuous.
\end{proof}

Let $\widetilde{\h}_0$ be the closure\footnote{It could well be that $\widetilde{\h}_0 = \widetilde{\h}$, see Proposition~\ref{propbaseh}, but we do not need this fact} of $\mathcal{C}^{\infty,\tilde{\upsilon}}\p{M}$ in $\widetilde{\h}$. Recall from \S \ref{secDCC} that for each $t \in \R$ we may define the operator $\mathcal{L}_t$ from \eqref{eqkoopman} as an operator from $\mathcal{D}^{\tilde{\upsilon}}\p{M}'$ to itself. We start by proving that, for $t \geq t_0$, the operator $\mathcal{L}_t$ is bounded from $\widetilde{\h}$ to $\widetilde{\h}_0$.

\begin{prop}\label{propLtbounded}
For all $t \in \left[t_0,+\infty \right[$ the operator $\mathcal{L}_t$ is bounded from $\widetilde{\h}$ to $\widetilde{\h}_0$. Moreover,as an operator from $\widetilde{\h}$ to $\widetilde{\h}_0$, the operator $\mathcal{L}_t$ depends continuously on $t \in \left[t_0,+\infty\right[$ in the strong operator topology.
\end{prop}

\begin{proof}
We only need to prove the result for $t \in \left[t_0,3t_0\right]$, and then use the group property of $\p{\mathcal{L}_t}_{t \in \R}$. Recall indeed that $\widetilde{\h}_0$ is a closed subspace of $\widetilde{\h}$. For all $t \in \R$ define 
\begin{equation*}
\widetilde{\mathcal{L}}_t : \bigoplus_{\omega \in \Omega} \p{\mathcal{S}^{\tilde{\upsilon}}}' \to \bigoplus_{\omega \in \Omega} \p{\mathcal{S}^{\tilde{\upsilon}}}'
\end{equation*}
by $\widetilde{\mathcal{L}}_t = \Phi \circ \mathcal{L}_t \circ S$. The operator $\widetilde{\mathcal{L}}_t$ may be described via a matrix of operators $\p{\widetilde{\mathcal{L}}_{\omega,\omega',t}}_{\omega,\omega' \in \Omega}$, that is, we have
\begin{equation}\label{eqmatrix}
\widetilde{\mathcal{L}}_t \p{u_{\omega}}_{\omega \in \Omega} = \p{\sum_{\omega' \in \Omega} \widetilde{\mathcal{L}}_{\omega,\omega',t} u_{\omega'}}_{\omega \in \Omega}.
\end{equation}
Now, if $t \in \left[t_0,3t_0\right]$ and $t' \in \left]-t_1,t_1\right[$, then the operator $\widetilde{\mathcal{L}}_{\omega,\omega',t+t'}$ for $\omega,\omega' \in \Omega$ may be described as
\begin{equation}\label{eqecrloc}
\begin{split}
\widetilde{\mathcal{L}}_{\omega,\omega',t+t'} u\p{x} & = \varphi_\omega \circ \kappa_\omega^{-1}\p{x} e^{\int_0^{t+t'} g \circ \phi^\tau \p{\kappa_\omega^{-1}\p{x}} \mathrm{d}\tau} \\ & \quad \times h_{\omega'} \circ \kappa_{\omega'} \circ \phi^{t+t'} \circ \kappa_\omega^{-1}\p{x} u \circ \kappa_{\omega'} \circ \phi^{t+t'} \circ \kappa_\omega^{-1}\p{x} \\
 & = G_{\omega,\omega',t,t'}\p{x} u \circ \mathcal{T}^{\omega',\omega,t}_{t'}\p{x},
\end{split}
\end{equation}
where $\p{\mathcal{T}^{\omega',\omega,t}_{t'}}_{t' \in \R}$ is the family of diffeomorphisms associated to $F_{\omega',\omega,t}$ by \eqref{eqdeffamily}, and
\begin{equation*}
G_{\omega,\omega',t,t'}\p{x} = \varphi_\omega \circ \kappa_\omega^{-1}\p{x} e^{\int_0^{t+t'} g \circ \phi^{\tau} \p{\kappa_\omega^{-1}\p{x}} \mathrm{d}\tau} h_{\omega'} \circ \kappa_{\omega'} \circ \phi^{t+t'} \circ \kappa_\omega^{-1}\p{x}
\end{equation*}
properly extended by zero. We can then apply Proposition~\ref{proplocop} to prove that $\widetilde{\mathcal{L}}_{\omega,\omega',t+t'}$ is bounded from $\h_{\Theta_\omega',\alpha}$ to $\h_{\Theta_\omega,\alpha}$. Then $\widetilde{\mathcal{L}}_{t+t'}$ is bounded from $\h_\Omega$ to itself thanks to \eqref{eqmatrix}. Notice that if $u \in \bigoplus_{\omega \in \Omega} S^{\tilde{\upsilon}}$ then $\widetilde{\mathcal{L}}_{t+t'} u = \Phi\p{\mathcal{L}_{t+t'} \circ S u} \in \Phi\p{\widetilde{\h}_0}$. Thus, since $\bigoplus_{\omega \in \Omega} S^{\tilde{\upsilon}}$ is dense in $\h_\Omega$, the operator $\widetilde{\mathcal{L}}_{t+t'}$ sends $\h_\Omega$ into $\Phi\p{\widetilde{\h}_0}$. Denote by $\Psi : \Phi\p{\widetilde{\h}_0}  \to \widetilde{\h}_0$ the inverse of the isomorphism induced by $\Phi$, and notice that $\mathcal{L}_{t+t'}$ coincide on $\widetilde{\h}$ with $\Psi \circ \widetilde{\mathcal{L}}_{t+t'} \circ \Phi$, and is thus bounded from $\widetilde{\h}$ to $\widetilde{\h}_0$. Finally, from Proposition~\ref{proplocop}, we know that $\widetilde{\mathcal{L}}_{t+t'} : \h_{\Omega} \to \h_{\Omega}$ depends continuously on $t' \in \left]-t_1,t_1\right[$ in the strong operator topology, and consequently so does $\mathcal{L}_{t+t'} : \widetilde{\h} \to \widetilde{\h}_0$.
\end{proof}

We want now to prove Schatten properties for operators defined in term of the $\mathcal{L}_t$'s for $t \geq t_0$. To do so, it is convenient to introduce $\p{\psi_\ell}_{\ell \in \Z}$, a $\frac{t_1}{3} \Z$ invariant smooth partition of unity on $\R$ (that is, we have $\psi_\ell = \psi_0\p{\cdot - \ell \frac{t_1}{3}}$) such that $\psi_0$ is supported in $\left]-\frac{t_1}{2},\frac{t_1}{2}\right[$.

\begin{prop}\label{proprecap}
Assume $\upsilon < 2$. There is $\varpi \in \R$ with the following property: if $h : \R_+^* \to \C$ and $k \in \N$ satisfy
\begin{enumerate}[label = (\roman*)]
\item $h$ is supported in $\left[t_0,+ \infty\right[$;
\item $h$ is $k$th time differentiable and its $k$th derivatives has bounded variations;
\item there is a constant $C > 0$ such that for every $\ell \in \N$ we have
\begin{equation*}
\begin{split}
\n{\psi_0 h\p{\cdot + \ell \frac{t_1}{3}}}_{\mathcal{C}^{k-1}} + \n{\p{\psi_0 h\p{\cdot + \ell \frac{t_1}{3}}}^{(k)}}_{BV} \leq C e^{- \varpi \ell};
\end{split}
\end{equation*}
\end{enumerate}
then the operator
\begin{equation}\label{eqoperateursynthetique}
\begin{split}
 \int_{0}^{+ \infty} h(t) \mathcal{L}_t \mathrm{d}t : \widetilde{\h} \to \widetilde{\h}_0
\end{split}
\end{equation}
is in the Schatten class $S_p$ for every $p \geq 1$ such that $p > \frac{d+1}{k+1}$. Moreover, if $k > d$ and if we see \eqref{eqoperateursynthetique} as an operator from $\widetilde{\h}_0$ to itself, we have
\begin{equation*}
\begin{split}
\textup{tr}\p{\int_{t_0}^{+ \infty} h\p{t} \mathcal{L}_t \mathrm{d}t} = \sum_{\gamma} T_\gamma^{\#} \frac{h\p{T_\gamma}}{\va{\det\p{I - \mathcal{P}_\gamma}}} \exp\p{\int_\gamma g},
\end{split}
\end{equation*}
where the sum on the right-hand side runs over closed periodic orbits\footnote{Recall that $T_\gamma$ is the length of $\gamma$, while $T_\gamma^\#$ denotes its primitive length and $\mathcal{P}_\gamma$ is a linearized Poincaré map. We will see during the proof of the proposition that this sum converges.} $\gamma$ of the flow $\p{\phi^t}_{t \in \R}$.

Finally, if $\upsilon \geq 2$, it remains true (under the same assumptions) that the operator \eqref{eqoperateursynthetique} is compact.
\end{prop}

\begin{rmq}
Notice that if $h$ is a $\mathcal{C}^\infty$ function supported in $\left[t_0,+\infty\right[$, then $h$ clearly satisfies the conditions (i)-(iii) from Proposition \ref{proprecap}. This will be the main application of Proposition \ref{proprecap} in order to prove the trace formula (see Lemma \ref{lmlaplace}). However, we will also need to consider other functions $h$ in the proof of Lemma \ref{lmdiscretespectrum} and in the Appendix \ref{appdet}.
\end{rmq}

For the sake of the proof, we split Proposition \ref{proprecap} into Lemmas \ref{lmsumtr}, \ref{lmtrper} and~\ref{lmtracegs1}.

\begin{lm}\label{lmsumtr}
Under the assumption of Proposition \ref{proprecap}, the operator \eqref{eqoperateursynthetique} is in the Schatten class $S_p$ for every $p \geq 1$ such that $p > \frac{d+1}{k+1}$. If $\upsilon \geq 2$, it remains true that \eqref{eqoperateursynthetique} is compact.
\end{lm}

\begin{proof}
Let $p \geq 1$ be such that $p > \frac{d+1}{k+1}$. Choose $N$ large enough so that $N t_1 \geq t_0$, and write for all $\ell \geq N$
\begin{equation*}
\int_\R \psi_{\ell}\p{t} h\p{t} \mathcal{L}_t \mathrm{d}t = \mathcal{L}_{\frac{N t_1}{3}}^q \int_\R \psi_0\p{t} h\p{t + \ell \frac{t_1}{3}} \mathcal{L}_{\frac{r t_1}{3} + t} \mathrm{d}t,
\end{equation*}
where $\ell=qN+r$ with $q,r \in \Z$ and $N \leq r < 2N$ (notice that $q \geq 0$). Applying Proposition~\ref{proplocop} as in the proof of Proposition \ref{propLtbounded}, we see that the operator 
\begin{equation}\label{eqpiecedebase}
\int_\R \psi_0\p{t} h\p{t + \ell \frac{t_1}{3}} \mathcal{L}_{\frac{r t_1}{3} + t} \mathrm{d}t
\end{equation}
is in the Schatten class $S_p$, with norm in this class an $\mathcal{O}\p{e^{- \varpi \ell}}$ (there is a finite number of possible values for $r$). Thus the sum
\begin{equation}\label{eqsumtr}
\sum_{\ell \in \Z} \int_\R \psi_\ell\p{t} h\p{t} \mathcal{L}_t \mathrm{d}t
\end{equation}
converges in $S_p$ provided that $\varpi$ is large enough(there are a finite number of non-zero terms with $k < N$ that are also in $S_p$ thanks to Proposition~\ref{proplocop} since $h\p{t}$ vanishes for $t \leq t_0$). Now notice that, for every $t \in \R$, the sum
\begin{equation*}
\sum_{\ell \in \Z} \psi_\ell\p{t} h\p{t} \mathcal{L}_t
\end{equation*}
converges in operator norm topology to $h\p{t} \mathcal{L}_t$, and the convergence is uniform in $t$ (provided that $\varpi$ is large enough), so that the sum \eqref{eqsumtr} is in fact the operator
\begin{equation*}
\int_{0}^{+ \infty} h(t) \mathcal{L}_t \mathrm{d}t,
\end{equation*}
which is consequently in $S_p$.

Finally, when $\upsilon \geq 2$ it remains true that the operator \eqref{eqpiecedebase} is compact according to Proposition \ref{proplocop}, and the convergence in the operator norm topology ensures that \eqref{eqoperateursynthetique} is compact.
\end{proof}

We need now to compute the trace of this operator when $k > d$ and $\upsilon < 2$. We will deduce the global formula for the trace from the local formula from Proposition~\ref{proplocop}.

\begin{lm}\label{lmtrper}
Under the assumptions of Proposition \ref{proprecap} and if $\ell \in \Z$ we have
\begin{equation}\label{eqtrper}
\textup{tr}\p{\int_\R \psi_{\ell}\p{t} h\p{t} \mathcal{L}_t \mathrm{d}t} = \sum_{\gamma} \psi_{\ell}\p{T_\gamma} \exp\p{\int_\gamma g} \frac{h\p{T_\gamma} T_\gamma^\#}{\va{\det\p{I - \mathcal{P}_\gamma}}},
\end{equation}
where the sum runs over periodic orbits $\gamma$ of the flow $\p{\phi^t}_{t \in \R}$. Here, we recall that $\p{\psi_\ell}_{\ell \in \Z}$ is a $\frac{t_1}{3} \Z$ invariant smooth partition of unity on $\R$ such that $\psi_0$ is supported in $\left]-\frac{t_1}{2},\frac{t_1}{2}\right[$.
\end{lm}

\begin{proof}
If $\ell$ is such that $\frac{\ell t_1}{3} < t_0 - \frac{t_1}{2}$ then \eqref{eqtrper} is immediate: both sides vanish. Otherwise, choose an integer $m \geq 0$ such that $\frac{\ell t_1}{3} - m t_0 \in \left[t_0 -\frac{t_1}{2}, 2t_0 \right]$ (one can for instance take $m$ to be the largest integer such that $\frac{\ell t_1}{3} - m t_0 \geq t_0 - \frac{t_1}{2}$) and define $t_3 = \max\p{t_0,\frac{\ell t_1}{3} - m t_0}$. This ensures that the support of $\psi_{\ell}$ is contained in $m t_0 + t_3 + \left]-t_1,t_1\right[$ and that $t_3 \in \left[t_0,2t_0\right]$. For all $\overrightarrow{\omega} = \p{\omega_1,\dots,\omega_{m}} \in \Omega^{m}$ define 
\begin{equation*}
U_{\overrightarrow{\omega}} = \bigcap_{j = 1}^{m} \phi^{-j t_0}\p{U_{\omega_i}}.
\end{equation*} 
Then choose a refinement $\p{\widetilde{U}_{\overrightarrow{\omega},w}}_{\p{\overrightarrow{\omega},w} \in \Omega^{m} \times \mathcal{W}}$ of $\p{U_{\overrightarrow{\omega}}}_{\omega \in \Omega^{m}}$ whose elements are small enough such that, if $\gamma$ is a periodic orbit of $\p{\phi^{t}}_{t \in \R}$ of length $T_\gamma$ less than $t_3 + m t_0 + t_1$, and $\p{\overrightarrow{\omega},w} \in \Omega^{m} \times \mathcal{W}$, then the intersection of $\gamma$ with $\widetilde{U}_{\overrightarrow{\omega},w}$ is an interval (i.e. connected, while possibly empty). This can be done because there are a finite number of such orbits. Choose a Gevrey partition of unity $\p{\theta_{\overrightarrow{\omega},w}}_{\p{\overrightarrow{\omega},w} \in \Omega^{m} \times \mathcal{W}}$ adapted to the open cover $\p{\widetilde{U}_{\overrightarrow{\omega},w}}_{\p{\overrightarrow{\omega},w} \in \Omega^{m}\times \mathcal{W}}$ of $M$. For $t \in t_3 + m t_0 + \left]-t_1,t_1\right[$, recall from the proof of Proposition \ref{propLtbounded} the operators $\widetilde{\mathcal{L}}_t = \Phi \circ \mathcal{L}_t \circ S$, and $\widetilde{\mathcal{L}}_{\omega,\omega',t}$, for $\omega,\omega' \in \Omega$, defined by the formula,
\begin{equation*}
\widetilde{\mathcal{L}}_{\omega,\omega',t} u\p{x} = \varphi_\omega \circ \kappa_\omega^{-1}\p{x} e^{\int_0^t g \circ \phi^\tau \p{\kappa_\omega^{-1}\p{x}} \mathrm{d}\tau} h_{\omega'} \circ \kappa_{\omega'} \circ \phi^t \circ \kappa_\omega^{-1}\p{x} u \circ \kappa_{\omega'} \circ \phi^t \circ \kappa_\omega^{-1}\p{x}.
\end{equation*}
Then write $\widetilde{\mathcal{L}}_{\omega,\omega',t}$ as a sum of operators
\begin{equation*}
\widetilde{\mathcal{L}}_{\omega,\omega',t} = \sum_{\p{\overrightarrow{\omega},w} \in \Omega^{m} \times \mathcal{W}} A_{\omega,\omega',\overrightarrow{\omega},w,t}
\end{equation*}
where, for $\overrightarrow{\omega} = \p{\omega_0,\dots,\omega_{m-1}}$ and $w \in \mathcal{W}$,
\begin{equation*}
\begin{split}
& A_{\omega,\omega',\overrightarrow{\omega},w,t}u \p{x} = \theta_{\overrightarrow{\omega},w}\p{\kappa_\omega^{-1}\p{x}} \widetilde{\mathcal{L}}_{\omega,\omega',t}u \p{x} \\
& \qquad \qquad = \p{\theta_{\overrightarrow{\omega},w} \varphi_\omega} \circ \kappa_{\omega}^{-1}\p{x} e^{\int_0^t g \circ \phi^\tau \p{\kappa_\omega^{-1}\p{x}} \mathrm{d}\tau} h_{\omega'} \circ \kappa_{\omega'} \circ \phi^{t} \circ \kappa_{\omega}^{-1} \p{x} \\
& \qquad \qquad \qquad \times u \circ \mathcal{T}^{\omega_{m},\omega', t_3}_{t -t_3- m t_0} \circ \mathcal{T}^{\omega_{m - 1},\omega_{m}, t_0}_{0} \circ \dots \circ \mathcal{T}^{\omega_1,\omega_2,t_0}_{0} \circ \mathcal{T}^{\omega,\omega_1,t_0}_{0}\p{x}.
\end{split}
\end{equation*}
Consequently, we can use Proposition~\ref{proplocop} to see that $A_{\omega,\omega',\overrightarrow{\omega},w,t} : \h_{\Theta_{\omega'},\alpha} \to \h_{\Theta_\omega,\alpha}$ is bounded (here, we recall that $\alpha$ has been fiwed after the proof of Lemma \ref{lmdecoupe}, when defining the space $\widetilde{\h}$). Then, working as in the proof of Lemma~\ref{lmsumtr}, the operator
\begin{equation}\label{eqtilde}
\int_{\R} \psi_{\ell}\p{t} h\p{t} \widetilde{\mathcal{L}}_t \mathrm{d}t
\end{equation}
is trace class, sends $\h_{\Omega}$ into $\Phi\p{\widetilde{\h}_0}$ and the induced operator is conjugated to the operator defined by \eqref{eqtilde} without the tilde. Consequently, using Lidskii's trace theorem, we get
\begin{equation*}
\begin{split}
\textup{tr}\Bigg( \int_\R \psi_{\ell}\p{t} &  h\p{t} \mathcal{L}_t \mathrm{d}t \Bigg) = \textup{tr}\p{\int_\R \psi_{\ell}\p{t} h\p{t} \widetilde{\mathcal{L}}_t \mathrm{d}t} \\
& = \sum_{\omega \in \Omega} \textup{tr}\p{\int_\R \psi_{\ell}\p{t} h\p{t} \widetilde{\mathcal{L}}_{\omega,\omega,t} \mathrm{d}t} \\ 
& = \sum_{\omega \in \Omega} \sum_{\p{\overrightarrow{\omega},w} \in \Omega^{m} \times \mathcal{W}} \textup{tr}\p{\int_\R \psi_{\ell}\p{t} h\p{t} A_{\omega,\omega,\overrightarrow{\omega},w,t} \mathrm{d}t}. \\
\end{split}
\end{equation*}
Next, we fix $\omega$ and $\p{\overrightarrow{\omega},w}$ and we will compute
\begin{equation*}
\textup{tr}\p{\int_\R \psi_{\ell}\p{t} h\p{t} \mathcal{A}_{\omega,\omega,\overrightarrow{\omega},w,t} \mathrm{d}t}
\end{equation*}
using Proposition~\ref{proplocop}. To do so, recall the family of cone-hyperbolic diffeomorphisms
\begin{equation*}
\p{\mathcal{T}^{\omega,\overrightarrow{\omega}}_{t}}_{t \in \R} \coloneqq \p{\mathcal{T}^{\omega_{m},\omega', t_3}_{t -t_3-m t_0} \circ \mathcal{T}^{\omega_{m - 1},\omega_{m}, t_0}_{0} \circ \dots \circ \mathcal{T}^{\omega_1,\omega_2,t_0}_{0} \circ \mathcal{T}^{\omega,\omega_1,t_0}_{0}}_{t \in \R}
\end{equation*}
and denote by $F_{\omega,\overrightarrow{\omega}} : \R^d \to \R^{d+1}$ the associated immersion. By Proposition~\ref{proplocop}, we have
\begin{equation*}
\begin{split}
& \textup{tr}\p{\int_\R \psi_{\ell}\p{t} h\p{t} \mathcal{A}_{\omega,\omega,\overrightarrow{\omega},w,t} \mathrm{d}t} \\& \quad  = \sum_{p \circ F_{\omega,\overrightarrow{\omega}}\p{x} = x} \Bigg( \frac{h\p{T_{\omega,\overrightarrow{\omega}}\p{x}}\psi_{\ell}\p{T_{\omega,\overrightarrow{\omega}}\p{x}}}{\va{\det\p{I - p \circ D_x F_{\omega,\overrightarrow{\omega}}}}} \int_{\R} G_{\omega,\overrightarrow{\omega},w,T_{\omega,\overrightarrow{\omega}}\p{x}}\p{x,y} \mathrm{d}y \Bigg),
\end{split}
\end{equation*}
where, as in Proposition~\ref{proplocop}, if $x \in \R^d$, then $T_{\omega,\overrightarrow{\omega}}\p{x}$ denotes the opposite of the last coordinate of $F_{\omega,\overrightarrow{\omega}}\p{x}$, and
\begin{equation*}
\begin{split}
G_{\omega,\overrightarrow{\omega},w,t}\p{x} &  = \p{\theta_{\overrightarrow{\omega},w} \varphi_\omega} \circ \kappa_{\omega}^{-1}\p{x} e^{\int_0^t g \circ \phi^\tau \p{\kappa_\omega^{-1}\p{x}} \mathrm{d}\tau} \\ & \quad h_{\omega} \circ \kappa_{\omega} \circ \phi^{t} \circ \kappa_{\omega}^{-1}\p{x},
\end{split}
\end{equation*}
properly extended by zero.

Now, denote by $P$ the (finite) set of $x \in \R^d$ such that $p \circ F_{\omega,\overrightarrow{\omega}}\p{x} = x$ and
\begin{equation*}
\begin{split}
& D\p{x} \coloneqq \frac{h\p{T_{\omega,\overrightarrow{\omega}}\p{x}}\psi_{\ell}\p{T_{\omega,\overrightarrow{\omega}}\p{x}}}{\va{\det\p{I - p \circ D_x F_{\omega,\overrightarrow{\omega}}}}} \int_{\R} G_{\omega,\overrightarrow{\omega},w,T_{\omega,\overrightarrow{\omega}}\p{x}}\p{x,y} \mathrm{d}y \neq 0,
\end{split}
\end{equation*}
and by $Q$ the (finite) set of periodic orbits $\gamma$ for $\p{\phi^t}_{t \in \R}$ such that 
\begin{equation*}
E\p{\gamma} \coloneqq \frac{h\p{T_{\gamma}} \psi_{\ell}\p{T_{\gamma}}}{\va{\det\p{I - \mathcal{P}_{\gamma}}}} e^{\int_{\gamma} g} \int_{\gamma^{\#}} \theta_{\overrightarrow{\omega},w} \varphi_{\omega} \neq 0.
\end{equation*}
We will construct a bijection $x \mapsto \gamma\p{x}$ between $P$ and $Q$ such that, for all $x \in P$, we have $D\p{x} = E\p{\gamma\p{x}}$. This will immediately imply that
\begin{equation*}
\begin{split}
\textup{tr}\p{\int_\R \psi_{\ell}\p{t} h\p{t} \mathcal{A}_{\omega,\omega,\overrightarrow{\omega},w,t} \mathrm{d}t} & = \sum_{\gamma} \frac{h\p{T_{\gamma}} \psi_{\ell}\p{T_{\gamma}}}{\va{\det\p{I - \mathcal{P}_{\gamma}}}} e^{\int_{\gamma} g} \int_{\gamma^{\#}} \theta_{\overrightarrow{\omega},w} \varphi_{\omega}
\end{split}
\end{equation*}
and we can then end the proof by summing over $\omega \in \Omega$ and $\p{\overrightarrow{\omega},w} \in \Omega^{m} \times \mathcal{W}$.

Let $x \in P$. Since $D\p{x} \neq 0$, there is $\tilde{y} \in \R$ such that $G_{\omega,\overrightarrow{\omega},w,T_{\omega,\overrightarrow{\omega}}\p{x}}\p{x,\tilde{y}}$ is non-zero. Set $z = \p{x,\tilde{y}}$, and notice that $z \in V_\omega$, so that $\kappa_\omega^{-1}\p{z}$ make sense. Moreover, since $G_{\omega,\overrightarrow{\omega},w,T_{\omega,\overrightarrow{\omega}}\p{x}}\p{z} \neq 0$, we must have $\phi^{jt_0}\p{z} \in U_{\omega_{j}} $ for $j \in \set{1,\dots,m}$, and $\phi^{T_{\omega,\overrightarrow{\omega}}\p{x}}\p{z} \in U_\omega$. In addition, since $\psi_{\ell}\p{T_{\omega,\overrightarrow{\omega}}\p{x}}\neq 0$, we know that $T_{\omega,\overrightarrow{\omega}}\p{x} \in t_3 + m t_0 + \left]-t_1,t_1\right[$, and thus Lemma~\ref{lmdecoupe} ensures that
\begin{equation*}
\begin{split}
\kappa_\omega \circ \phi^{T_{\omega,\overrightarrow{\omega}}\p{x}} & \circ \kappa_\omega^{-1} \p{z} = \mathcal{T}^{\omega_{m},\omega', t_3}_{t -t_3- m t_0} \circ \mathcal{T}^{\omega_{m - 1},\omega_{m}, t_0}_{0} \circ \dots \circ \mathcal{T}^{\omega_1,\omega_2,t_0}_{0} \circ \mathcal{T}^{\omega,\omega_1,t_0}_{0}\p{z} \\
   & = \mathcal{T}^{\omega,\overrightarrow{\omega}}_{T_{\omega,\overrightarrow{\omega}}\p{x}}\p{z} = F_{\omega,\overrightarrow{\omega}}\p{x} + T_{\omega,\overrightarrow{\omega}}\p{x} e_{d+1} + \tilde{y}_{d+1} e_{d+1} \\
   & = p \circ F_{\omega,\overrightarrow{\omega}}\p{x} - T_{\omega,\overrightarrow{\omega}}\p{x} e_{d+1} + T_{\omega,\overrightarrow{\omega}}\p{x} e_{d+1} + \tilde{y}_{d+1} e_{d+1} \\
   & = z.
\end{split}
\end{equation*}
Consequently, there is a periodic orbit $\gamma\p{x}$ of length $T_{\gamma\p{x}} = T_{\omega,\overrightarrow{\omega}}\p{x}$ for $\p{\phi^t}_{t \in \R}$ passing through the point $\kappa_\omega^{-1}\p{z}$. Notice that, while the point $\kappa_{\omega}^{-1}\p{z}$ depends on the choice of $\tilde{y}$, the orbit $\gamma\p{x}$ does not (another choice of $\tilde{y}$ would only change $\kappa_{\omega}^{-1}\p{z}$ into another point of the orbit $\gamma\p{x}$). The map $D_{\kappa_\omega^{-1}\p{z}} \phi^{T_{\gamma\p{x}}}$ is conjugated via $D_{\kappa_\omega^{-1}\p{z}} \kappa_\omega$ to $D_z \mathcal{T}^{\omega,\overrightarrow{\omega}}_{T_{\omega,\overrightarrow{\omega}}\p{x}}$. However, in a base adapted to the decomposition of the tangent space into the stable and unstable directions and the direction of the flow, the matrix of the map $D_{\kappa_\omega^{-1}\p{z}} \phi^{T_{\gamma\p{x}}}$ is
\begin{equation*}
\left[
\begin{array}{cc}
\mathcal{P}_{\gamma\p{x}} & 0 \\
0 & 1 \\
\end{array}
\right],
\end{equation*}
while the matrix of $D_z \mathcal{T}^{\omega,\overrightarrow{\omega}}_{T_{\omega,\overrightarrow{\omega}}\p{x}}$ in the canonical basis of $\R^{d+1}$ is of the form
\begin{equation*}
\left[
\begin{array}{cc}
p \circ D_x F_{\omega,\overrightarrow{\omega}} & 0 \\
\ast & 1
\end{array}
\right].
\end{equation*}
Thus, the linear maps $\mathcal{P}_{\gamma\p{x}}$ and $p \circ D_x F_{\omega,\overrightarrow{\omega}}$ have the same spectrum, which implies that $\det\p{I - \mathcal{P}_{\gamma\p{x}}} = \det\p{I - D_x F_{\omega,\overrightarrow{\omega}}}$. Denote by $I_x$ the set of $y \in \R$ such that $G_{\omega,\overrightarrow{\omega},w,T_{\omega,\overrightarrow{\omega}}\p{x}}\p{x,y} \neq 0$. Then for all $y \in I_x$, we have 
\begin{equation*}
e^{\int_0^{T_{\omega,\overrightarrow{\omega}}\p{x}} g \circ \phi^\tau \p{\kappa_\omega^{-1}\p{x,y}} \mathrm{d}\tau} = \exp\p{\int_{\gamma\p{x}} g}.
\end{equation*}
Moreover, the map $I_x \ni y \to \kappa_\omega^{-1}\p{x,y} = \phi^{y - \tilde{y}} \p{\kappa_\omega^{-1}\p{z}}$ is injective (the length of $I_x$ is at most $2t_1$, and there is no periodic orbit of $\p{\phi^t}_{t \in \R}$ of length less than $3t_1$), and its image is $\gamma \cap \widetilde{U}_{\overrightarrow{\omega},w}$ (thanks to our assumption on the refinement), so that a change of variable gives
\begin{equation*}
\int_{\R} G_{\omega,\overrightarrow{\omega},w,T_{\omega,\overrightarrow{\omega}}\p{x}}\p{x,y} \mathrm{d}y = \exp\p{\int_{\gamma\p{x}} g} \int_{\gamma^{\#}\p{x}} \theta_{\overrightarrow{\omega},w} \varphi_\omega,
\end{equation*}
and thus we have $E\p{\gamma\p{x}} = D\p{x} \neq 0$, in particular $\gamma \in Q$. It remains to prove that $P \ni x \mapsto \gamma\p{x} \in Q$ is a bijection.

If $x \in P$ then the intersection of $\gamma\p{x}$ with $\widetilde{U}_{\overrightarrow{\omega},w}$ is an interval, and thus $\kappa_\omega\p{\gamma\p{x}\cap \widetilde{U}_{\overrightarrow{\omega},w}}$ is contained in a line perpendicular to $\R^d \times \set{0}$ (recall that $\kappa_\omega$ is a flow box) and this line projects on $x \in \R^d$. Thus $\gamma\p{x}$ determines $x$ and consequently the map $x \mapsto \gamma\p{x}$ is injective.

Reciprocally, if $\gamma \in Q$ then $\gamma$ must intersect $\widetilde{U}_{\overrightarrow{\omega},w}$ on a non-empty interval that is sent by $\kappa_\omega$ into a line perpendicular to $\R^{d} \times \set{0}$, that projects on a point $x \in \R^{d}$. Choose $y \in \R$ such that $\p{x,y}$ is the image by $\kappa_\omega$ of some point $\tilde{z} \in \gamma$ such that $\theta_{\overrightarrow{\omega},w}\p{\tilde{z}} \varphi_\omega\p{\tilde{z}} \neq 0$. Working as in the other case, we see that $\mathcal{T}^{\omega,\overrightarrow{\omega}}_{T_\gamma}\p{x,y} = \p{x,y}$, and thus $p \circ F_{\omega,\overrightarrow{\omega}}\p{x} = x$ and $T_\gamma = T_{\omega,\overrightarrow{\omega}}\p{x}$. The same calculation as above implies that $D\p{x} = E\p{\gamma} \neq 0$, so that $x \in P$. Finally, it is clear that $\gamma = \gamma\p{x}$ from the construction of $\gamma\p{x}$: these two periodic orbits pass through the point $\tilde{z}$. Thus, the map $x \mapsto\gamma\p{x}$ is surjective, and the proof is over.
\end{proof}

\begin{lm}\label{lmtracegs1}
Under the assumptions of Proposition \ref{proprecap}, the series
\begin{equation}\label{eqsumcv}
\sum_{\gamma} \frac{h\p{T_\gamma} T_\gamma^{\#}}{\va{\det\p{I - \mathcal{P}_{\gamma}}}} e^{\int_{\gamma} g}
\end{equation}
converges absolutely and
\begin{equation*}
\textup{tr}\p{\int_{0}^{+ \infty} h(t) \mathcal{L}_t \mathrm{d}t} = \sum_{\gamma} \frac{h\p{T_\gamma} T_\gamma^{\#}}{\va{\det\p{I - \mathcal{P}_{\gamma}}}} e^{\int_{\gamma} g}.
\end{equation*}
\end{lm}

\begin{proof}
First, use Lemma~\ref{lmtrper} and (the proof) of Lemma~\ref{lmsumtr}, with $g$ replaced by $\n{\Re\p{g}}_{\infty}$ $h$ replaced by $1+ \va{h}^2$, to get that (this can also be seen using an estimates on the number of periodic orbit for $\p{\phi_t}_{t \in \R}$ such as \cite[Lemma 2.2]{DZdet}):
\begin{equation*}
\sum_{\ell \in \Z} \sum_{\gamma} \psi_{\ell }\p{T_\gamma} \exp\p{T_\gamma \n{\Re\p{g}}_{\infty}} \frac{\va{h\p{T_\gamma}} T_\gamma^\# }{\va{\det\p{I - \mathcal{P}_\gamma}}} < + \infty.
\end{equation*}
We can then use the Fubini--Tonelli and monotone convergence theorems to get that
\begin{equation*}
\begin{split}
 & \sum_{\ell \in \Z} \sum_{\gamma} \psi_{\ell}\p{T_\gamma} \exp\p{T_\gamma \n{\Re\p{g}}_{\infty}} \frac{\va{h\p{T_\gamma}} T_\gamma^\#}{\va{\det\p{I - \mathcal{P}_\gamma}}} \\& \qquad \qquad = \sum_{\gamma} \sum_{\ell \in \Z} \psi_{\ell}\p{T_\gamma} \exp\p{T_\gamma \n{\Re\p{g}}_{\infty}} \frac{\va{h\p{T_\gamma}} T_\gamma^\#}{\va{\det\p{I - \mathcal{P}_\gamma}}} \\
 & \qquad \qquad = \sum_{\gamma} \exp\p{T_\gamma \n{\Re\p{g}}_{\infty}} \frac{\va{h\p{T_\gamma}} T_\gamma^\# }{va{\det\p{I - \mathcal{P}_\gamma}}} < + \infty. 
\end{split}
\end{equation*}
This proves that the sum \eqref{eqsumcv} converges absolutely and provides integrability and domination which allow us to apply Fubini's theorem and the dominated convergence theorem to get
\begin{equation*}
\begin{split}
\textup{tr}\Bigg( \int_{0}^{+ \infty} h(t) & \mathcal{L}_t \mathrm{d}t \Bigg) = \sum_{\ell \in \Z} \textup{tr}\p{\int_{0}^{+ \infty} \psi_{\ell}\p{t} h(t) \mathcal{L}_t \mathrm{d}t} \\
 & = \sum_{\ell \in \Z} \sum_{\gamma} \psi_{\ell}\p{T_\gamma} \exp\p{\int_\gamma g} \frac{h\p{T_\gamma} T_\gamma^\#}{\va{\det\p{I - \mathcal{P}_\gamma}}} \\
 & = \sum_{\gamma}\sum_{\ell \in \Z} \psi_{\ell}\p{T_\gamma} \exp\p{\int_\gamma g} \frac{h\p{T_\gamma} T_\gamma^\# }{\va{\det\p{I - \mathcal{P}_\gamma}}} \\
 & = \sum_{\gamma} \exp\p{\int_\gamma g} \frac{h\p{T_\gamma} T_\gamma^\#}{\va{\det\p{I - \mathcal{P}_\gamma}}}.
 \end{split}
\end{equation*}
\end{proof}

We end this section with the proof of two merely technical lemmas that will be useful in the following section to construct and study the anisotropic Hilbert spaces from Theorem~\ref{thmmain}.

\begin{lm}\label{lmmesurab}
For all $u \in \mathcal{D}^{\tilde{\upsilon}}\p{M}'$, the map $\R \ni t \mapsto \n{\mathcal{L}_t u}_{\widetilde{\h}} $ is measurable (with the convention that $\n{u}_{\widetilde{\h}} = \infty$ if $u \notin \widetilde{\h}$).
\end{lm}

\begin{proof}
Let us prove first that the map $ \mathcal{D}^{\tilde{\upsilon}}\p{M}' \ni u \mapsto \n{u}_{\widetilde{\h}}$ is measurable. Since the inclusion of $\widetilde{\h}$ in $\mathcal{D}^{\tilde{\upsilon}}\p{M}'$ is continuous (hence measurable) and $\n{\cdot}_{\widetilde{\h}}$ is continuous on $\widetilde{\h}$, we only need to check that $\widetilde{\h}$ is a measurable subset of $\mathcal{D}^{\tilde{\upsilon}}\p{M}'$. Keeping track of the different steps in the definition of $\widetilde{\h}$, we see that it is enough to prove that $L^2_{\textup{loc}}$ is a measurable subset of $\p{\mathcal{S}^{\tilde{\upsilon}}}'$, which is clear with the following characterization of $L^2_{\textup{loc}}$:
\begin{equation*}
\begin{split}
& L^2_{\textup{loc}} = \Bigg\{ u \in \p{\mathcal{S}^{\tilde{\upsilon}}}' : \forall \textup{ compact } K \subseteq \R^{d+1}, \exists C > 0 , \\ & \qquad \qquad \qquad \qquad \qquad \qquad \qquad \forall \varphi \in \mathcal{S}^{\tilde{\upsilon}} \textup{ supported in } K, \va{\langle u ,\varphi \rangle} \leq C \n{\varphi}_2 \Bigg\}.
\end{split}
\end{equation*}

Finally, recall that, if $u \in \mathcal{D}^{\tilde{\upsilon}}\p{M}'$, the map $\R \ni t \mapsto \mathcal{L}_t u \in \mathcal{D}^{\tilde{\upsilon}}\p{M}'$ is measurable (and even $\mathcal{C}^\infty$) according to Lemma~\ref{lmderiv} to end the proof.
\end{proof}

\begin{lm}\label{lmunifinc}
There is a continuous semi-norm $N$ on $\mathcal{C}^{\infty,\tilde{\upsilon}}\p{M}$ such that for all $u \in \mathcal{C}^{\infty,\tilde{\upsilon}}\p{M}$ and $t \in \left[-t_0,t_0\right]$ we have
\begin{equation*}
\n{\mathcal{L}_t u}_{\widetilde{\h}} \leq N\p{u}.
\end{equation*}
The same is true replacing $\mathcal{L}_t$ by $\p{\mathcal{L}_{-t}}^*$ and $\mathcal{C}^{\infty,\tilde{\upsilon}}\p{M}$ by $\mathcal{D}^{\tilde{\upsilon}}\p{M}$.
\end{lm}

\begin{proof}
Since the inclusion of $\mathcal{C}^{\infty,\tilde{\upsilon}}\p{M}$ in $\widetilde{\h}$ is continuous and $\p{\mathcal{L}_t}_{t \in \R}$ is a group, we only need to prove that there is $\epsilon >0$ such that for every continuous semi-norm $N_1$ on $\mathcal{C}^{\infty,\tilde{\upsilon}}\p{M}$ there is a continuous semi-norm $N_2$ on $\mathcal{C}^{\infty,\tilde{\upsilon}}\p{M}$ such that for all $u \in \mathcal{C}^{\infty,\tilde{\upsilon}}\p{M}$, and $t \in \left[-\epsilon,\epsilon\right]$ we have
\begin{equation}\label{eqach}
N_1\p{\mathcal{L}_t u} \leq N_2\p{u}.
\end{equation}
In fact, we only need to achieve \eqref{eqach} for $N_1$ of the form $\n{\cdot}_{\kappa_\omega,\varphi_\omega,\kappa,\tilde{\nu}}$ for $\omega \in \Omega$ and $\kappa \in \R_{+}^*$ (because these semi-norms generate the topology of $\mathcal{C}^{\infty,\tilde{\upsilon}}\p{M}$). But then, it becomes clear that \eqref{eqach} can be achieved, since the $\kappa_\omega$ are flow boxes. The proof for the adjoint is similar.
\end{proof}

\section{Global space: second step}\label{gs2}

Given the spaces $\widetilde{\h}$ and $\widetilde{\h}_0$ and Proposition~\ref{proprecap} from the previous section, the proofs of Theorem~\ref{thmmain}, Proposition~\ref{propmain}, and Proposition~\ref{propmqr} are now reduced to functional analysis, and we deal with these proofs in this last section. These proofs are split into several lemmas as follow: as far as Theorem~\ref{thmmain}  and Proposition \ref{propmqr} are concerned, (i) is contained in Lemma~\ref{lmwb}, (ii) is in Lemma~\ref{lmbounded}, (iii) is a consequence of Lemma~\ref{lmbounded} and Lemma~\ref{lmdomain}, (iv) is in Lemma \ref{lmdiscretespectrum}, and (v) is in Lemma~\ref{lmlaplace} (with $2t_0$ instead of $t_0$). We end the section with the proof of Proposition~\ref{propmain}. First of all, we define the space $\h$.

\begin{df}
Thanks to Lemma~\ref{lmmesurab}, we may define for all $u \in \mathcal{D}^{\tilde{\upsilon}}\p{M}'$,
\begin{equation}\label{eqdefnh}
\n{u}_{\h}^2 = \int_{0}^{t_0} \n{\mathcal{L}_t u}_{\widetilde{\h}}^2 \mathrm{d}t,
\end{equation}
and then define the space
\begin{equation*}
\widehat{\h} = \set{u \in \mathcal{D}^{\tilde{\upsilon}}\p{M}' : \n{u}_{\h}^2 < \infty}
\end{equation*}
endowed with the norm $\n{\cdot}_{\h}$. Let $\h$ be the closure of $\mathcal{C}^{\infty,\tilde{\upsilon}}\p{M}$ in $\widehat{\h}$ (for some $\tilde{\upsilon} \in \left] \upsilon, \frac{1}{1 - \alpha}\right[$, where $\alpha$ has been defined in \S \ref{gs1}, we recall in particular that if $\upsilon < 2$ then $\alpha < \frac{1}{2}$).
\end{df}

\begin{lm}\label{lmwb}
$\h$ and $\widehat{\h}$ are separable Hilbert spaces. The inclusion of $\h$ and $\widehat{\h}$ in $\mathcal{D}^{\tilde{\upsilon}}\p{M}'$ are continuous, and $\mathcal{C}^{\infty,\tilde{\upsilon}}$ is contained in $\h$ and $\widehat{\h}$, and the inclusion is continuous.
\end{lm}

\begin{proof}
We only need to prove the lemma for $\widehat{\h}$ (the statements for $\h$ immediately follows). Notice that the map
\begin{equation}\label{eqisom}
\begin{array}{ccc}
\widehat{\h} & \to & L^2\p{\left[0,t_0\right],\widetilde{\h}} \\
u & \mapsto & \p{\mathcal{L}_t u}_{0 \leq t \leq t_0}
\end{array}
\end{equation}
is an isometry. To show that $\widehat{\h}$ is a separable Hilbert space, we only need to prove that the image of the map \eqref{eqisom} is closed . Let $\p{u_n}_{n \in \N}$ be a sequence in $\widehat{\h}$ such that the sequence $\p{\p{\mathcal{L}_t u_n}_{0 \leq t \leq t_0}}_{n \in \N}$ converges to $\p{v\p{t}}_{0 \leq t \leq t_0}$ in $L^2\p{\left[0,t_0\right],\widetilde{\h}}$. Then there is a subset $A$ of $\N$ and a Borel subset $B$ of full measure in $\left[0,t_0\right]$ such that, for all $t \in B$, the sequence $\p{\mathcal{L}_t u_n}_{n \in A}$ converges to $v\p{t}$ in $\widetilde{\h}$ (in particular, it converges in $\mathcal{D}^{\tilde{\upsilon}}\p{M}'$). Choose $t' \in B$ and set $u = \mathcal{L}_{-t'} v\p{t'} \in \mathcal{D}^{\tilde{\upsilon}}\p{M}'$. Then, for all $t \in B$ and $n \in A$, we have
\begin{equation*}
\mathcal{L}_t u_n = \mathcal{L}_{t-t'} \p{\mathcal{L}_{t'} u_n}.
\end{equation*}
Letting $n$ tend to infinity, we have
\begin{equation*}
v\p{t} = \mathcal{L}_{t-t'} v\p{t'} = \mathcal{L}_t \p{\mathcal{L}_{-t'} v\p{t'}} = \mathcal{L}_t u.
\end{equation*}
Since $v \in L^2\p{\left[0,t_0\right],\widetilde{\h}}$, this implies that $u \in \widehat{\h}$, and thus the image of $\widehat{\h}$ under the map \eqref{eqisom} is closed, so that $\widehat{\h}$ is a Hilbert space.

To prove that the inclusion of $\widehat{\h}$ in $\mathcal{D}^{\tilde{\upsilon}}\p{M}'$ is continuous, just notice that if $\varphi \in \mathcal{C}^{\infty,\tilde{\upsilon}}\p{M}$ then
\begin{equation*}
\langle u, \phi \rangle = \frac{1}{t_0} \int_0^{t_0} \langle \mathcal{L}_t u, \p{\mathcal{L}_{-t}}^* \varphi \rangle \mathrm{d}t,
\end{equation*}
and use Lemma~\ref{lmunifinc}. The continuous inclusion of $\mathcal{C}^{\infty,\tilde{\upsilon}}\p{M}$ in $\widehat{\h}$ is an immediate consequence of Lemma~\ref{lmunifinc}.
\end{proof}

We now prove that $\h$ has the property that $\widetilde{\h}_0$ missed: the operator $\mathcal{L}_t$ for $t \geq 0$ is bounded from $\h$ to itself.

\begin{lm}\label{lmbounded}
For all $t \geq 0$, the operator $\mathcal{L}_t$ is bounded from $\mathcal{H}$ to itself. Moreover, $\p{\mathcal{L}_t}_{t \geq 0}$ is a strongly continuous semi-group of operators on $\h$.
\end{lm}

\begin{proof}
If $u \in \widetilde{\h}$ and $t \geq t_0$ then we have
\begin{equation*}
\begin{split}
\n{\mathcal{L}_t u}_{\h}^2 & = \int_0^{t_0} \n{\mathcal{L}_\tau \mathcal{L}_t u}_{\widetilde{\h}}^2 \mathrm{d}\tau = \int_0^{t_0} \n{\mathcal{L}_t \mathcal{L}_\tau u}_{\widetilde{\h}}^2 \mathrm{d}\tau \leq \n{\mathcal{L}_t}_{\widetilde{\h} \to \widetilde{\h}}^2 \n{u}_{\h}^2.
\end{split}
\end{equation*}
If $0 \leq t \leq t_0$ then we have
\begin{equation*}
\begin{split}
\n{\mathcal{L}_t u}_{\h}^2 & = \int_{0}^{t} \n{\mathcal{L}_{t_0} \mathcal{L}_\tau u}^2_{\widetilde{\h}} \mathrm{d}\tau + \int_{t}^{t_0} \n{\mathcal{L}_\tau u }_{\widetilde{\h}}^2 \mathrm{d}\tau \\
   & \leq \p{ 1 + \n{\mathcal{L}_{t_0}}_{\widetilde{\h} \to \widetilde{\h}}^2} \n{u}_{\h}^2.
\end{split}
\end{equation*}
Thus $\mathcal{L}_t$ is bounded from $\widetilde{\h}$ to itself, but since $\mathcal{L}_t$ sends $\mathcal{C}^{\infty,\tilde{\upsilon}}\p{M}$ into $\mathcal{C}^{\infty,\tilde{\upsilon}}\p{M}$ (and thus into $\h$), the operator $\mathcal{L}_t$ induces a bounded operator $\mathcal{L}_t : \h \to \h$. Since $\p{\mathcal{L}_t}_{t \geq 0}$ is locally uniformly bounded and $\p{\mathcal{L}_t u}_{t \geq 0}$ depends continuously on $t$ as an element of $\h$ when $u \in \mathcal{C}^{\infty,\tilde{\upsilon}}\p{M}$ (see Lemma~\ref{lmderiv}), the semi-group $\p{\mathcal{L}_t}_{t \geq 0}$ is strongly continuous.
\end{proof}

Notice that, according to Lemma~\ref{lmdomain}, the generator of the semi-group $\p{\mathcal{L}_t}_{t \geq 0}$ is $X$. We prove now a lemma that allows us to go from $\h$ to $\widetilde{\h}_0$ and back, in order to prove that the properties that we stated for $\widetilde{\h}_0$ in Proposition~\ref{proprecap} may be extended to $\h$.

\begin{lm}\label{lmpassage}
For all $t \geq t_0$, the operator $\mathcal{L}_t$ is bounded from $\widetilde{\h}$ to $\h$. If $z \in \C$ is such that $\Re\p{z} \gg 1$ then $\p{z- X}^{-1}$ is bounded from $\h$ to $\widetilde{\h}$.
\end{lm}

\begin{proof}
Let $u \in \widetilde{\h}$ then 
\begin{equation*}
\n{\mathcal{L}_{t} u}_{\h}^2 \leq \sup_{\tau \in \left[t,t+t_0\right]} \n{\mathcal{L}_{\tau}}_{\widetilde{\h} \to \widetilde{\h}}^2 \n{u}_{\widetilde{\h}}^2.
\end{equation*}
Thus $\mathcal{L}_t$ is bounded from $\widetilde{\h}$ to $\widehat{\h}$. Since it sends $\mathcal{C}^{\infty,\tilde{\upsilon}}\p{M}$ into itself, $\mathcal{L}_t$ sends $\widetilde{\h}_0$ into $\h$.

Now, recall \cite[Problem 1.15 p.487]{Kato} that if $\Re\p{z} \gg 1$ and $u \in \h$ then
\begin{equation*}
\p{z-X}^{-1} u = \int_{0}^{+ \infty} e^{-zt} \mathcal{L}_t u \mathrm{d}t.
\end{equation*}
But recall that the norm of $u$ in $\h$ is the norm of $\p{\mathcal{L}_t u}_{0 \leq t \leq t_0}$ in the space $L^2\p{\left[0,t_0\right],\widetilde{\h}}$. Thus, for all $n \in \N$, the norm of $\p{\mathcal{L}_t u}_{n t_0 \leq t \leq \p{n+1} t_0}$ in the space $L^2\p{\left[0,t_0\right],\widetilde{\h}}$ is smaller than $\n{\mathcal{L}_{t_0}}^n_{\h \to \h} \n{u}_{\h}$. Thus if $\Re\p{z} > \ln\p{\n{\mathcal{L}_{t_0}}_{\h \to \h}}$, then, by Cauchy--Schwarz inequality, there is a constant $C >0$ such that the $L^1$ norm of $t \mapsto e^{-zt}\mathcal{L}_t u$ is smaller than $C \n{u}_{\h \to \h}$. Thus $\p{z-X}^{-1}$ is bounded from $\h$ to $\widetilde{\h}$.
\end{proof}

We are now ready to prove that the spectrum of $X$ acting on $\h$ is discrete.

\begin{lm}\label{lmdiscretespectrum}
The spectrum of $X$ acting on $\h$ is made of isolated eigenvalues of finite multiplicity which coincide with the Ruelle resonances of $X$ (multiplicity taken into account).
\end{lm}

\begin{proof}
According to Lemma \ref{lmresgen}, it is enough to prove that the spectrum of $X$ acting on $\h$ is made of isolated eigenvalues of finite multiplicity. Let $z \in \C$ be such that $\Re z \gg 1$. Since $X$ is the generator of a strongly continuous semi-group, $z$ belongs to the resolvent set of $X$. From \cite[Problem 6.16 p.177]{Kato}, we see that we only need to prove that the essential spectral radius of $(z-X)^{-1}$ is zero (see Definition \ref{dfesr}).

To do so, let $\chi : \R_+^* \to \left[0,1\right]$ be a compactly supported $\mathcal{C}^\infty$ function such that $\chi(t) = 1$ if $t \leq 2 t_0$. Then, according to \cite[Problem 1.15 p.487]{Kato}, for all $n \geq 1$ we have (provided that $\Re z$ is large enough):
\begin{equation}\label{eqdecomporesolvent}
\begin{split}
(z - X)^{-n} & = \frac{1}{(n-1)!}\int_{0}^{+ \infty} e^{-zt} t^{n-1} \mathcal{L}_t \mathrm{d}t \\
             & = \frac{1}{(n-1)!}\int_{0}^{+ \infty} \chi(t) e^{-zt} t^{n-1} \mathcal{L}_t \mathrm{d}t + \frac{1}{(n-1)!}\int_{0}^{+ \infty} h_n(t) \mathcal{L}_t \mathrm{d}t,
\end{split}
\end{equation} 
where the function $h_n : \R_+^* \to \C$ is defined by $h_n(t) = \p{1 - \chi(t)} e^{-zt} t^{n-1}$. Set also $\tilde{h}_n(t) = z h_n(t+t_0) + h_n'(t+t_0)$, so that for all $t \in \R_+^*$ we have
\begin{equation*}
\begin{split}
h_n(t+t_0) = e^{-zt} \int_0^t e^{z \tau} \tilde{h}_n(\tau) \mathrm{d}\tau.
\end{split}
\end{equation*}
Then, notice that
\begin{equation}\label{eqfactorop}
\begin{split}
\mathcal{L}_{t_0} \circ \int_{0}^{+ \infty} \tilde{h}_n(\tau) \mathcal{L}_\tau \mathrm{d}\tau \circ (z-X)^{-1} & = \int_{0}^{+ \infty} \int_{0}^{+ \infty} e^{-zt} \tilde{h}_n(\tau) \mathcal{L}_{t_0 + t + \tau} \mathrm{d}t \mathrm{d}\tau \\
         & = \int_{0}^{+ \infty}  \p{\int_{\tau}^{+ \infty} e^{-zu} e^{z \tau} \tilde{h}_n(\tau) \mathcal{L}_{t_0+u} \mathrm{d}u} \mathrm{d}\tau \\
         & = \int_{0}^{+ \infty} e^{-zu} \p{\int_0^u e^{z \tau} \tilde{h}_n(\tau) \mathrm{d}\tau} \mathcal{L}_{t_0 + u} \mathrm{d}u\\
         & = \int_0^{+ \infty} h_n(t_0 + u) \mathcal{L}_{t_0+u} \mathrm{d}u \\
         & = \int_0^{+ \infty} h_n(t) \mathcal{L}_t \mathrm{d}t.
\end{split}
\end{equation}
Moreover, if $\Re z$ is large enough, then, for every $n \geq 1$, the function $\tilde{h}_n$ satisfies the assumptions of Proposition \ref{proprecap} and consequently the operator 
\begin{equation*}
\begin{split}
\int_0^{+ \infty} \tilde{h}_n(t) \mathrm{d}t : \widetilde{\h} \to \widetilde{\h}
\end{split}
\end{equation*}
is compact. It follows then from \eqref{eqfactorop} and Lemma \ref{lmpassage} that the operator
\begin{equation*}
\begin{split}
\frac{1}{(n-1)!}\int_{0}^{+ \infty} h_n(t) \mathcal{L}_t \mathrm{d}t : \h \to \h
\end{split}
\end{equation*}
is compact. On the other hand, we see that the operator norm of 
\begin{equation*}
\begin{split}
\frac{1}{(n-1)!}\int_{0}^{+ \infty} \chi(t) e^{-zt} t^{n-1} \mathcal{L}_t \mathrm{d}t : \h \to \h
\end{split}
\end{equation*}
is less than $ \frac{C (2 t_0)^{n}}{(n-1)!}$ for some constant $C > 0$. With \eqref{eqdecomporesolvent}, it follows then from Hennion's argument \cite{hennion} based on Nussbaum formula \cite{nussbaum} that the essential spectral radius of $(z-X)^{-1}$ on $\h$ is less than
\begin{equation*}
\begin{split}
\liminf_{n \to + \infty} \p{\frac{C (2 t_0)^n}{\p{n-1}!}}^{\frac{1}{n}} = 0.
\end{split}
\end{equation*}
\end{proof}

We can now give the proof of the most interesting property of the Hilbert space $\widetilde{\h}$.

\begin{lm}\label{lmlaplace}
Let $h$ be a $\mathcal{C}^\infty$ function supported on a compact subset of $\left[2 t_0,+\infty\right[$. Then the operator
\begin{equation}\label{eqoph}
\int_{0}^{+ \infty} h\p{t} \mathcal{L}_t \mathrm{d}t : \h \to \h
\end{equation}
is compact.  Its non-zero spectrum is the intersection of $\C \setminus \set{0}$ with the image of the spectrum of $X$ by $\lambda \mapsto \textup{Lap}(h)(-\lambda)$, where $\textup{Lap}$ denotes the Laplace transform. 

If $\upsilon < 2$, the operator \eqref{eqoph} is trace class and
\begin{equation*}
\textup{tr}\p{\int_{0}^{+ \infty} h\p{t} \mathcal{L}_t \mathrm{d}t} = \sum_{\gamma} \frac{h\p{T_\gamma} T_\gamma^{\#}}{\va{\det\p{I - \mathcal{P}_{\gamma}}}} \exp\p{\int_\gamma g}.
\end{equation*}
\end{lm}

\begin{proof}
As in the proof of Lemma \ref{lmdiscretespectrum}, define the function $\tilde{h}$ on $\R_+^*$ by
\begin{equation*}
\widetilde{h}\p{t} = z h\p{t + t_0} + h'\p{ t + t_0}.
\end{equation*}
Since $\tilde{h}$ is $\mathcal{C}^\infty$ and compactly supported in $\left[t_0,+ \infty\right[$, it satisfies the assumption of Proposition \ref{proprecap} and, working as in the proof of Lemma \ref{lmdiscretespectrum}, we see that the operator
\begin{equation*}
\begin{split}
\int_{0}^{+ \infty} h\p{t} \mathcal{L}_t \mathrm{d}t  = \mathcal{L}_{t_0} \circ \int_{t_0}^{+ \infty} \widetilde{h}\p{t} \mathcal{L}_t \mathrm{d}t \circ \p{z-X}^{-1} : \h \to \h
\end{split}
\end{equation*}
is compact.

In order to identify the on-zero spectrum of \eqref{eqoph}, we denote by $f$ the function defined by $f(z) = \textup{Lap}(h)(-z)$ and by $A$ the operator \eqref{eqoph}. If $\lambda \in \C$, denote by $E_\lambda$ the generalized eigenspace of $X$ associated to $\lambda$ and, if $\lambda \neq 0$, by $F_\lambda$ the generalized eigenspace of $A$ associated to $\lambda$. We want to prove that for all $\mu \in \C \setminus \set{0}$ we have
\begin{equation*}
\begin{split}
F_\mu = \bigoplus_{\substack{\lambda \in \C \\f(\lambda) = \mu}} E_\lambda,
\end{split}
\end{equation*}
which is a more precise statement that our claim on the eigenvalues of $A$. Let $\lambda \in \sigma\p{X}$ be such that $f(\lambda) \neq 0$. Since $X$ commutes with $\mathcal{L}_t$ for $t \geq 0$, it commutes with $A$ so that $E_\lambda$ is stable by $A$. We denote by $\widetilde{X}$ and $\widetilde{A}$ the endomorphisms of $E_\lambda$ induced respectively by $X$ and $A$. Since $E_\lambda$ is finite-dimensional (according to Lemma \ref{lmdiscretespectrum}), the operator $\widetilde{X}$ is bounded, and we may define for $t \geq 0$ the operator $e^{t\widetilde{X}}$ on $E_\lambda$. Then, $e^{t \widetilde{X}}$ is nothing else than the operator induced by $\mathcal{L}_t$ on $E_\lambda$ (they solve the same Cauchy problem). It follows that we have
\begin{equation}\label{eqidentificationeigenspaces}
\begin{split}
\widetilde{A} = \int_0^{+ \infty} h(t) e^{t \widetilde{X}} \mathrm{d}t = f\p{\widetilde{X}},
\end{split}
\end{equation}
where $f\p{\widetilde{X}}$ is meant in the sense of the holomorphic calculus of bounded operators (we may develop $e^{t \widetilde{X}}$ in power series). Since $\sigma\p{\widetilde{X}} = \set{\lambda}$ by definition of $E_\lambda$, it follows that $\sigma\p{\widetilde{A}} = \set{f(\lambda)}$, which gives
\begin{equation*}
\begin{split}
E_\lambda \subseteq F_{f(\lambda)}.
\end{split}
\end{equation*}
Reciprocally, let $\mu \in \sigma\p{A} \setminus \set{0}$. From the equality
\begin{equation*}
\begin{split}
X A = -\int_0^{+ \infty} h'(t) \mathcal{L}_t \mathrm{d}t,
\end{split}
\end{equation*}
we see that the range of $A$ is included in the domain of $X$. In particular, $F_\mu$ is contained in the domain of $X$ and thus $X$ induces a bounded operator on the finite dimensional space $F_\mu$. Applying as above the holomorphic calculus of bounded operators, we get that
\begin{equation*}
\begin{split}
F_\mu \subseteq \bigoplus_{\substack{\lambda \in \C \\f(\lambda) = \mu}} E_\lambda,
\end{split}
\end{equation*}
and \eqref{eqidentificationeigenspaces} is proven.

If $\upsilon < 2$, we may replace ``compact'' by ``trace class'' in the argument above. Then, using Lemma~\ref{lmesr} as in the proof of Lemma~\ref{lmresgen}, the operator
\begin{equation*}
\int_{0}^\infty h\p{t} \mathcal{L}_t \mathrm{d}t
\end{equation*}
has the same non-zero spectrum when acting on $\h$ or on $\widetilde{\h}_0$ and thus, by Lidskii's trace theorem, the same trace. This ends the proof with Proposition~\ref{proprecap}.
\end{proof}

\begin{rmq}
As pointed out after the statement of Theorem~\ref{thmmain}, the point (v) of Theorem~\ref{thmmain} proves trace formula \eqref{eqtrform} which was stated as an equality between distributions on $\R_+^*$. However, it is clear from the proof that the equality in fact holds in the dual of the space of compactly supported $\mathcal{C}^{d+2}$ functions on $\R_+^*$ whose $d+2$th derivative has bounded variation. In fact, using the same trick as in the proof of Proposition~\ref{propmain}, we see that trace formula holds in the dual of the space of compactly supported $\mathcal{C}^{d+1}$ functions on $\R_+^*$ whose $d+1$th derivative has bounded variations.  
\end{rmq}

Finally, we end this section with the proof of Proposition~\ref{propmain}.

\begin{proof}[Proof of Proposition~\ref{propmain}]
First of all, we need to prove that, when $\Re\p{z} \gg 1$, the essential spectral radius (see Definition \ref{lmesr} in Appendix \ref{apprri}) of the operator
\begin{equation}\label{eqruse}
\mathcal{L}_{t_0} \p{z-X}^{-1} = \int_{t_0}^{+ \infty} e^{-z\p{t-t_0}} \mathcal{L}_t \mathrm{d}t
\end{equation}
acting on $\h$ is zero. From the proof of Lemma \ref{lmdiscretespectrum}, we know that the essential spectral radius of $\p{z - X}^{-1}$ is zero. Then if $r > 0$ is such that $\p{z-X}^{-1}$ has no eigenvalue of modulus $r$ we may define the spectral projection
\begin{equation*}
\Pi_r  = \frac{1}{2 i \pi} \int_{\partial \mathbb{D}\p{0,r}} \p{w - \p{z-X}^{-1}}^{-1} \mathrm{d}w.
\end{equation*}
Then $I - \Pi_r$ has finite rank and the spectral radius of $\p{z-X}^{-1} \Pi_r$ is less than $r$. Since $\mathcal{L}_{t_0}$ commutes with $\p{z-X}^{-1}$, it also commutes with $\Pi_r$ and thus the spectral radius of $\mathcal{L}_{t_0} \p{z-X}^{-1} \Pi_r$ is less than $\n{\mathcal{L}_{t_0}} r$. Then writing
\begin{equation}\label{eqdecesr}
\mathcal{L}_{t_0} \p{z-X}^{-1} = \mathcal{L}_{t_0} \p{z-X}^{-1} \Pi_r + \mathcal{L}_{t_0} \p{z-X}^{-1} \p{1 - \Pi_r}
\end{equation}
and using Hennion's argument \cite{hennion} as in the proof of Lemma \ref{lmdiscretespectrum} (notice that the second term of the right-hand side of \eqref{eqdecesr} has finite rank), we see that the essential spectral radius of $\mathcal{L}_{t_0} \p{z-X}^{-1}$ is less than $\n{\mathcal{L}_{t_0}} r$. Since $r >0$ may be chosen arbitrarily small, the essential spectral radius of $\mathcal{L}_{t_0} \p{z-X}^{-1}$ is zero. Consequently, using functional calculus in finite dimension as in the proof of Lemma \ref{lmlaplace}, we may prove that the spectrum of $\mathcal{L}_{t_0} \p{z-X}^{-1}$ is made of the $\frac{e^{t_0 \lambda}}{z - \lambda}$ when $\lambda$ runs over the Ruelle spectrum of $X$.

On the other hand, according to Proposition~\ref{proprecap} (with $h = \mathds{1}_{\left[t_0,+\infty\right[}$ and $k=0$), the right-hand side of \eqref{eqruse} defines an operator on $\widetilde{\h}_0$ which is in the Schatten class $S_p$ for any $p > d+1$ (in particular it is compact and has essential spectral radius zero). We may use Lemma~\ref{lmesr} as in the proof of Lemma~\ref{lmresgen} to get that the spectrum of this operator is the same as the spectrum of the operator \eqref{eqruse} acting on $\h$, that we just described. Consequently, for all $p > d+1$, since the operator acting on $\widetilde{\h}_0$ is in the Schatten class $S_p$, its spectrum is in $\ell^p$ (see \cite[Corollary 3.4 p.54]{Gohb}), so that
\begin{equation*}
\sum_{\lambda \textup{ resonances of } X} \va{\frac{e^{\lambda t_0}}{z - \lambda}}^p < + \infty.
\end{equation*}
Since $t_0 >0$ and $p > d+1$ are arbitrary, Proposition~\ref{propmain} follows.
\end{proof}

\appendix

\section*{Appendix}

\section{Ruelle resonances are intrinsic}\label{apprri}

As pointed out before, the Banach spaces $\B$ that appear in Theorem~\ref{thmfonda} are highly non-canonical. To prove that Ruelle resonances do not depend on the choice of these spaces, there is a classical argument based on the investigation of a meromorphic continuation of the resolvent of $X$ as an operator from $\mathcal{C}^\infty\p{M}$ to $\mathcal{D}'\p{M}$. To deal with spaces that are not intermediary between $\mathcal{C}^\infty\p{M}$ and $\mathcal{D}'\p{M}$, it is easier to use an approach based on the following Lemma~\ref{lmesr}, whose proof may be found in \cite[Lemma A.3]{Bal2} or \cite[Lemma A.1]{Tsu}. Recall first the following definition.

\begin{df}[Isolated eigenvalue of finite multiplicity and essential spectral radius]\label{dfesr}
If $\B$ is a Banach space, $X$ an \emph{a priori} unbounded operator on $\B$ and $\lambda \in \C$, we say that $\lambda$ is an \emph{isolated eigenvalue of finite multiplicity} for $X$ if $\lambda$ is an isolated point of $\sigma\p{X}$ and the rank of the spectral projector
\begin{equation*}
\Pi_{\lambda} = \frac{1}{2 i \pi} \int_{\partial \mathbb{D}\p{\lambda,r}} \p{z-X}^{-1} \mathrm{d}z,
\end{equation*}
where $r$ is any small enough positive real number so that $\sigma\p{X} \cap \overline{\mathbb{D}}\p{\lambda,r} = \set{\lambda}$, is finite (this rank is by definition the algebraic multiplicity of $\lambda$).

Now, if $X$ is bounded we define the \emph{essential spectral radius} of $X$ as the infimum of $\rho > 0$ such that the intersection of $\sigma\p{X}$ with $\set{z \in \C : \va{z} > \rho}$ contains only isolated eigenvalues of finite multiplicity.
\end{df}

\begin{lm}\label{lmesr}
Let $\B$ be a Hausdorff topological vector space. Let $\B_1$ and $\B_2$ be Banach spaces continuously included in $\B$ such that $\B_1 \cap \B_2$ is dense both in $\B_1$ and in $\B_2$. Let $L : \B \to \B$ be a continuous linear map that preserves $\B_1$ and $\B_2$. Assume that the maps induced by $L$ on $\B_1$ and $\B_2$ are bounded operators whose essential spectral radius is smaller than some number $\rho >0$. Then the eigenvalues in $\set{z \in \C : \va{z} > \rho}$ coincide. Furthermore, the corresponding generalized eigenspaces coincide and are contained in $\B_1 \cap \B_2$.
\end{lm}

Applying Lemma~\ref{lmesr} to the resolvent of $X$, we may prove the two following lemmas. Lemma~\ref{lmdefres} asserts that Ruelle resonances are well-defined, while Lemma~\ref{lmresgen} ensures that the spectrum of $X$ acting on the space $\h$ given by Theorem~\ref{thmmain} coincides with the Ruelle spectrum (recall Definition \ref{defresonances}). The proofs of Lemmas \ref{lmdefres} and~\ref{lmresgen} are very similar and consequently we will only prove Lemma~\ref{lmresgen}, in order to show that there are no particular difficulties when working with unusual classes of regularity.

\begin{lm}\label{lmdefres}
Let $\B$ and $\widetilde{\B}$ be two Banach spaces and $A > 0$ be a positive real number. Assume that $\B$ and $\widetilde{\B}$ both satisfy the points (i)-(iv) from Theorem~\ref{thmfonda} for this particular value of $A$. Then the intersections of $\set{z \in \C : \Re\p{z} > - A}$ with the spectrum of $X$ acting on $\B$ and $\widetilde{\B}$ coincide.
\end{lm}

\begin{lm}\label{lmresgen}
Assume that $M$, the flow $\p{\phi^t}_{t \in \R}$, and $g$ are $\mathcal{C}^{\kappa,\upsilon}$ for some $\kappa > 0$ and $\upsilon >1$. Let $\B$ be a Banach space such that for some $\tilde{\upsilon} > \upsilon$ and $A \in \R$ we have: 
\begin{enumerate}[label=(\roman*)]
\item $\mathcal{C}^{\infty,\tilde{\upsilon}}\p{M} \subseteq \B \subseteq \mathcal{D}^{\tilde{\upsilon}}\p{M}'$, all the inclusions being continuous, the first one having dense image;
\item for all $t \in \R_+$, the operator $\mathcal{L}_t$ defined by \eqref{eqkoopman} is bounded from $\B$ to itself;
\item $\p{\mathcal{L}_t}_{t \geq 0}$ forms a strongly continuous semi-group of operator acting on $\B$, whose generator is $X$;
\item the intersection of the spectrum of $X$ acting on $\B$ with $\set{ z \in \C: \Re\p{z} > -A}$ is made of isolated eigenvalues of finite multiplicity.
\end{enumerate}
Then the intersection of $\set{ z \in \C: \Re\p{z} > -A}$ with the spectrum of $X$ acting on $\B$ is the intersection of $\set{ z \in \C: \Re\p{z} > -A}$ with the Ruelle spectrum of $X$ from Definition \ref{defresonances}.
\end{lm}

\begin{proof}
Apply Theorem~\ref{thmfonda} (with the same value of $A$) to get a Banach space $\widetilde{\B}$ such that in particular the intersection of $\set{ z \in \C: \Re\p{z} > -A}$ with the spectrum of $X$ acting on $\widetilde{\B}$ coincides with the intersection of $\set{ z \in \C: \Re\p{z} > -A}$ with the Ruelle spectrum of $X$ (by definition of the Ruelle spectrum). Now choose a positive real number $z_0$ large enough so that the resolvent $\p{z_0 - X}^{-1}$ is well-defined both on $\B$ and $\widetilde{\B}$. Notice that the resolvents of $X$ acting on $\B$ and $\widetilde{\B}$ coincide on the intersection $\B \cap \widetilde{\B}$. Indeed, from (iii) and \cite[Problem 1.15 p.487]{Kato}, it follows that, if $u \in \B \cap \widetilde{\B}$, then $\p{z_0 - X}^{-1}$ is defined as an element of $\mathcal{D}^{\tilde{\upsilon}}\p{M}'$ by
\begin{equation*}
\forall \mu \in \mathcal{D}^{\tilde{\upsilon}}\p{M} : \left\langle \p{z_0 - X}^{-1} u,\mu \right\rangle = \int_0^{+ \infty} e^{-z_0 t} \left\langle \mathcal{L}_t u,\mu\right\rangle.
\end{equation*} 
Thus we may extend $\p{z_0 - X}^{-1}$ to $\B + \widetilde{\B}$ by setting that $\p{z_0 - X}^{-1} u$ is equal to $ \p{z_0 - X}^{-1} v + \p{z_0 - X}^{-1} w$, if $u = v + w$ with $v \in \B$ and $w \in \widetilde{\B}$ (it does not depend on the choice of $v$ and $w$). Furthermore, this extension is continuous when $\B + \widetilde{\B}$ is endowed with the norm $\n{\cdot}_{\B + \widetilde{\B}}$ defined by
\begin{equation*}
\forall u \in \B + \widetilde{\B} : \n{u}_{\B + \widetilde{\B}} = \inf_{\substack{u = v + w \\ v \in \B, w \in \widetilde{\B}}} \n{v}_{\B} + \n{w}_{\widetilde{\B}}.
\end{equation*} 

Let $A' <A$ and $R >0$, provided that $z_0$ is large enough we have
\begin{equation}\label{eqoutside}
\frac{1}{\sqrt{\p{z_0 + A'}^2 + R^2}} \geq \frac{1}{z_0 + A}.
\end{equation}
The map $\lambda \mapsto \p{z_0 - \lambda}^{-1}$ induces a bijection between the spectrum of $X$ acting on $\B$ and the spectrum of $\p{z_0 - X}^{-1}$ acting on $\B$, but it also sends $ \set{z \in \C : \Re\p{z} \leq - A}$ into the disc of center $0$ and radius $\frac{1}{z_0 + A}$. Consequently, the essential spectral radius of $\p{z_0 - X}^{-1}$ acting on $\B$ is less than $\frac{1}{z_0+A}$. The same is true for the same reason replacing $\B$ by $\widetilde{\B}$. Thus we may apply Lemma~\ref{lmesr} (with $\rho = \frac{1}{z_0 + A}$ and $\B_1,\B_2$ and $\B$ being respectively $\B,\widetilde{\B}$ and $\B + \widetilde{\B}$) to see that the spectrum of $\p{z_0 -X}^{-1}$ outside of the disc of center $0$ and radius $\frac{1}{z_0 +A}$ is the same on $\B$ and on $\widetilde{\B}$. But the map $\lambda \mapsto \p{z_0 - \lambda}^{-1}$ sends $\set{z \in \C : -A' \leq \Re (z) \leq z_0 \textup{ and } \va{\Im\p{z}} \leq R}$ outside of the disc of center $0$ and radius $\frac{1}{z_0+A}$ (see \eqref{eqoutside}). Consequently, the intersection of $\set{z \in \C : -A' \leq \Re (z) \leq z_0 \textup{ and } \va{\Im\p{z}} \leq R}$ with the spectrum of $X$ acting on $\B$ coincides with the intersection of $\set{z \in \C : -A' \leq \Re (z) \leq z_0 \textup{ and } \va{\Im\p{z}} \leq R}$ with the set of Ruelle resonances of $X$. Since $R >0$ and $A' < A$ are arbitrary, and $z_0$ may be chosen arbitrarily large, the lemma is proven.
\end{proof}

\section{Proofs of Lemmas \ref{lmderiv} and~\ref{lmdomain}}\label{appDCC}

\begin{proof}[Proof of Lemma~\ref{lmderiv}]
We only need to prove the first point: the same argument with $\mathcal{C}^{\infty,\tilde{\upsilon}}\p{M}$ replaced by $\mathcal{D}^{\tilde{\upsilon}}\p{M}$, and $\mathcal{L}_t$ and $X$ replaced by their formal adjoints gives the second point.

We start with the case $g = 0$. Using the group property of $\p{\mathcal{L}_t}_{t \in \R}$, we only need to prove differentiability at $t=0$. Then we may cover $M$ by flow boxes, and thus we only need to show that if $u \in \mathcal{S}^{\tilde{\upsilon}}$ is supported in a compact subset $K$ of $\R^{d+1}$ then 
\begin{equation}\label{eqconv}
\frac{u\p{\cdot + t e_{d+1}} - u}{t} \underset{t \to 0}{\to} \partial_{x_{d+1}} u \textup{ in } \mathcal{S}^{\tilde{\upsilon}},
\end{equation}
where $e_{d+1}$ denotes the last vector of the canonical basis of $\R^{d+1}$. Up to enlarging $K$ we may assume that for all $t \in \left[-1,1\right]$ the function $u\p{\cdot + t e_{d+1}}$ is supported in $K$. Then if $x \in K$, $\alpha \in \N^{d+1}$, and $t \in \left[-1,1\right]$ we have with Taylor's formula (for any $\kappa'' >0$):
\begin{equation*}
\begin{split}
& \va{\frac{\partial^\alpha u\p{x+t e_{d+1}} - \partial^{\alpha}u\p{x}}{t} - \partial^{\alpha} \partial_{x_{d+1}} u\p{x} }  \\  & \quad \quad \quad = \va{\frac{\partial^\alpha u\p{x+t e_{d+1}} - \partial^{\alpha}u\p{x}}{t} - \partial_{x_{d+1}} \partial^\alpha u\p{x} }\\
     & \quad \quad \quad \leq  \frac{\n{\partial_{x_{d+1}}^2 \partial^\alpha u}_{\infty}}{2} \va{t} \leq \frac{\va{t}}{2} \n{u}_{\kappa'',\tilde{\upsilon}} \exp\p{\frac{\p{\va{\alpha} + 2}^{\tilde{\upsilon}}}{\kappa''}}.
\end{split}
\end{equation*}
Thus if $\kappa',\kappa'' > 0$ and for $R > 0$ depending only on $K$, we have for all $x \in \R^{d+1}, \alpha \in \N^{d+1}$ and $m \in \N$:
\begin{equation*}
\begin{split}
\p{1 + \va{x}}^m \va{\frac{\partial^\alpha u\p{x+t e_{d+1}} - \partial^{\alpha}u\p{x}}{t} - \partial^{\alpha} \partial_{x_{d+1}} u\p{x} } \exp\p{ - \frac{\p{m + \va{\alpha}}^{\tilde{\upsilon}}}{\kappa'}} \\ \leq \frac{\va{t}}{2} \n{u}_{\kappa'',\tilde{\upsilon}} R^m \exp\p{\frac{\p{\va{\alpha} + 2}^{\tilde{\upsilon}}}{\kappa''} - \frac{\p{m + \va{\alpha}}^{\tilde{\upsilon}}}{\kappa'}}.
\end{split}
\end{equation*}
Thus if $\kappa' > 0$ and $\kappa'' > \kappa' $, then there is a constant $C >0$ (that only depends on $K, \tilde{\upsilon},\kappa'$, and $\kappa''$) such that for all $t \in \left[-1,1\right]$ we have
\begin{equation*}
\n{\frac{u\p{\cdot + t e_{d+1}} - u}{t} - \partial_{x_{d+1}} u}_{\kappa',\tilde{\upsilon}} \leq C \va{t} \n{u}_{\kappa'',\tilde{\upsilon}},
\end{equation*}
which implies \eqref{eqconv} and thus ends the proof of the lemma in the case $g =0$.

In order to deduce the result in the case of a general $g$ from the case $g=0$, we only need to prove that the map
\begin{equation}\label{equnecourbe}
\begin{split}
t \mapsto \exp\p{ \int_0^t g \circ \phi^\tau \mathrm{d}\tau}
\end{split}
\end{equation}
is $\mathcal{C}^\infty$ from $\R$ to $\mathcal{C}^{\infty,\tilde{\upsilon}}\p{M}$. Indeed, the multiplication is continuous from the product $\mathcal{C}^{\infty,\tilde{\upsilon}}\p{M} \times \mathcal{C}^{\infty,\tilde{\upsilon}}\p{M}$ to $\mathcal{C}^{\infty,\tilde{\upsilon}}\p{M}$. The map \eqref{equnecourbe} is easily seen to be $\mathcal{C}^\infty$ from $\R$ to $\mathcal{C}^0\p{M}$, and one may notice that its derivatives are valued in $\mathcal{C}^{\infty,\tilde{\upsilon}}$ (recall that the classes of regularity $\mathcal{C}^{\kappa,\tilde{\upsilon}}$, and hence $\mathcal{C}^{\infty,\tilde{\upsilon}}$, are closed by composition) with uniform bounds locally in $t$. Then, by successive applications of Taylor's formula at order $1$ with integral remainder, one gets that the map \eqref{equnecourbe} is $\mathcal{C}^\infty$ from $\R$ to $\mathcal{C}^{\infty,\tilde{\upsilon}}\p{M}$, ending the proof of the lemma (we use the exact formula for the remainder in order to bound it in $\mathcal{C}^{\infty,\tilde{\upsilon}}\p{M}$).
\end{proof}

\begin{proof}[Proof of Lemma~\ref{lmdomain}]
Denote for now the generator of $\p{\mathcal{L}_t}_{t \geq 0}$ by $\widetilde{X}$. Let $u \in \B$ be in the domain of $\widetilde{X}$, then the map $\R_+ \ni t \mapsto \mathcal{L}_t u \in \B$ is differentiable at $0$ and its derivative at $0$ is $\widetilde{X}u$ (by definition of $\widetilde{X}$). Since $\B \subseteq \mathcal{D}^{\tilde{\upsilon}}\p{M}$ is continuous, the same is true for the map $\R_+ \ni t \mapsto \mathcal{L}_t u \in \mathcal{D}^{\tilde{\upsilon}}\p{M}'$, whose derivative at $0$ is $Xu$ according to Lemma~\ref{lmderiv}. Thus $\widetilde{X}u = Xu \in \B$.

Reciprocally, if $u \in \B$ is such that $Xu \in \B$, then we may define a $\mathcal{C}^1$ map $c: \R_+ \to \B$ by $c\p{t} = u + \int_0^t \mathcal{L}_\tau X u \mathrm{d}\tau$ for all $t \in \R_+$. Notice that $c'\p{0} = Xu$. Since $\B \subseteq \mathcal{D}^{\tilde{\upsilon}}\p{M}$ is continuous, the map $c$ is still $\mathcal{C}^1$ when seen as a map from $\R_+$ to $\mathcal{D}^{\tilde{\upsilon}}\p{M}$ and we have $c\p{0} = u$ and $c'\p{t} = \mathcal{L}_t X u$ for all $t \in \R_+$, so that $c\p{t} = \mathcal{L}_t u$ for all $t \in \R_+$, using Lemma~\ref{lmderiv}. This proves that $u$ belongs to the domain of $\widetilde{X}$.
\end{proof}

\section{Factorization of the dynamical determinant}\label{appdet}

We prove here, under the hypotheses of Theorem~\ref{thmmain}, a Hadamard-like factorization \eqref{eqfacto} for the dynamical determinant $d_g$ defined by \eqref{eqdefdet}. Let $t_0 > 0$ be shorter than any periodic orbit of $\p{\phi^t}_{t \in \R}$. Then, working as in the proof of Proposition \ref{propmain}, we see that, for $\Re z \gg 1$, the essential spectral radius of 
\begin{equation}\label{eqilestpresqueatrace}
\begin{split}
\mathcal{L}_{t_0}(z-X)^{-(d+2)} = \frac{1}{(d+1)!}\int_{t_0}^{+ \infty} e^{-z(t-t_0)}(t-t_0)^{d+1} \mathcal{L}_t \mathrm{d}t : \h \to \h
\end{split}
\end{equation}
is zero. Then, applying holomorphic functional calculus in finite dimension as in the proof of Lemma \ref{lmlaplace}, we see that the spectrum of \eqref{eqilestpresqueatrace} is made of the $\frac{e^{\lambda t_0}}{(z-\lambda)^{d+2}}$ for $\lambda$ in the spectrum of $X$. Then, for $\Re z \gg 1$, Proposition \ref{proprecap} implies that the right hand side of \eqref{eqilestpresqueatrace} defines a trace class operator on $\widetilde{\h}_0$. From Lemma \ref{lmesr}, we see that the spectrum of \eqref{eqilestpresqueatrace} is the same when acting on $\h$ or on $\tilde{h}_0$. Then, using Lidskii's Trace Theorem and Proposition \ref{proprecap}, we see that\footnote{Notice that the global trace formula \eqref{eqtrform} may be deduced from this equality using residue's formula.},
\begin{equation*}
\sum_{\lambda \textup{ resonance}} \frac{e^{\lambda t_0}}{\p{z - \lambda}^{d+2}} =\frac{1}{\p{d+1}!} \sum_{\gamma} T_\gamma^{\#} \exp\p{\int_\gamma g} \p{T_\gamma - t_0}^{d+1} \frac{e^{-z \p{T_\gamma - t_0}}}{\va{\det\p{I - \mathcal{P}_\gamma}}}.
\end{equation*}
For all $\lambda \in \C \setminus \set{0}$ notice that the meromorphic map
\begin{equation*}
z \mapsto - \sum_{n \geq d+1} \frac{z^n}{\lambda^{n+1}} e^{-\p{z-\lambda}t_0} = \frac{e^{-\p{z-\lambda} t_0}}{z- \lambda} + \sum_{n = 0}^{d} \frac{z^n}{\lambda^{n+1}}e^{-\p{z-\lambda}t_0}
\end{equation*}
has a unique pole in $\lambda$ whose order is $1$ and whose residue is $1$. Thus there is an entire function $G_{\lambda,t_0}$ such that for all $z \in \C$
\begin{equation*}
\frac{G_{\lambda,t_0}'\p{z}}{G_{\lambda,t_0}\p{z}} = - \sum_{n \geq d+1} \frac{z^n}{\lambda^{n+1}} e^{-\p{z-\lambda}t_0} = \frac{e^{-\p{z-\lambda} t_0}}{z- \lambda} + \sum_{n = 0}^{d} \frac{z^n}{\lambda^{n+1}}e^{-\p{z-\lambda}t_0}
\end{equation*}
and $G_{\lambda,t_0}\p{0} = 1$. Choose for $G_{0,t_0}$ any logarithmic primitive of $z \mapsto \frac{e^{-t_0 z}}{z}$.

Now, choose $R > 0$ and if $\va{\lambda} \geq 2 R$ notice that for all $z \in \mathbb{D}\p{0,R}$ we have
\begin{equation*}
\va{\frac{G_{\lambda,t_0}'\p{z}}{G_{\lambda,t_0}\p{z}}} \leq 2 e^{R t_0} R^{d+1} \frac{e^{\Re\p{\lambda}t_0}}{\va{\lambda}^{d+2}}
\end{equation*}
and using the fact that $G_{\lambda,t_0}$ has a logarithm on $\mathbb{D}\p{0,R}$ that vanishes in $0$ (since $G_{\lambda,t_0}$ vanishes only at $\lambda$) we get that, for some constance $C$ depending only on $R$ and all $z \in \mathbb{D}\p{0,R}$
\begin{equation*}
\va{1 - G_{\lambda,t_0}\p{z}} \leq C \frac{e^{\Re\p{\lambda}t_0}}{\va{\lambda}^{d+2}}.
\end{equation*}
Using Proposition~\ref{propmain}, this implies that the infinite product
\begin{equation*}
\widetilde{d}_g\p{z} = \prod_{\lambda \textup{ resonance}} G_{\lambda,t_0}\p{z}
\end{equation*}
converges uniformly on all compact subset of $\C$. Notice that the zeros of $\widetilde{d}_g$ are precisely the Ruelle resonances. Now, we find that
\begin{equation*}
\p{e^{zt_0} \p{\ln G_{\lambda,t_0}\p{z}}'}^{\p{d+1}} =\p{-1}^{d+1}\p{d+1}! \frac{e^{\lambda t_0}}{\p{z - \lambda}^{d+2}}
\end{equation*}
and thus, for $\Re z \gg 1$,
\begin{equation}\label{eqderlog}
\begin{split}
\p{e^{zt_0} \p{\ln \widetilde{d}_g\p{z}}'}^{\p{d+1}} & = \p{-1}^{d+1}\p{d+1}!\sum_{\lambda \textup{ resonance}} \frac{e^{\lambda t_0}}{\p{z - \lambda}^{d+2}} \\
     & = \p{-1}^{d+1}\sum_{\gamma} T_\gamma^{\#} \exp\p{\int_\gamma g} \p{T_\gamma- t_0}^{d+1} \frac{e^{-z \p{T_\gamma - t_0}}}{\va{\det\p{I - \mathcal{P}_\gamma}}}\\
     & = \p{ e^{z t_0} \sum_{\gamma} T_\gamma^{\#} \exp\p{\int_\gamma g} \frac{e^{- zT_\gamma}}{\va{\det\p{I - \mathcal{P}_\gamma}}}}^{\p{d+1}} \\
     & = \p{ e^{z t_0} \p{ \ln d_g\p{z}}'}^{\p{d+1}},
\end{split}
\end{equation}
where $d_g$ is the usual dynamical determinant defined by \eqref{eqdefdet}. 
From \eqref{eqderlog} we deduce that there are a polynomial $P$ of degree at most $d$ and $\mu \in \C$ such that, for all $z \in \C$, we have the Hadamard-like factorization
\begin{equation}\label{eqfacto}
d_g\p{z} = \mu \exp\p{P\p{z} e^{-t_0z}} \prod_{\lambda \textup{ resonance}} G_{\lambda,t_0}\p{z}.
\end{equation}

In order to make this factorization more explicit, let us describe the $G_{\lambda,t_0}$'s. For all $\lambda \in \C \setminus \set{0} $, define the polynomial
\begin{equation*}
Q_{\lambda,t_0} = - \sum_{k=0}^{d} \p{\sum_{n=k}^{d} \frac{k!}{n!} \frac{(t_0 - \lambda)^{n-k-1}}{\lambda^{k+1}} }X^k,
\end{equation*}
and notice that
\begin{equation*}
\p{Q_{\lambda,t_0}\p{z}e^{-z\p{t_0 - \lambda}}}' = \sum_{n = 0}^{d} \frac{z^n}{\lambda^{n+1}}e^{-\p{z-\lambda}t_0}.
\end{equation*}
Thus we have for all $\lambda \in \C  \setminus \set{0}$ and $z \in \C$
\begin{equation*}
\begin{split}
G_{\lambda,t_0}\p{z} = \p{1- \frac{z}{\lambda}} & \exp \p{Q_{\lambda,t_0}\p{z}e^{-\p{z - \lambda} t_0 } - Q_{\lambda,t_0}(0) e^{\lambda t_0}} \\ & \qquad \qquad \qquad \qquad \qquad \qquad \times \exp\p{z \int_0^1 \frac{e^{-(zu - \lambda)t_0}-1}{zu - \lambda}\mathrm{d}u}.
\end{split}
\end{equation*}
The last factor is a logarithmic primitive of $ z \mapsto \frac{e^{-\p{z-\lambda}t_0}-1}{z-\lambda}$.

\section{Applications of the trace formula}\label{apptrdet}

As applications of the trace formula, we prove here Proposition~\ref{propapp} and Corollary~\ref{corphrag}.

\begin{proof}[Proof of Proposition~\ref{propapp}]
It is folklore to prove the implication (i) $\Rightarrow$ (ii) from residue theorem, see also \cite[Theorem 17]{MunozMarco}. Let us prove the implication (i) $\Rightarrow$ (ii).

Choose $x >0$ large enough so that the series
\begin{equation}\label{eqdefx}
\sum_\gamma T_\gamma^{\#} \frac{e^{-x T_\gamma}}{\va{\det\p{I - \mathcal{P}_\gamma}}} \exp\p{\int_\gamma g}
\end{equation}
converges absolutely and $x > \Re\p{\lambda}+\epsilon$ for all the resonances $\lambda$ and some $\epsilon >0$. Write $k = \lceil \rho \rceil$. Choose $z \in \C$ such that $\Re\p{z} > x$. Then, we can find a sequence $\p{\varphi_n}_{n \in \N}$ of $\mathcal{C}^\infty$ functions, compactly supported in $\R_+^*$ such that
\begin{equation}\label{eqappro}
\lim_{n \to + \infty} \sup_{t \in \R} e^{tx}\va{\varphi_n\p{t} - t^{k}e^{-zt}} = 0
\end{equation}
and 
\begin{equation}\label{eqboun}
\sup_{\substack{n \in \N \\ t \in \R_+^*}} \va{e^{tx} \varphi^{\p{k}}\p{t}} < + \infty.
\end{equation}
Then, with \eqref{eqdefx} and \eqref{eqappro}, we have
\begin{equation*}
\sum_\gamma T_\gamma^{\#} \frac{\varphi_n\p{T_\gamma}}{\va{\det\p{I - \mathcal{P}_\gamma}}} \exp\p{\int_\gamma g} \underset{n \to + \infty}{\to} \sum_\gamma T_\gamma^{\#} \frac{e^{-z T_\gamma} T_{\gamma}^{k}}{\va{\det\p{I - \mathcal{P}_\gamma}}} \exp\p{\int_\gamma g}.
\end{equation*}
Now, since the trace formula holds (by assumption), we know that for all $n \in \N$ we have
\begin{equation*}
\sum_\gamma T_\gamma^{\#} \frac{\varphi_n\p{T_\gamma}}{\va{\det\p{I - \mathcal{P}_\gamma}}} \exp\p{\int_\gamma g} = \sum_{\lambda \textup{ resonances}} L\p{\varphi_n}\p{- \lambda}.
\end{equation*}
However, recall that
\begin{equation*}
L\p{\varphi_n}\p{-\lambda} = \int_0^\infty e^{\lambda t} \varphi_n\p{t}\mathrm{d}t 
\end{equation*}
so that, using \eqref{eqappro}, we have,
\begin{equation*}
L\p{\varphi_n}\p{-\lambda} \underset{n \to + \infty}{\to} \int_0^\infty t^{k} e^{-\p{z - \lambda} t}\mathrm{d}t = \frac{k!}{\p{z-\lambda}^{k+1}}.
\end{equation*}
Now, if $\lambda$ is non-zero, we have
\begin{equation*}
L\p{\varphi_n}\p{- \lambda} = \frac{\p{-1}^{k}}{\lambda^{k}} \int_0^{+ \infty} e^{\lambda t} \varphi_n^{\p{k}} \p{t} \mathrm{d}t.
\end{equation*}
Thus, \eqref{eqboun}, with $x > \Re\p{\lambda} +\epsilon$, and the second hypothesis provide a domination of $L\p{\varphi_n}\p{- \lambda} $, so that we have, using the dominated convergence theorem,
\begin{equation*}
\sum_{\lambda \textup{ resonances}} L\p{\varphi_n}\p{- \lambda} \underset{n \to + \infty}{\to} k! \sum_{\lambda \textup{ resonances}} \frac{1}{\p{z-\lambda}^{k+1}}.
\end{equation*}
Finally we have (when $\Re\p{z} \gg 1$)
\begin{equation*}
\begin{split}
k! \sum_{\lambda \textup{ resonances}} \frac{1}{\p{z-\lambda}^{k+1}} & = \sum_\gamma T_\gamma^{\#} \frac{e^{-z T_\gamma} T_{\gamma}^{k}}{\va{\det\p{I - \mathcal{P}_\gamma}}} \exp\p{\int_\gamma g} \\ & = \p{-1}^{k+1} \p{\ln d_g}^{\p{k+1}}(z).
\end{split}
\end{equation*}
Let $P$ denote the canonical product of genus $k-1$ whose zeros are the Ruelle resonances of $X$ (well-defined by \cite[(2.6.4)]{Boas} thanks to \eqref{eqsommeresonne}). Then we see that, if $z$ is not a Ruelle resonance for $X$, we have
\begin{equation}\label{eqderiveecanonical}
\begin{split}
\p{\ln P}^{(k)}(z) = \p{-1}^k \p{k-1}! \sum_{\lambda \textup{ resonances}} \frac{1}{(z-\lambda)^k}.
\end{split}
\end{equation}
It follows that $\p{\ln d_g}^{(k+1)} = \p{\ln P}^{(k+1)}$ and consequently there is a complex number $a$ such that for every $z \in \C$ that is not a Ruelle resonance for $X$ we have
\begin{equation*}
\begin{split}
\p{\ln P}^{(k)}(z) = \p{\ln d_g}^{(k)}(z) + a.
\end{split}
\end{equation*}
With \eqref{eqsommeresonne}, \eqref{eqderiveecanonical} and dominated convergence we see that $\p{\ln P}^{(k)}(r) \underset{\substack{r \to + \infty \\r \in \R}}{\to} 0$. By direct inspection, we see that $\p{\ln d_g}^{(k)}(r) \underset{\substack{r \to + \infty \\r \in \R}}{\to} 0$, and consequently $a = 0$. Thus, there is a polynomial $Q$ of degree at most $k-1 \leq \rho$ such that, for every $z \in \C$, we have
\begin{equation*}
\begin{split}
d_g(z) = e^{Q(z)} P(z),
\end{split}
\end{equation*}
and the result follows since $P$ has order less than $\rho$ by \cite[Theorem 2.6.5]{Boas}.

\end{proof}

\begin{proof}[Proof of Corollary~\ref{corphrag}]
Proposition~\ref{propapp} implies that $d_g$ has order less than $1$. But notice that $d_g$ is bounded on a line (choose a line parallel to the imaginary axis corresponding to a large positive real part) and thus has to be constant by the Phragm\'en--Lindel\"of Theorem \cite[Theorem 1.4.1]{Boas}. Finally, it has to be constant equal to $1$ since $d_g(z) \underset{\substack{z \to + \infty \\ z \in \R}}{\to} 1$.
\end{proof}

\section{Expanding maps of the circle and the condition $\upsilon < 2 $}\label{appupsilon}

In order to discuss the condition $\upsilon < 2$ in Theorem~\ref{thmmain}, we can consider a very simple example: expanding maps of the circle $\mathbb{S}^1 = \R / \Z$. An analogue of the space $\h$ from Theorem~\ref{thmmain} would then be an isotropic space of the type (here $\p{\hat{f}(n)}_{n \in \Z}$ denotes the sequence of Fourier coefficient of a function $f$)
\begin{equation*}
\h_{\alpha,\beta} = \set{f \in \mathcal{C}^\infty\p{\mathbb{S}^1,\C} : \sum_{n \in \Z} \va{\hat{f}\p{n}}^2 e^{2 \beta \ln\p{1 + \va{n}}^{\frac{1}{\alpha}}} < + \infty},
\end{equation*}
where $\beta > 0$ and $\alpha \in \left]\frac{\upsilon - 1}{\upsilon},1 \right[$ (this is the same condition as in Proposition~\ref{proplocop}), endowed with the norm
\begin{equation*}
\n{f}_{\alpha,\beta} = \sqrt{\sum_{n \in \Z} \va{\hat{f}\p{n}}^2 e^{2 \beta \ln\p{1 + \va{n}}^{\frac{1}{\alpha}}}}.
\end{equation*}
Then the transfer operator 
\begin{equation*}
\mathcal{L} : f \mapsto \frac{f\p{\frac{\cdot}{2}} + f\p{\frac{\cdot + 1}{2}}}{2}
\end{equation*}
associated to the doubling map may be written as 
\begin{equation*}
\mathcal{L} = \sum_{n \in \Z} \langle e_{2n}, \cdot \rangle_{L^2} e_n,
\end{equation*}
where $e_n : x \mapsto e^{2 i \pi n x}$ (the sum converges in strong operator topology on the space of continuous endomorphisms of $\h_{\alpha,\beta}$). Thus, the singular values of $\mathcal{L}$ acting on $\h_{\alpha,\beta}$ are the $ e^{\beta\p{\ln\p{1 + \va{n}}^{\frac{1}{\alpha}} - \ln\p{1 + 2\va{n}}^{\frac{1}{\alpha}}}} $ for $n \in \Z$. Using the fact that
\begin{equation*}
\ln\p{1 + \va{n}}^{\frac{1}{\alpha}} - \ln\p{1 + 2\va{n}}^{\frac{1}{\alpha}} \underset{\va{n} \to + \infty}{=} - \frac{\ln 2}{\alpha} \ln\p{1 + \va{n}}^{\frac{1}{\alpha} - 1} + O\p{\ln\p{1 + \va{n}}^{\frac{1}{\alpha} - 2}}
\end{equation*}
we see that $\mathcal{L}$ acting on $\h_{\alpha,\beta}$ is trace class when $\alpha < \frac{1}{2}$ and is not trace class when $\alpha > \frac{1}{2}$ (in the case $\alpha = \frac{1}{2}$ it depends on the value of $\beta$). Thus, we need to chose $\alpha < \frac{1}{2}$ if we want $\mathcal{L}$ to be nuclear. For general maps, this choice is possible only when $\upsilon < 2$ (see the condition in Proposition~\ref{proplocop}).

Consequently, using our method to prove the trace formula for $\mathcal{C}^{\kappa,\upsilon}$ Anosov flows (or hyperbolic diffeomorphisms as in \cite{lagtf}) would require to construct Hilbert spaces in a totally different way, if $\upsilon \geq 2$.

\section*{Acknowledgement}

I would like to thank Viviane Baladi for her careful reading of the different versions of this work. I would also like to thank Sébastien Gou\"ezel, Maciej Zworski, Semion Dyatlov and Shu Shen for discussions about the trace formula and useful suggestions. 

\bibliographystyle{plain}
\bibliography{biblio}

\end{document}